\documentclass[12pt,reqno,openbib,runningheads,a4paper]{amsart}%
\usepackage{eurosym}
\usepackage{amssymb}
\usepackage{amsfonts}
\usepackage{amsmath}
\usepackage{graphicx}
\usepackage[authoryear]{natbib}
\usepackage[dvipsnames]{xcolor}
\usepackage[legalpaper,bookmarks=true,colorlinks=true,linkcolor=blue,citecolor=blue]%
{hyperref}%
\providecommand{\U}[1]{\protect\rule{.1in}{.1in}}
\providecommand{\U}[1]{\protect\rule{.1in}{.1in}}
\newtheorem{theorem}{Theorem}[section]

\newtheorem{proposition}{Proposition}[section]

\newtheorem{lemma}{Lemma}[section]

\newtheorem{remark}{Remark}[section]

\makeatletter
\renewcommand{\@biblabel}[1]{}
\makeatother
\begin{document}

\begin{center}
{\Large \textbf{Tail empirical process and weighted extreme value index
estimator for randomly right-censored data}}\medskip\medskip

{\large Brahim Brahimi, Djamel Meraghni, Abdelhakim Necir}$^{\ast}${\large ,}
{\large Louiza Soltane}\medskip

{\small \textit{Laboratory of Applied Mathematics, Mohamed Khider University,
Biskra, Algeria}}\medskip\medskip%
\[
\]

\end{center}

\noindent\textbf{Abstract}\medskip

\noindent A tail empirical process for heavy-tailed and right-censored data is
introduced and its Gaussian approximation is established. In this context, a
(weighted) new Hill-type estimator{\Large \textbf{ }}for positive extreme
value index is proposed and its consistency and asymptotic normality are
proved by means of the aforementioned process in the framework of second-order
conditions of regular variation. In a comparative simulation study, the newly
defined estimator is seen to perform better than the already existing ones in
terms of both bias and mean squared error. As a real data example, we apply
our estimation procedure to evaluate the tail index of the survival time of
Australian male Aids patients. It is noteworthy that our approach may also
serve to develop other statistics related to the distribution tail such as
second-order parameter and reduced-bias tail index estimators. Furthermore,
the proposed tail empirical process provides a goodness-of-fit test for
Pareto-like models under censorship.\bigskip

\medskip


\vfill

\vfill

\noindent{\small $^{\text{*}}$Corresponding author:
\texttt{necirabdelhakim@yahoo.fr} \newline\noindent\textit{E-mail
addresses:}\newline\texttt{brah.brahim@gmail.com} (B.~Brahimi)\newline%
\texttt{djmeraghni@yahoo.com} (D.~Meraghni)\newline%
\texttt{louiza\_stat@yahoo.com} (L.~Soltane)}

\section{\textbf{Introduction\label{sec1}}}

\noindent Let $X_{1},...,X_{n}$ be $n\geq1$ independent copies of a
non-negative continuous random variable (rv) $X,$ defined over some
probability space $\left(  \Omega,\mathcal{A},\mathbb{P}\right)  ,$ with
cumulative distribution function (cdf) $F.\ $These rv's are censored to the
right by a sequence of independent copies $Y_{1},...,Y_{n}$ of a non-negative
continuous rv $Y,$ independent of $X$ and having a cdf $G.$ At each stage
$1\leq j\leq n,$ we only can observe the rv's $Z_{j}:=\min\left(  X_{j}%
,Y_{j}\right)  $ and $\delta_{j}:=\mathbf{1}\left\{  X_{j}\leq Y_{j}\right\}
,$ with $\mathbf{1}\left\{  \cdot\right\}  $ denoting the indicator function.
The latter rv indicates whether there has been censorship or not. If we denote
by $H$ the cdf of the observed $Z^{\prime}s,$ then, in virtue of the
independence of $X$ and $Y,$ we have $1-H=\left(  1-F\right)  \left(
1-G\right)  .$ Throughout the paper, we will use the notation $\overline
{\mathcal{S}}(x):=\mathcal{S}(\infty)-\mathcal{S}(x),$ for any $\mathcal{S}.$
Assume further that $F$ and $G$ are heavy-tailed or, in other words, that
$\overline{F}$ and $\overline{G}$ are regularly varying at infinity with
negative indices $-1/\gamma_{1}$ and $-1/\gamma_{2}$ respectively, notation:
$\overline{F}\in\mathcal{RV}_{\left(  -1/\gamma_{1}\right)  }$ and
$\overline{G}\in\mathcal{RV}_{\left(  -1/\gamma_{2}\right)  }.$ That is%
\begin{equation}
\frac{\overline{F}\left(  tx\right)  }{\overline{F}\left(  t\right)
}\rightarrow x^{-1/\gamma_{1}}\text{ and }\frac{\overline{G}\left(  tx\right)
}{\overline{G}\left(  t\right)  }\rightarrow x^{-1/\gamma_{2}},\label{rv}%
\end{equation}
as $t\rightarrow\infty,$ for any $x>0.$\ This class of
distributions\textbf{\ }includes models such as Pareto, Burr, Fr\'{e}chet,
$\alpha-$stable $\left(  0<\alpha<2\right)  ,$ t-Student and log-gamma, known
to be very appropriate for fitting large insurance claims, large fluctuations
of prices, financial log-returns, ... \citep[see,
e.g.,][]{Res06}. The regular variation of $\overline{F}$ and $\overline{G}$
implies that $\overline{H}\in\mathcal{RV}_{\left(  -1/\gamma\right)  },$ where
$\gamma:=\gamma_{1}\gamma_{2}/\left(  \gamma_{1}+\gamma_{2}\right)  .$ Since
weak approximations of extreme value theory based statistics are achieved in
the second-order framework \citep[see][]{deHS96}, then it seems quite natural
to suppose that cdf $H$ satisfies the well-known second-order condition of
regular variation: for any $x>0$%
\[
\dfrac{\overline{H}\left(  tx\right)  /\overline{H}\left(  t\right)
-x^{-1/\gamma}}{A\left(  t\right)  }\rightarrow x^{-1/\gamma}\dfrac
{x^{\nu/\gamma}-1}{\nu\gamma},
\]
as $t\rightarrow\infty,$ where $\nu\leq0$\ is the second-order parameter and
$A$ is a function tending to $0,$ not changing sign near infinity and having a
regularly varying absolute value at infinity with index $\nu/\gamma.$\ If
$\nu=0,$ interpret $\left(  x^{\nu/\gamma}-1\right)  /\left(  \nu
\gamma\right)  $ as $\log x.$ Let us denote this assumption by $\overline
{H}\in2\mathcal{RV}_{\left(  -1/\gamma,\nu\right)  }\left(  A\right)  .$ For
the use of second-order conditions in exploring the estimators asymptotic
behaviors, see, for instance, Theorem 3.2.6 in \cite{deHF06}, page 74. In the
last decade, several authors showed an increasing interest in the issue of
estimating the extreme-value index (EVI) when the data are subject to random
censoring. In this context, \cite{BeGDFF07} proposed estimators for the EVI
and high quantiles and discussed their asymptotic properties, when the
observations are censored by a deterministic threshold, while \cite{EnFG08}
adapted various classical EVI estimators to the case where the threshold of
censorship is random, and proposed a unified method to establish their
asymptotic normality. Here, we remind the adjustment they made to Hill
estimator \citep[][]{Hill75} so as to estimate the tail index $\gamma_{1}$
under random censorship.\ Let $\left\{  \left(  Z_{i},\delta_{i}\right)
,\text{ }1\leq i\leq n\right\}  $ be a sample from the couple of rv's $\left(
Z,\delta\right)  $ and $Z_{1,n}\leq...\leq Z_{n,n}$\ represent the order
statistics pertaining to $\left(  Z_{1},...,Z_{n}\right)  .$ If we denote the
concomitant of the $i$th order statistic by $\delta_{\left[  i:n\right]  }$
(i.e. $\delta_{\left[  i:n\right]  }=\delta_{j}$ if $Z_{i,n}=Z_{j}),$\ then
the adapted Hill estimator of $\gamma_{1},$ defined by \cite{EnFG08}, is given
by the formula $\widehat{\gamma}_{1}^{\left(  EFG\right)  }:=\widehat{\gamma
}^{\left(  H\right)  }/\widehat{p},$ where%
\begin{equation}
\widehat{\gamma}^{\left(  H\right)  }:=\dfrac{1}{k}\sum\limits_{i=1}^{k}%
\log\frac{Z_{n-i+1:n}}{Z_{n-k:n}}\text{ and }\widehat{p}:=\dfrac{1}{k}%
%TCIMACRO{\dsum \limits_{i=1}^{k}}%
%BeginExpansion
{\displaystyle\sum\limits_{i=1}^{k}}
%EndExpansion
\delta_{\left[  n-i+1:n\right]  },\label{phate}%
\end{equation}
for a suitable sample fraction $k=k_{n},$ are respectively Hill's estimator
\citep[][]{Hill75} of $\gamma$ and an estimator of the proportion
$p:=\gamma/\gamma_{1}$ of the observed extreme values. By following similar
procedures, \cite{NDD2014} addressed the nonparametric estimation of the
conditional EVI and large quantiles for heavy-tailed distributions which was
generalized, a couple of years later, by \cite{Stupfler2016} to all three
extreme domains of attraction namely, F\'{e}chet, Gumbel and Weibull types of
distributions. In his working paper \cite{Stupfler2017} considered the
dependent random right-censoring scheme and develop an interesting new topic.
For their part, \cite{WW2014} used Kaplan-Meier integration and the synthetic
data approach of \cite{Leurg87}, to respectively introduce two Hill-type
estimators%
\[
\widehat{\gamma}_{1}^{\left(  WW1\right)  }:=\sum\limits_{i=1}^{k}%
\frac{1-F_{n}\left(  Z_{n-i+1:n}\right)  }{1-F_{n}\left(  Z_{n-k:n}\right)
}\log\frac{Z_{n-i+1:n}}{Z_{n-i:n}},
\]
and%
\[
\widehat{\gamma}_{1}^{\left(  WW2\right)  }:=\sum\limits_{i=1}^{k}%
\frac{1-F_{n}\left(  Z_{n-i+1:n}\right)  }{1-F_{n}\left(  Z_{n-k:n}\right)
}\frac{\delta_{\left[  n-i+1:n\right]  }}{i}\log\frac{Z_{n-i+1:n}}{Z_{n-k:n}},
\]
where
\begin{equation}
F_{n}\left(  x\right)  =1-%
%TCIMACRO{\dprod _{Z_{i:n}\leq x}}%
%BeginExpansion
{\displaystyle\prod_{Z_{i:n}\leq x}}
%EndExpansion
\left(  1-\frac{1}{n-i+1}\right)  ^{\delta_{\left[  i:n\right]  }%
},\label{kaplan}%
\end{equation}
is the famous Kaplan-Meier estimator of cdf $F$ \citep[][]{KM58}. In their
simulation study, the authors pointed out that, for weak censoring $\left(
\gamma_{1}<\gamma_{2}\right)  ,$ their estimators perform better than
$\widehat{\gamma}_{1}^{\left(  EFG\right)  },$ in terms of bias and mean
squared error. However, in the strong censoring case $\left(  \gamma
_{1}>\gamma_{2}\right)  ,$ they noticed that the results become unsatisfactory
for both $\widehat{\gamma}_{1}^{\left(  WW1\right)  }$ and $\widehat{\gamma
}_{1}^{\left(  WW2\right)  }.$ With additional assumptions on $\gamma_{1}$ and
$\gamma_{2},$ they only established the consistency of their estimators,
without any indication on the asymptotic normality. Though these conditions
are legitimate from a theoretical viewpoint, they may be considered as
constraints in case studies. Very recently, \cite{B18} used Worms's estimators
to derive new reduced-bias ones, constructed bootstrap confidence intervals
for $\gamma_{1}$ and applied their results to long-tailed car insurance
portfolio. Likewise, \cite{BBWG-16} considered maximum likelihood inference,
on the basis of the relative excesses $\left\{  Z_{n-i+1:n}/Z_{n-k:n}\right\}
_{i=1}^{k},$ to propose a new reduced-bias estimator for the tail index of
models belonging to Hall's class \citep[][]{Hall82}. The problem is that, even
if this family includes a great number of usual heavy-tailed distributions, it
represents a restriction to the larger class of regularly varying cdf's, in
particular those with null second-order parameter $\left(  \nu=0\right)  $
such as the log-gamma model. As we can see, the only EVI estimator that does
not impose any restrictive assumptions on the model is the one introduced by
\cite{EnFG08}. For this reason, we intend to construct a new weighted
estimator to the index $\gamma_{1}$ that enjoys the benefits of $\widehat
{\gamma}_{1}^{\left(  EFG\right)  }.$

\subsection{Constructing a new estimator for $\gamma_{1}$}

First, we introduce two very crucial sub-distribution functions $H^{\left(
i\right)  }\left(  z\right)  :=\mathbb{P}\left\{  Z_{1}\leq z,\delta
_{1}=i\right\}  ,$ $i=0,1,$ for $z>0,$ so that one has $H\left(  z\right)
=H^{\left(  0\right)  }\left(  z\right)  +H^{\left(  1\right)  }\left(
z\right)  .$ The empirical counterparts are, respectively, defined by%
\[
H_{n}^{\left(  0\right)  }\left(  z\right)  :=\frac{1}{n}\sum_{i=1}%
^{n}\mathbf{1}\left\{  Z_{i}\leq z\right\}  \left(  1-\delta_{i}\right)
,\text{ }H_{n}^{\left(  1\right)  }\left(  z\right)  :=\frac{1}{n}\sum
_{i=1}^{n}\mathbf{1}\left\{  Z_{i}\leq z\right\}  \delta_{i},
\]
and $H_{n}\left(  z\right)  :=\dfrac{1}{n}\sum_{i=1}^{n}\mathbf{1}\left\{
Z_{i}\leq z\right\}  =H_{n}^{\left(  0\right)  }\left(  z\right)
+H_{n}^{\left(  1\right)  }\left(  z\right)  .$ From Lemma 4.1 of
\cite{BMN-2015}, under the first-order conditions $\left(  \ref{rv}\right)  ,$
we have $\overline{H}^{\left(  1\right)  }\left(  t\right)  /\overline
{H}\left(  t\right)  \rightarrow p,$ as $t\rightarrow\infty,$ which implies
that $\overline{H}^{\left(  1\right)  }\in\mathcal{RV}_{\left(  -1/\gamma
\right)  }$ too. Then it is natural to also assume that $\overline{H}^{\left(
1\right)  }$ satisfies the second-order condition of regular variation, in the
sense that $\overline{H}^{\left(  1\right)  }\in2\mathcal{RV}_{\left(
-1/\gamma,\nu_{1}\right)  }\left(  A_{1}\right)  .$ From Theorem 1.2.2 in
\cite{deHF06}, the assumption $\overline{H}\in\mathcal{RV}_{\left(
-1/\gamma\right)  }$ implies that $\int_{t}^{\infty}x^{-1}\overline{H}\left(
x\right)  dx/\overline{H}\left(  t\right)  \rightarrow\gamma,$ as
$t\rightarrow\infty,$ which, by integration by parts, gives
\begin{equation}
\dfrac{1}{\overline{H}\left(  t\right)  }\int_{t}^{\infty}\log\left(
z/t\right)  dH\left(  z\right)  \rightarrow\gamma,\text{ as }t\rightarrow
\infty.\label{hill-theo}%
\end{equation}
In other words, we have
\[
I\left(  t\right)  :=\left(  \dfrac{\overline{H}\left(  t\right)  }%
{\overline{H}^{\left(  1\right)  }\left(  t\right)  }\right)  \dfrac
{1}{\overline{H}\left(  t\right)  }\int_{t}^{\infty}\log\left(  z/t\right)
dH\left(  z\right)  \rightarrow\gamma/p=\gamma_{1}.
\]
Taking $t=t_{n}=Z_{n-k:n}$ and replacing $\overline{H}^{\left(  1\right)  }$
and $\overline{H}$ by their respective empirical counterparts $\overline
{H}_{n}^{\left(  1\right)  }$ and $\overline{H}_{n}$ yield that $I\left(
t\right)  $ becomes, in terms of $n,$%
\begin{equation}
\left(  \frac{\overline{H}_{n}\left(  Z_{n-k:n}\right)  }{\overline{H}%
_{n}^{\left(  1\right)  }\left(  Z_{n-k:n}\right)  }\right)  \frac
{1}{\overline{H}_{n}\left(  Z_{n-k:n}\right)  }\int_{Z_{n-k:n}}^{\infty}%
\log\left(  z/Z_{n-k:n}\right)  dH_{n}\left(  z\right)  .\label{efg}%
\end{equation}
We have $\overline{H}_{n}\left(  Z_{n-k:n}\right)  =k/n$ and $\overline{H}%
_{n}^{\left(  1\right)  }\left(  Z_{n-k:n}\right)  =n^{-1}\sum_{i=1}^{k}%
\delta_{\left[  n-i+1:n\right]  },$ then it is readily checked that%
\[
\frac{\overline{H}_{n}^{\left(  1\right)  }\left(  Z_{n-k:n}\right)
}{\overline{H}_{n}\left(  Z_{n-k:n}\right)  }=\widehat{p}\text{ and }%
\int_{Z_{n-k:n}}^{\infty}\log\left(  z/Z_{n-k:n}\right)  \frac{dH_{n}\left(
z\right)  }{\overline{H}_{n}\left(  Z_{n-k:n}\right)  }=\widehat{\gamma
}^{\left(  H\right)  }.
\]
Substituting this in $\left(  \ref{efg}\right)  ,$ leads to the definition of
$\widehat{\gamma}_{1}^{\left(  EFG\right)  }.$ By incorporating the quantity
$\overline{H}\left(  tz\right)  /\overline{H}^{\left(  1\right)  }\left(
tz\right)  $ inside the integral $\int_{1}^{\infty}\log\left(  z\right)
dH\left(  tz\right)  /\overline{H}\left(  t\right)  ,$ we get from Lemma
\ref{Lemma-1} (for $r=1)$%
\[
\frac{1}{\overline{H}\left(  t\right)  }\int_{1}^{\infty}\frac{\log\left(
z\right)  dH\left(  tz\right)  }{\overline{H}^{\left(  1\right)  }\left(
tz\right)  /\overline{H}\left(  tz\right)  }\rightarrow\gamma_{1},\text{ as
}t\rightarrow\infty.
\]
But one has to be careful, because a division by zero may occur in the
estimation procedure. Indeed, we have for instance $\overline{H}_{n}^{\left(
1\right)  }\left(  Z_{n:n}\right)  =\delta_{\left[  n:n\right]  }$ which may
be $0$ or $1.$ To avoid this boring situation, we add to the denominator a
suitable non-null sequence tending to zero and being in agreement with the
normalizing constant $\sqrt{k}$ corresponding to the limit distributions of
tail indices estimators. For convenience, to have Gaussian approximations of
order $O_{\mathbb{P}}\left(  k^{-1/2}\right)  $ (tending to zero in
probability), we choose the sequence $k^{-1}$ and we show in Lemma
\ref{Lemma-1}, that for a given sequence $t_{n}\rightarrow\infty,$ we have
\begin{equation}
\frac{1}{\overline{H}\left(  t_{n}\right)  }\int_{1}^{\infty}\frac{\log\left(
z\right)  dH\left(  t_{n}z\right)  }{\overline{H}^{\left(  1\right)  }\left(
t_{n}z\right)  /\overline{H}\left(  t_{n}z\right)  +k^{-1}}\rightarrow
\gamma_{1},\text{ as }n\rightarrow\infty.\label{lim-princ}%
\end{equation}
Thus, by letting $t_{n}=Z_{n-k:n}$ and by replacing $\overline{H}^{\left(
1\right)  }$ and $\overline{H}$ by $\overline{H}_{n}^{\left(  1\right)  }$ and
$\overline{H}_{n}$ respectively, the left-hand side becomes%
\[
\frac{1}{\overline{H}_{n}\left(  Z_{n-k:n}\right)  }\int_{Z_{n-k:n}}^{Z_{n:n}%
}\frac{\overline{H}_{n}\left(  z\right)  \log\left(  z/Z_{n-k:n}\right)
dH_{n}\left(  z\right)  }{\overline{H}_{n}^{\left(  1\right)  }\left(
z\right)  +k^{-1}\overline{H}_{n}\left(  z\right)  }.
\]
This may be rewritten into%
\begin{align*}
&  \frac{n}{k}\int_{0}^{\infty}\mathbf{1}\left(  Z_{n-k:n}<z<Z_{n:n}\right)
\frac{\overline{H}_{n}\left(  z\right)  \log\left(  z/Z_{n-k:n}\right)
dH_{n}\left(  z\right)  }{\overline{H}_{n}^{\left(  1\right)  }\left(
z\right)  +k^{-1}\overline{H}_{n}\left(  z\right)  }\\
&  =\frac{1}{k}\sum_{i=1}^{n}\mathbf{1}\left(  Z_{n-k:n}<Z_{i:n}%
<Z_{n:n}\right)  \frac{\overline{H}_{n}\left(  Z_{i:n}\right)  \log\left(
Z_{i:n}/Z_{n-k:n}\right)  }{\overline{H}_{n}^{\left(  1\right)  }\left(
Z_{i:n}\right)  +k^{-1}\overline{H}_{n}\left(  Z_{i:n}\right)  },
\end{align*}
which equals
\[
\frac{1}{k}\sum_{i=n-k+1}^{n-1}\frac{\overline{H}_{n}\left(  Z_{i:n}\right)
\log\left(  Z_{i:n}/Z_{n-k:n}\right)  }{\overline{H}_{n}^{\left(  1\right)
}\left(  Z_{i:n}\right)  +k^{-1}\overline{H}_{n}\left(  Z_{i:n}\right)
}=\frac{1}{k}\sum_{i=n-k+1}^{n-1}\frac{\left(  n-i\right)  \log\left(
Z_{i:n}/Z_{n-k:n}\right)  }{\sum_{j=i+1}^{n}\delta_{\left[  j:n\right]
}+\left(  n-i\right)  /k}.
\]
Finally, changing $i$ by $n-i$ and $j$ by $n-j+1$ yields a new (random)
weighted estimator for the EVI $\gamma_{1}$ as follows:%
\begin{equation}
\widehat{\gamma}_{1}:=\frac{1}{k}\sum_{i=1}^{k-1}\frac{i\log\left(
Z_{n-i:n}/Z_{n-k:n}\right)  }{\sum_{j=1}^{i}\delta_{\left[  n-j+1:n\right]
}+i/k}.\label{K-estimator}%
\end{equation}
For the purpose of establishing the consistency and asymptotic normality of
$\widehat{\gamma}_{1},$ we next introduce a tail empirical process for
censored data. For $x\geq1,$ let us define%
\begin{equation}
\Delta_{n}\left(  x\right)  :=\frac{n}{k}\int_{xZ_{n-k:n}}^{\infty}%
\frac{\overline{H}_{n}\left(  z\right)  dH_{n}\left(  z\right)  }{\overline
{H}_{n}^{\left(  1\right)  }\left(  z\right)  +k^{-1}\overline{H}_{n}\left(
z\right)  }.\label{deltan}%
\end{equation}
By integrating by parts, we show easily that $\widehat{\gamma}_{1}=\int
_{1}^{\infty}x^{-1}\Delta_{n}\left(  x\right)  dx,$ thereby, motivated by the
tail product-limit process for truncated data, recently given in
\cite{BchMN-16a}, we defined this tail empirical process by%
\begin{equation}
D_{n}\left(  x\right)  :=\sqrt{k}\left(  \Delta_{n}\left(  x\right)
-p^{-1}x^{-1/\gamma}\right)  ,\text{ }x\geq1,\label{process}%
\end{equation}
so that
\[
\sqrt{k}\left(  \widehat{\gamma}_{1}-\gamma_{1}\right)  =\int_{1}^{\infty
}x^{-1}D_{n}\left(  x\right)  dx.
\]
In the non censoring case $\left(  p=1\right)  ,$ we have $X\equiv Z,$
$\overline{F}_{n}=\overline{H}_{n}=\overline{H}_{n}^{\left(  1\right)  },$
with $F_{n}$ being the usual empirical cdf. In this case, we have%
\[
\Delta_{n}\left(  x\right)  =\dfrac{n}{k}\overline{F}_{n}\left(
xX_{n-k:n}\right)  +O_{\mathbb{P}}\left(  k^{-1}\right)  ,
\]
thus
\begin{equation}
D_{n}\left(  x\right)  =\sqrt{k}\left(  \dfrac{n}{k}\overline{F}_{n}\left(
xX_{n-k:n}\right)  -x^{-1/\gamma}\right)  +O_{\mathbb{P}}\left(
k^{-1/2}\right)  ,\label{D}%
\end{equation}
corresponds (asymptotically) to the tail empirical process for complete data
\citep[see, e.g., page 161 in][]{deHF06}. Note that one may think that it
would have been more natural to simply use Kaplan-Meier estimator of $F,$
given in $\left(  \ref{kaplan}\right)  ,$ to define a tail empirical process
in the censoring case. This was done in \cite{BMN-2016}, but the asymptotic
properties were only established under the condition $p>1/2,$ which would
constitute a restriction for applications. Indeed, there exist real datasets
used in case studies with proportions $p$ estimated at less than a half. We
can cite, amongst others, the aids survival data to which \cite{EnFG08}
applied their methodology and approximately got $0.28$ for $p$ (it is exacly
the same value that we ourselves will find later on in Section \ref{sec4}) and
the car liability insurance data studied in \cite{BBWG-16} and \cite{B18}
where the estimated value of $p$ was $0.40.$\medskip

\noindent The rest of the paper is organized as follows. In Section
\ref{sec2}, we provide our main result, namely two weak approximations leading
to consistency and asymptotic normality of $\widehat{\gamma}_{1},$ whose
proofs are postponed to Section \ref{sec5}. The finite sample behavior of the
proposed estimator is checked by simulation in Section \ref{sec3}, where a
comparison with the already existing ones is made as well. Section \ref{sec4}
is devoted to an application to the survival time of Australian male Aids
patients. Finally, some results that are instrumental to the proofs are given
in the Appendix.

\section{\textbf{Main results\label{sec2}}}

\noindent In the sequel, the functions $f^{\leftarrow}\left(  s\right)
:=\inf\left\{  x:f\left(  x\right)  \geq s\right\}  ,$ $0<s<1$ and
$U_{f}\left(  t\right)  :=f^{\leftarrow}\left(  1-1/t\right)  ,$ $t>1,$
respectively stand for the quantile and tail quantile functions pertaining to
cdf $f.$ For further use, we set $h=h_{n}:=U_{H}\left(  n/k\right)  $ and
define
\begin{equation}
r_{n}\left(  x\right)  :=\frac{n}{k}\int_{hx}^{\infty}\frac{\overline
{H}\left(  w\right)  dH\left(  w\right)  }{\overline{H}^{\left(  1\right)
}\left(  w\right)  +k^{-1}\overline{H}\left(  w\right)  },\text{ }x\geq1.
\label{rn}%
\end{equation}
Let us now state our first result in which we provide Gaussian approximations
both to $\sqrt{k}\left(  \Delta_{n}\left(  x\right)  -r_{n}\left(  x\right)
\right)  $ and $D_{n}\left(  x\right)  .$

\begin{theorem}
\label{Theorem1}Assume that $\overline{F}\in\mathcal{RV}_{\left(
-1/\gamma_{1}\right)  }$ and $\overline{G}\in\mathcal{RV}_{\left(
-1/\gamma_{2}\right)  }$ and let $k=k_{n}$ be an integer sequence such that
$k\rightarrow\infty$ and $k/n\rightarrow0.$ Then there exists a sequence of
Brownian bridges $\left\{  B_{n}\left(  s\right)  ;\text{ }0\leq
s\leq1\right\}  $ defined on the probability space $\left(  \Omega
,\mathcal{A},\mathbb{P}\right)  ,$ such that, for every $0\leq\tau<1/8,$ we
have%
\begin{equation}
\sup_{x\geq1}x^{\tau/\gamma}\left\vert \left(  \Delta_{n}\left(  x\right)
-r_{n}\left(  x\right)  \right)  -k^{-1/2}p^{-1}\sqrt{\dfrac{n}{k}}%
\mathcal{L}_{n}\left(  x^{-1/\gamma}\right)  \right\vert \overset{\mathbb{P}%
}{\rightarrow}0, \label{TP-C1}%
\end{equation}
as $n\rightarrow\infty,$ where $\left\{  \mathcal{L}_{n}\left(  w\right)
;0<w\leq1\right\}  $ is a centred Gaussian process defined by%
\[
\mathcal{L}_{n}\left(  w\right)  :=B_{n}^{\ast}\left(  \dfrac{k}{n}w\right)
-p^{-1}wB_{n}^{\ast}\left(  \dfrac{k}{n}\right)  -%
%TCIMACRO{\dint _{0}^{w}}%
%BeginExpansion
{\displaystyle\int_{0}^{w}}
%EndExpansion
s^{-1}\left\{  B_{n}^{\ast}\left(  \dfrac{k}{n}s\right)  -p^{-1}B_{n}\left(
p\dfrac{k}{n}s\right)  \right\}  ds,
\]
with $B_{n}^{\ast}\left(  s\right)  :=B_{n}\left(  ps\right)  -B_{n}\left(
1-\left(  1-p\right)  s\right)  ,$ for $0\leq ps<1.$ If, in addition, we
assume that
\[
\overline{H}\in2\mathcal{RV}_{\left(  -1/\gamma,\nu\right)  }\left(  A\right)
\text{ and }\overline{H}^{\left(  1\right)  }\in2\mathcal{RV}_{\left(
-1/\gamma,\nu_{1}\right)  }\left(  A_{1}\right)  ,
\]
then%
\begin{equation}
\sup_{x\geq1}x^{\tau/\gamma}\left\vert D_{n}\left(  x\right)  -p^{-1}%
\sqrt{\dfrac{n}{k}}\mathcal{L}_{n}\left(  x^{-1/\gamma}\right)  -\mathcal{R}%
_{n}\left(  x\right)  \right\vert \overset{\mathbb{P}}{\rightarrow}0,\text{ }
\label{TP-C2}%
\end{equation}
as $n\rightarrow\infty,$ where%
\[
\mathcal{R}_{n}\left(  x\right)  :=p^{-1}x^{-1/\gamma}\sqrt{k}\left\{
\dfrac{x^{\nu/\gamma}-1}{\nu\gamma}A\left(  h\right)  +\frac{\nu_{\ast}%
+x^{\nu_{\ast}/\gamma}-1}{\gamma\nu_{\ast}\left(  1-\nu_{\ast}\right)
}A_{\ast}\left(  h\right)  -A_{2}\left(  h\right)  \right\}  +o\left(
x^{-\tau/\gamma}\right)  ,
\]
with $\nu_{\ast}:=\max\left(  \nu,\nu_{1}\right)  ,$ $A_{\ast}:=\mathbf{1}%
\left\{  \nu>\nu_{1}\right\}  A-\mathbf{1}\left\{  \nu\leq\nu_{1}\right\}
A_{1}$ and $A_{2}:=\overline{H}^{\left(  1\right)  }/\overline{H}-p,$ provided
that $\sqrt{k}A_{i}\left(  h\right)  =O\left(  1\right)  ,$\textbf{ }$i=1,2$
and\textbf{ }$\sqrt{k}A\left(  h\right)  =O\left(  1\right)  .$
\end{theorem}

\noindent In the following theorem, we establish the consistency and
asymptotic normality of our new estimator $\widehat{\gamma}_{1}.$

\begin{theorem}
\label{Theorem2}Assume that $\overline{F}\in\mathcal{RV}_{\left(
-1/\gamma_{1}\right)  }$ and $\overline{G}\in\mathcal{RV}_{\left(
-1/\gamma_{2}\right)  }$ and let $k=k_{n}$ be an integer sequence such that
$k\rightarrow\infty$ and $k/n\rightarrow0,$ then $\widehat{\gamma}_{1}%
\overset{\mathbb{P}}{\rightarrow}\gamma_{1},$ as $n\rightarrow\infty.$ If, in
addition, we assume that $\overline{H}\in2\mathcal{RV}_{\left(  -1/\gamma
,\nu\right)  }\left(  A\right)  $ and $\overline{H}^{\left(  1\right)  }%
\in2\mathcal{RV}_{\left(  -1/\gamma,\nu_{1}\right)  }\left(  A_{1}\right)  ,$
so that $\sqrt{k}A\left(  h\right)  ,$ $\sqrt{k}A_{1}\left(  h\right)  $ and
$\sqrt{k}A_{2}\left(  h\right)  $ be asymptotically bounded, then%
\begin{align*}
\sqrt{k}\gamma_{1}^{-1}\left(  \widehat{\gamma}_{1}-\gamma_{1}\right)
-p^{-1}\mu_{n}  &  =\sqrt{\dfrac{n}{k}}%
%TCIMACRO{\dint _{0}^{1}}%
%BeginExpansion
{\displaystyle\int_{0}^{1}}
%EndExpansion
s^{-1}B_{n}^{\ast}\left(  \dfrac{k}{n}s\right)  ds-p^{-1}\sqrt{\dfrac{n}{k}%
}B_{n}^{\ast}\left(  \dfrac{k}{n}\right)  \medskip\\
&  +\sqrt{\dfrac{n}{k}}%
%TCIMACRO{\dint _{0}^{1}}%
%BeginExpansion
{\displaystyle\int_{0}^{1}}
%EndExpansion
s^{-1}\left\{  B_{n}^{\ast}\left(  \dfrac{k}{n}s\right)  -p^{-1}B_{n}\left(
p\dfrac{k}{n}s\right)  \right\}  \log\left(  s\right)  ds+o_{\mathbb{P}%
}\left(  1\right)  ,
\end{align*}
where%
\begin{equation}
\mu_{n}:=\frac{\sqrt{k}A\left(  h\right)  }{1-\nu}+\frac{2-\mathbb{\nu}_{\ast
}}{\left(  1-\mathbb{\nu}_{\ast}\right)  ^{2}}\sqrt{k}A_{\ast}\left(
h\right)  -\gamma\sqrt{k}A_{2}\left(  h\right)  . \label{mu-n}%
\end{equation}
Whenever $\sqrt{k}A\left(  h\right)  ,$ $\sqrt{k}A_{1}\left(  h\right)  $ and
$\sqrt{k}A_{2}\left(  h\right)  $ respectively converge to finite real numbers
$\lambda,$ $\lambda_{1}$ and $\lambda_{2},$ then
\[
\sqrt{k}\left(  \widehat{\gamma}_{1}-\gamma_{1}\right)  \overset{\mathcal{D}%
}{\rightarrow}\mathcal{N}\left(  p^{-1}\mu,\left(  9-8p\right)  \left(
\gamma_{1}^{2}/p\right)  \right)  ,
\]
where
\[
\mu:=\lambda/\left(  1-\nu\right)  +\left(  2-\mathbb{\nu}_{\ast}\right)
\left(  1-\mathbb{\nu}_{\ast}\right)  ^{-2}\lambda_{\ast}-\gamma\lambda_{2},
\]
with $\lambda_{\ast}:=\mathbf{1}\left\{  \nu>\nu_{1}\right\}  \lambda
-\mathbf{1}\left\{  \nu\leq\nu_{1}\right\}  \lambda_{1}.$
\end{theorem}

\begin{remark}
It is to be noted that for distributions in Hall's class of models,
(\citeauthor{Hall82}, \citeyear{Hall82}),\textbf{ }the three
assumptions\textbf{ }$\sqrt{k}A_{i}\left(  h\right)  =O\left(  1\right)
,$\textbf{ }$i=1,2$ and\textbf{ }$\sqrt{k}A\left(  h\right)  =O\left(
1\right)  $ may be gathered in a single one, namely $\sqrt{k}\left(
k/n\right)  ^{-\nu}=O\left(  1\right)  ,$ which is already used by
\cite{BBWG-16} in Theorem 1, written as $\sqrt{k}\left(  k/n\right)
^{\beta^{\ast}}=O\left(  1\right)  ,$ where $0<\beta^{\ast}=-\nu.$ Indeed, let
us assume that both $F$ and $G$ belong to this family, which contains the most
popular heavy-tailed cdf's, such as Burr, Fr\'{e}chet, Generalized Pareto,
Generalized Extreme Value, t-Student, ... Explicitly, there exist constants
$\gamma_{1},\gamma_{2}>0,$ $\eta_{F},\eta_{G}<0,$ $c_{F},c_{G}>0$ and
$d_{F},d_{G}\neq0,$ such that, as $x\rightarrow\infty$%
\[
\overline{F}\left(  x\right)  =c_{F}x^{-1/\gamma_{1}}\left(  1+d_{F}%
x^{\eta_{F}/\gamma_{1}}\left(  1+o\left(  1\right)  \right)  \right)  ,
\]
and%
\[
\overline{G}\left(  x\right)  =c_{G}x^{-1/\gamma_{2}}\left(  1+d_{G}%
x^{\eta_{G}/\gamma_{2}}\left(  1+o\left(  1\right)  \right)  \right)  .
\]
This implies that we have, as $x\rightarrow\infty$%
\[
\overline{H}\left(  x\right)  =cx^{-1/\gamma}\left(  1+dx^{\nu/\gamma}\left(
1+o\left(  1\right)  \right)  \right)  ,
\]
and
\[
\overline{H}^{\left(  1\right)  }\left(  x\right)  =c_{1}x^{-1/\gamma}\left(
1+d_{1}x^{\nu_{1}/\gamma}\left(  1+o\left(  1\right)  \right)  \right)  ,
\]
where $c:=c_{F}c_{G},$ $c_{1}:=pc,$ $\nu_{1}=\nu:=\max\left(  \eta_{F}%
,\eta_{G}\right)  ,$
\[%
\begin{array}
[c]{cc}%
d:=\left\{
\begin{tabular}
[c]{lll}%
$d_{F}$ & if & $\eta_{F}>\eta_{G},$\\
$d_{G}$ & if & $\eta_{F}<\eta_{G},$\\
$d_{F}+d_{G}$ & if & $\eta_{F}=\eta_{G}=\nu,$%
\end{tabular}
\right.  & d_{1}:=\left\{
\begin{tabular}
[c]{lll}%
$\frac{\eta_{F}-1}{p\eta_{F}-1}d_{F}$ & if & $\eta_{F}>\eta_{G},$\\
$\frac{1}{1-\left(  1-p\right)  \eta_{G}}d_{G}$ & if & $\eta_{F}<\eta_{G},$\\
$\frac{\eta_{1}-1}{p\eta_{F}+\left(  1-p\right)  \eta_{G}-1}\left(
d_{F}+d_{G}\right)  $ & if & $\eta_{F}=\eta_{G}=\nu.$%
\end{tabular}
\right.
\end{array}
\]
It is easy to check that both $\overline{H}$ and $\overline{H}^{\left(
1\right)  }$satisfy the second-order condition of regular variation with
convergence rates $A\left(  t\right)  \sim\nu\gamma dt^{\nu/\gamma}$ and
$A_{1}\left(  t\right)  \sim\nu\gamma d_{1}t^{\nu/\gamma},$ respectively.
Moreover, we have $\overline{H}^{\left(  1\right)  }\left(  x\right)
/\overline{H}\left(  x\right)  =p+p\left(  d_{1}-d\right)  x^{\nu/\gamma
}\left(  1+o\left(  1\right)  \right)  ,$ and therefore $A_{2}\left(
t\right)  \sim p\left(  d_{1}-d\right)  t^{\nu/\gamma}.$ In this context,
$U_{H}\left(  t\right)  \sim c^{-1}t^{\gamma},$ as $t\rightarrow\infty,$ thus
$h=U_{H}\left(  n/k\right)  \sim c^{-1}\left(  k/n\right)  ^{-\gamma},$ as
$n\rightarrow\infty.$ Consequently, by replacing $t$ by $c^{-1}\left(
k/n\right)  ^{-\gamma},$ we get that $\sqrt{k}A_{1}\left(  h\right)  ,$
$\sqrt{k}A_{2}\left(  h\right)  $ and $A\left(  h\right)  $ are indeed all of
convergence rate $\sqrt{k}\left(  k/n\right)  ^{-\nu}.$
\end{remark}

\begin{remark}
We emphasize that the asymptotic variance of $\widehat{\gamma}_{1}$ is
$\left(  9-8p\right)  $ times that of $\widehat{\gamma}_{1}^{\left(
EFG\right)  }.$ This factor corresponds to earlier findings for
weighted/kernel estimators of the extreme value index in non-censored data
cases, see for instance \cite{CDM} (complete data) and \cite{BchMN-16b}
(truncated data).
\end{remark}

\section{\textbf{Simulation study\label{sec3}}}

\noindent Now, we carry out an extensive simulation study to illustrate the
behavior of the proposed estimator $\widehat{\gamma}_{1}$ and compare its
performance, in terms of (absolute) bias and root of the mean squared error
(RMSE), with those of $\widehat{\gamma}_{1}^{\left(  EFG\right)  }$ and
$\widehat{\gamma}^{\left(  WW1\right)  }.$\ To this end, we consider Burr's,
Fr\'{e}chet's and the log-gamma models respectively defined, for $x>0,$ by:

\begin{itemize}
\item Burr $(\beta,\tau,\lambda):$ $F\left(  x\right)  =1-\left(  \dfrac
{\beta}{\beta+x^{\tau}}\right)  ^{\lambda},$ for $\beta,\tau,\lambda>0,$ with
$\gamma=1/(\tau\lambda).\medskip$

\item Fr\'{e}chet $\left(  \gamma\right)  :$ $F(x)=\exp\left(  -x^{-1/\gamma
}\right)  ,$ for $\gamma>0.\medskip$

\item log-gamma $(\gamma,\beta):$ $F\left(  x\right)  =\left(  \beta^{\gamma
}\Gamma\left(  \gamma\right)  \right)  ^{-1}\int_{0}^{x}t^{-1}\left(  \log
t\right)  ^{\gamma-1}\exp\left(  -\beta^{-1}\ln t\right)  dt,$ for $\gamma>0$
and $\beta>0.$
\end{itemize}

\noindent We select two values, $33\%$ and $70\%,$ for the proportion $p$ of
observed upper statistics and we make several combinations of the parameters
of each model. For each case, we generate $200$ samples of size $n=200$ and we
take our overall results (given in the form of graphical representations) by
averaging over all $200$ independent replications. We consider three censoring
schemes, namely Burr censored by Burr, Fr\'{e}chet censored by Fr\'{e}chet and
log-gamma censored by log-Gamma, that we illustrate by Figures \ref{M1},
\ref{M2} and \ref{M3} respectively. In each figure, we represent the biases
and the RMSE's of all three estimators $\widehat{\gamma}_{1},$ $\widehat
{\gamma}_{1}^{\left(  EFG\right)  }$ and $\widehat{\gamma}^{\left(
WW1\right)  }$ as functions of the number $k$ of the largest order statistics.
Our overall conclusion is that, from the top panels of all three figures, we
see that the newly proposed estimator $\widehat{\gamma}_{1}$ and the adapted
Hill one $\widehat{\gamma}_{1}^{\left(  EFG\right)  }$ perform almost equally
well and better than Worms's estimator $\widehat{\gamma}^{\left(  WW1\right)
}$ in the strong censoring case. However, the bottom panels of Figure \ref{M1}
and Figure \ref{M2} show that, for the weak censoring scenario, the latter has
a slight edge (especially for small values of $k)$ over the other two which
still behave almost similarly. This agrees with the conclusion of \cite{B17}
and means that $\widehat{\gamma}^{\left(  WW1\right)  }$ is not reliable
enough in the strong censoring situation. We notice, from Figure \ref{M3},
that when considering the log-gamma model with strong censoring, the estimator
$\widehat{\gamma}_{1}^{\left(  EFG\right)  }$ outperforms the remaining two.
But, with weak censoring, $\widehat{\gamma}_{1}$ is clearly better than
$\widehat{\gamma}_{1}^{\left(  EFG\right)  }$ while $\widehat{\gamma}^{\left(
WW1\right)  }$ does not even work in this case, which is something of very striking.%

%TCIMACRO{\FRAME{ftbpFU}{14.713cm}{14.6712cm}{0pt}{\Qcb{Bias (left panel) and
%RMSE (right panel) of $\widehat{\gamma}_{1}$ (red) $\widehat{\gamma}%
%_{1}^{\left(  EFG\right)  }$ (black) and $\widehat{\gamma}_{1}^{\left(
%WW1\right)  }$ (green) based on $200$ samples of size $200$ from a Burr
%distribution censored by another Burr model with $p=0.33$ (top) and $p=0.70$
%(bottom).}}{\Qlb{M1}}{burr.eps}{\special{ language "Scientific Word";
%type "GRAPHIC";  display "USEDEF";  valid_file "F";  width 14.713cm;
%height 14.6712cm;  depth 0pt;  original-width 5.7648in;
%original-height 5.7493in;  cropleft "0";  croptop "1";  cropright "1";
%cropbottom "0";
%filename 'C:/Users/Djamel MERAGHNI/Dropbox/d/Djamel/burr.eps';file-properties "XNPEU";}%
%} }%
%BeginExpansion
\begin{figure}
[ptb]
\begin{center}
\includegraphics[
height=14.6712cm,
width=14.713cm
]%
{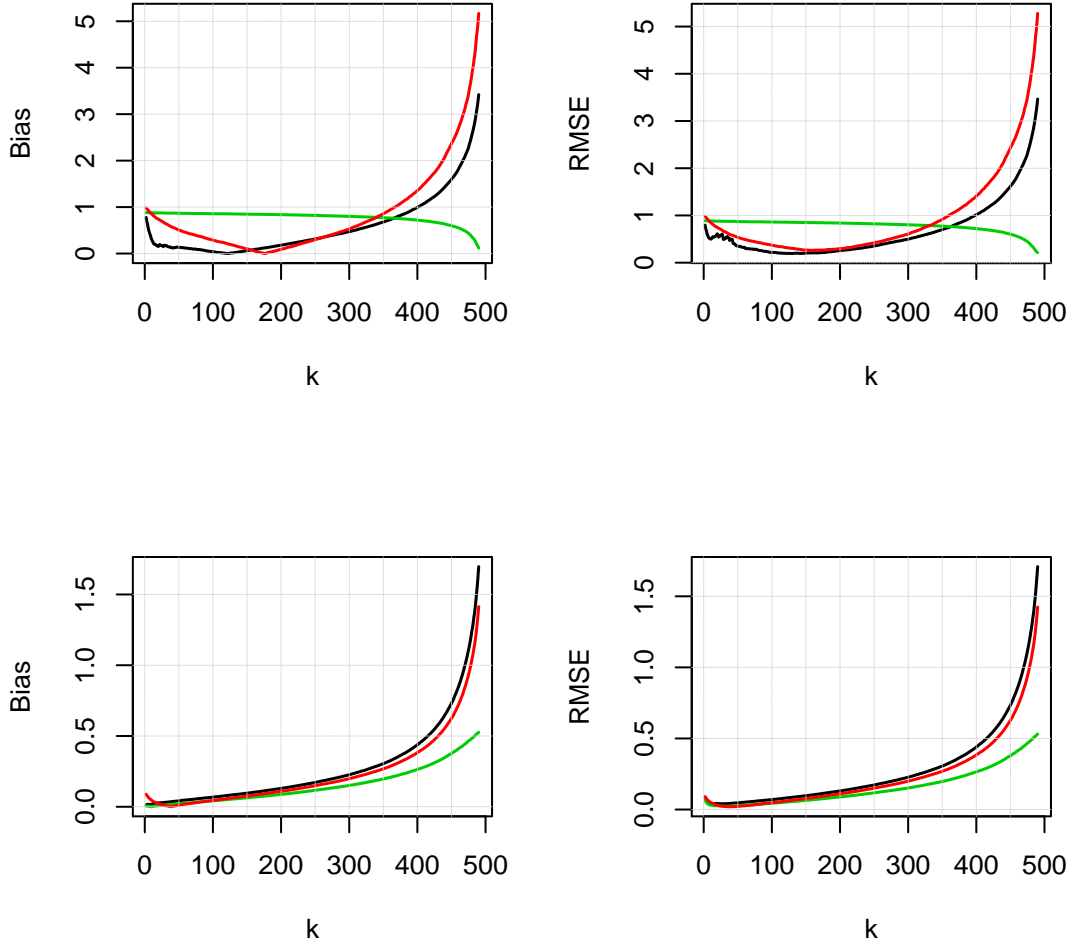}%
\caption{Bias (left panel) and RMSE (right panel) of $\widehat{\gamma}_{1}$
(red) $\widehat{\gamma}_{1}^{\left(  EFG\right)  }$ (black) and $\widehat
{\gamma}_{1}^{\left(  WW1\right)  }$ (green) based on $200$ samples of size
$200$ from a Burr distribution censored by another Burr model with $p=0.33$
(top) and $p=0.70$ (bottom).}%
\label{M1}%
\end{center}
\end{figure}
%EndExpansion
%

%TCIMACRO{\FRAME{ftbpFU}{5.7925in}{5.7761in}{0pt}{\Qcb{Bias (left panel) and
%RMSE (right panel) of $\widehat{\gamma}_{1}$ (red) $\widehat{\gamma}%
%_{1}^{\left(  EFG\right)  }$ (black) and $\widehat{\gamma}_{1}^{\left(
%WW1\right)  }$ (green) based on $200$ samples of size $200$ from a Fr\'{e}chet
%distribution censored by another Fr\'{e}chet model with $p=0.33$ (top) and
%$p=0.70$ (bottom).}}{\Qlb{M2}}{frech.eps}%
%{\special{ language "Scientific Word";  type "GRAPHIC";
%maintain-aspect-ratio TRUE;  display "USEDEF";  valid_file "F";
%width 5.7925in;  height 5.7761in;  depth 0pt;  original-width 5.7648in;
%original-height 5.7493in;  cropleft "0";  croptop "1";  cropright "1";
%cropbottom "0";
%filename 'C:/Users/Djamel MERAGHNI/Dropbox/d/Djamel/frech.eps';file-properties "XNPEU";}%
%} }%
%BeginExpansion
\begin{figure}
[ptb]
\begin{center}
\includegraphics[
height=5.7761in,
width=5.7925in
]%
{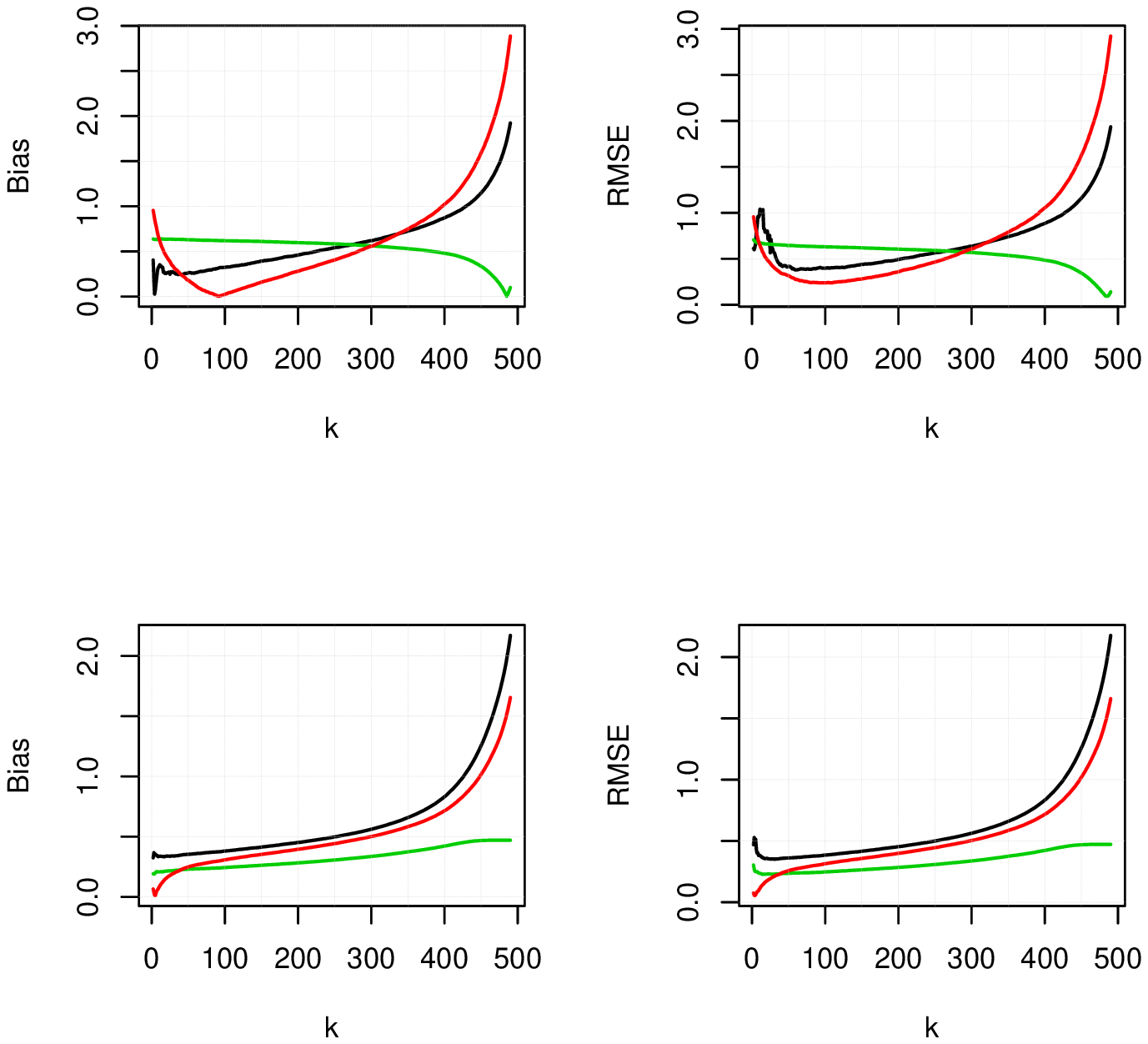}%
\caption{Bias (left panel) and RMSE (right panel) of $\widehat{\gamma}_{1}$
(red) $\widehat{\gamma}_{1}^{\left(  EFG\right)  }$ (black) and $\widehat
{\gamma}_{1}^{\left(  WW1\right)  }$ (green) based on $200$ samples of size
$200$ from a Fr\'{e}chet distribution censored by another Fr\'{e}chet model
with $p=0.33$ (top) and $p=0.70$ (bottom).}%
\label{M2}%
\end{center}
\end{figure}
%EndExpansion
%

%TCIMACRO{\FRAME{ftbpFU}{5.7925in}{5.7761in}{0pt}{\Qcb{Bias (left panel) and
%RMSE (right panel) of $\widehat{\gamma}_{1}$ (red) $\widehat{\gamma}%
%_{1}^{\left(  EFG\right)  }$ (black) and $\widehat{\gamma}_{1}^{\left(
%WW1\right)  }$ (green) based on $200$ samples of size $200$ from a log-gamma
%distribution censored by another log-gamma model with $p=0.33$ (top) and
%$p=0.70$ (bottom).}}{\Qlb{M3}}{gamma.eps}%
%{\special{ language "Scientific Word";  type "GRAPHIC";
%maintain-aspect-ratio TRUE;  display "USEDEF";  valid_file "F";
%width 5.7925in;  height 5.7761in;  depth 0pt;  original-width 5.7648in;
%original-height 5.7493in;  cropleft "0";  croptop "1";  cropright "1";
%cropbottom "0";
%filename 'C:/Users/Djamel MERAGHNI/Dropbox/d/Djamel/gamma.eps';file-properties "XNPEU";}%
%} }%
%BeginExpansion
\begin{figure}
[ptb]
\begin{center}
\includegraphics[
height=5.7761in,
width=5.7925in
]%
{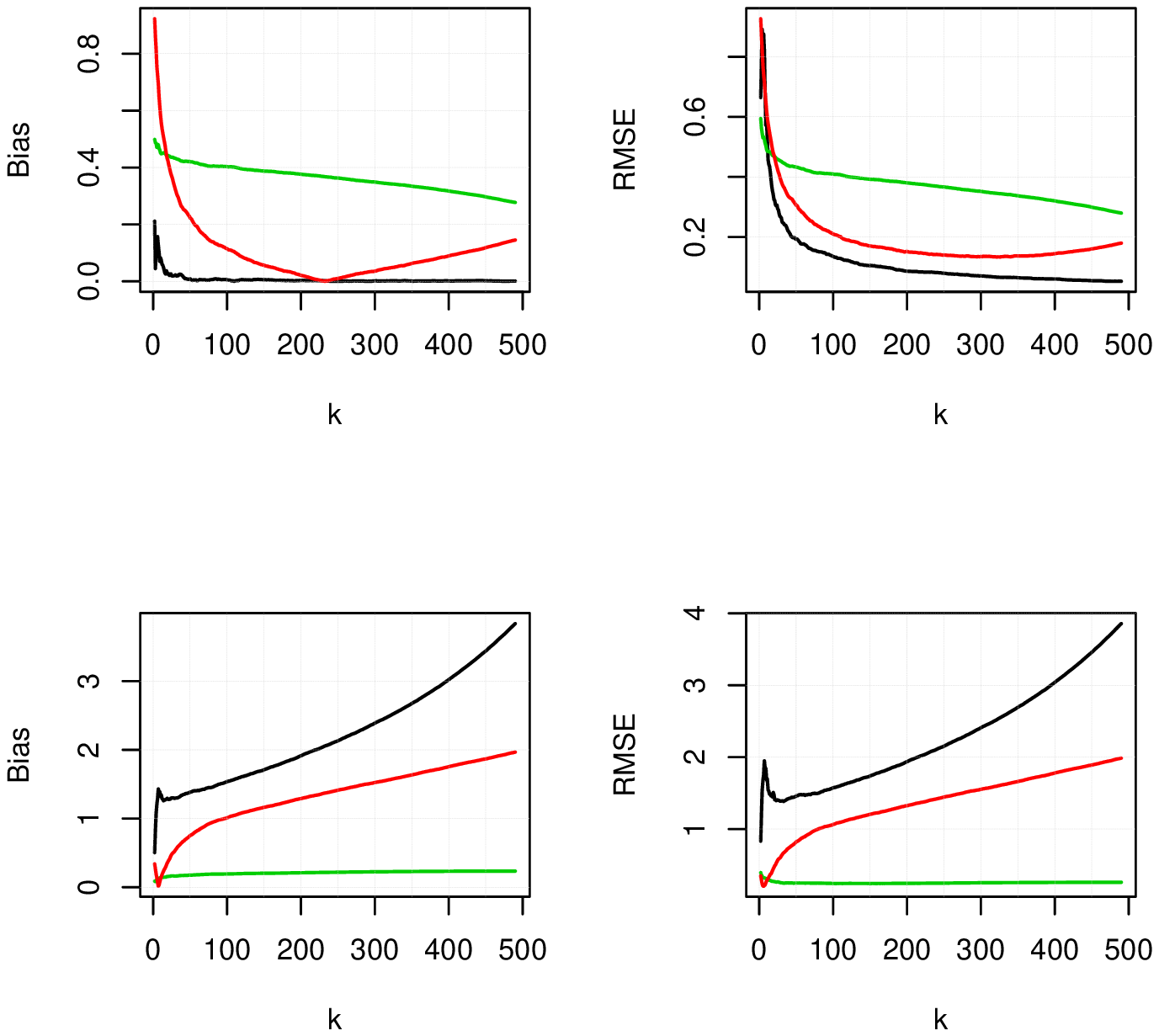}%
\caption{Bias (left panel) and RMSE (right panel) of $\widehat{\gamma}_{1}$
(red) $\widehat{\gamma}_{1}^{\left(  EFG\right)  }$ (black) and $\widehat
{\gamma}_{1}^{\left(  WW1\right)  }$ (green) based on $200$ samples of size
$200$ from a log-gamma distribution censored by another log-gamma model with
$p=0.33$ (top) and $p=0.70$ (bottom).}%
\label{M3}%
\end{center}
\end{figure}
%EndExpansion

\section{\textbf{Application to Australian Aids data\label{sec4}}}

The data file consists in $n=2843$ Australian patients who were diagnosed with
Aids on July $1^{st},$ 1991. The file contains the identification number, the
dates of first diagnosis, birth and death, as well as the state and the
encrypted transmission category. The data are available in the package "MASS"
of the statistical software R. Our objective is to apply the newly proposed
estimation procedure to evaluate the tail index $\gamma_{1}$ of the survival
time of the patients. To this end, we first select\textbf{ }the optimal number
of top statistics used in the estimate computation. By applying the adaptive
algorithm of Reiss and Thomas (see, \cite{ReTo07}, page 137), we find that
$522$ extreme observations are needed to obtain a proportion estimate value
$\widehat{p}=0.28.$ This represents a strong censoring rate of around $70\%$
for which $\widehat{\gamma}^{\left(  WW1\right)  }$ is not recommended for the
estimation of the tail index $\gamma_{1},$ as seen in Section \ref{sec3}.
Therefore, we only compute the other two estimates $\widehat{\gamma}_{1}$ and
$\widehat{\gamma}_{1}^{\left(  EFG\right)  }$ for which the abovementioned
technique gives $0.20.$ and $0.21$ respectively. Note that this latter value
differs from that found by \cite{EnFG08} who used a graphical approach which
is not as objective as a numerical procedure.

\section{\textbf{Proofs\label{sec5}}}

\subsection{Preliminaries}

Let $U_{i}:=\delta_{i}H^{\left(  1\right)  }\left(  Z_{i}\right)  +\left(
1-\delta_{i}\right)  \left(  \theta+H^{\left(  0\right)  }\left(
Z_{i}\right)  \right)  ,$ $i=1,...,n,$ be a sequence of iid rv's uniformly
distributed on $(0,1)$ \citep[][]{EnKo92}, and define the corresponding
empirical cdf and empirical process by%
\begin{equation}
\mathbb{U}_{n}\left(  s\right)  :=\frac{1}{n}\sum\limits_{i=1}^{n}%
\mathbf{1}\left\{  U_{i}\leq s\right\}  \text{ and }\alpha_{n}\left(
s\right)  :=\sqrt{n}\left(  \mathbb{U}_{n}\left(  s\right)  -s\right)  ,\text{
}0\leq s\leq1, \label{e}%
\end{equation}
respectively. Thereby, we may represent, almost surely, both $H_{n}^{\left(
0\right)  }$ and $H_{n}^{\left(  1\right)  }$ in term of $\mathbb{U}_{n},$ as
follows
\begin{equation}
H_{n}^{\left(  0\right)  }\left(  v\right)  =\mathbb{U}_{n}\left(  H^{\left(
0\right)  }\left(  v\right)  +\theta\right)  -\mathbb{U}_{n}\left(
\theta\right)  ,\text{ for }0<H^{\left(  0\right)  }\left(  v\right)
<1-\theta, \label{HN0}%
\end{equation}
and
\begin{equation}
H_{n}^{\left(  1\right)  }\left(  v\right)  =\mathbb{U}_{n}\left(  H^{\left(
1\right)  }\left(  v\right)  \right)  ,\text{ for }0<H^{\left(  1\right)
}\left(  v\right)  <\theta. \label{HN1}%
\end{equation}
For more details, one refers to \cite{DeEn96}. Therefore, in view of the
previous representations, we have almost surely%
\begin{equation}
\sqrt{n}\left(  \overline{H}_{n}^{\left(  1\right)  }\left(  v\right)
-\overline{H}^{\left(  1\right)  }\left(  v\right)  \right)  =\alpha
_{n}\left(  \theta\right)  -\alpha_{n}\left(  \theta-\overline{H}^{\left(
1\right)  }\left(  v\right)  \right)  ,\text{ for }0<\overline{H}^{\left(
1\right)  }\left(  v\right)  <\theta, \label{rep-H1}%
\end{equation}
and%
\begin{equation}
\sqrt{n}\left(  \overline{H}_{n}^{\left(  0\right)  }\left(  v\right)
-\overline{H}^{\left(  0\right)  }\left(  v\right)  \right)  =-\alpha
_{n}\left(  1-\overline{H}^{\left(  0\right)  }\left(  v\right)  \right)
,\text{ for }0<\overline{H}^{\left(  0\right)  }\left(  v\right)  <1-\theta.
\label{rep-H0}%
\end{equation}
Our methodology strongly relies on the well-known Gaussian approximation,
given by \cite{CsCsHM86} in Corollary 2.1, which says that on the probability
space $\left(  \Omega,\mathcal{A},\mathbb{P}\right)  ,$ there exists a
sequence of Brownian bridges $\left\{  B_{n}\left(  s\right)  ;\text{ }0\leq
s\leq1\right\}  $ such that for every $0<\lambda<\infty$ and $0\leq\eta<1/4,$%
\begin{equation}
\sup_{\lambda/n\leq s\leq1-\lambda/n}\frac{n^{\eta}\left\vert \alpha
_{n}\left(  s\right)  -B_{n}\left(  s\right)  \right\vert }{\left(  s\left(
1-s\right)  \right)  ^{1/2-\eta}}=O_{\mathbb{P}}\left(  1\right)  .
\label{approx}%
\end{equation}
For the increments $\alpha_{n}\left(  \theta\right)  -\alpha_{n}\left(
\theta-s\right)  ,$ we will need an approximation of the same type as $\left(
\ref{approx}\right)  .$ Following similar arguments, mutatis mutandis, as
those used to in the proof of assertions $\left(  2.2\right)  $ of Theorem 2.1
and $\left(  2.8\right)  $ of Theorem 2.2 in \cite{CsCsHM86}, \cite{Necir2017}%
, recently showed, in a technical report, that for every $0<\theta<1,$
$0<\lambda<\infty$ and $0\leq\eta<1/4,$ one also has%
\begin{equation}
\sup_{\lambda/n\leq s<\theta}\frac{n^{\eta}\left\vert \left\{  \alpha
_{n}\left(  \theta\right)  -\alpha_{n}\left(  \theta-s\right)  \right\}
-B_{n}\left(  s\right)  \right\vert }{s^{1/2-\eta}}=O_{\mathbb{P}}\left(
1\right)  . \label{approx2}%
\end{equation}

\subsection{Proof of Theorem \ref{Theorem1}}

To get the asymptotic weak approximation given in Theorem \ref{Theorem1}, we
will perform successive decompositions that will produce several remainder
terms $\mathbf{R}_{ni}\left(  x\right)  .$ We show, in Lemmas \ref{Lemma-6},
\ref{Lemma-8}, \ref{Lemma-9}, \ref{Lemma-10}, \ref{Lemma-12}, \ref{Lemma-13}
and \ref{Lemma-14} of the Appendix, that for a every fixed $0\leq\tau<1/8$ and
for all large $n,$ we have $\mathbf{R}_{ni}\left(  x\right)  =o_{\mathbb{P}%
}\left(  x^{-\tau/\gamma}\right)  ,$ for $i=1,...,16,$ uniformly over
$x\geq1.$ Then, once one of the remainder terms appears in the following
decompositions it is systematically replaced by $o_{\mathbb{P}}\left(
x^{-\tau/\gamma}\right)  .$ Let us begin by setting $\varphi_{n}\left(
u\right)  :=\left(  u+k^{-1}\right)  ^{-1}$ and rewrite both $\left(
\ref{rn}\right)  $ and $\left(  \ref{deltan}\right)  $ into%
\[
\Delta_{n}\left(  x\right)  =\frac{n}{k}\int_{xZ_{n-k:n}}^{\infty}\varphi
_{n}\left(  \frac{\overline{H}_{n}^{\left(  1\right)  }\left(  w\right)
}{\overline{H}_{n}\left(  w\right)  }\right)  dH_{n}\left(  w\right)  \text{
and }r_{n}\left(  x\right)  =\frac{n}{k}\int_{xh}^{\infty}\varphi_{n}\left(
\frac{\overline{H}^{\left(  1\right)  }\left(  w\right)  }{\overline{H}\left(
w\right)  }\right)  dH\left(  w\right)  .
\]
Observe now that $\sqrt{k}\left(  \Delta_{n}\left(  x\right)  -r_{n}\left(
x\right)  \right)  =:T_{n}\left(  x\right)  $ may be rewritten into%
\begin{align*}
&  \frac{n}{k}\int_{xZ_{n-k:n}}^{\infty}\left\{  \varphi_{n}\left(
\frac{\overline{H}_{n}^{\left(  1\right)  }\left(  w\right)  }{\overline
{H}_{n}\left(  w\right)  }\right)  -\varphi_{n}\left(  \frac{\overline
{H}^{\left(  1\right)  }\left(  w\right)  }{\overline{H}\left(  w\right)
}\right)  \right\}  dH_{n}\left(  w\right) \\
&  +\frac{n}{\sqrt{k}}\int_{xZ_{n-k:n}}^{\infty}\varphi_{n}\left(
\frac{\overline{H}^{\left(  1\right)  }\left(  w\right)  }{\overline{H}\left(
w\right)  }\right)  d\left(  H_{n}\left(  w\right)  -H\left(  w\right)
\right)  +\frac{n}{k}\int_{xZ_{n-k:n}}^{xh}\varphi_{n}\left(  \frac
{\overline{H}^{\left(  1\right)  }\left(  w\right)  }{\overline{H}\left(
w\right)  }\right)  dH\left(  w\right)  .
\end{align*}
In view of Taylor's expansion, we write $\varphi_{n}\left(  u\right)
=u^{-1}+O\left(  k^{-1}\right)  u^{-2},$ then by using this latter twice, for
$u=\overline{H}_{n}^{\left(  1\right)  }\left(  w\right)  /\overline{H}%
_{n}\left(  w\right)  $ and $u=\overline{H}^{\left(  1\right)  }\left(
w\right)  /\overline{H}\left(  w\right)  ,$ we may decompose $T_{n}\left(
x\right)  $ into the sum of%
\begin{align*}
&  \frac{n}{\sqrt{k}}\int_{xZ_{n-k:n}}^{\infty}\left(  \frac{\overline{H}%
_{n}\left(  w\right)  }{\overline{H}_{n}^{\left(  1\right)  }\left(  w\right)
}-\frac{\overline{H}\left(  w\right)  }{\overline{H}^{\left(  1\right)
}\left(  w\right)  }\right)  dH_{n}\left(  w\right) \\
&  -\frac{n}{\sqrt{k}}\int_{xZ_{n-k:n}}^{\infty}\frac{\overline{H}\left(
w\right)  }{\overline{H}^{\left(  1\right)  }\left(  w\right)  }d\left(
H_{n}\left(  w\right)  -H\left(  w\right)  \right)  +\frac{n}{k}%
\int_{xZ_{n-k:n}}^{xh}\frac{\overline{H}\left(  w\right)  }{\overline
{H}^{\left(  1\right)  }\left(  w\right)  }dH\left(  w\right)  ,
\end{align*}
and three remainder terms%
\[
\mathbf{R}_{n1}\left(  x\right)  :=O_{\mathbb{P}}\left(  k^{-1}\right)
\frac{n}{\sqrt{k}}\int_{xZ_{n-k:n}}^{\infty}\left\{  \left(  \frac
{\overline{H}_{n}\left(  w\right)  }{\overline{H}_{n}^{\left(  1\right)
}\left(  w\right)  }\right)  ^{2}+\left(  \frac{\overline{H}\left(  w\right)
}{\overline{H}^{\left(  1\right)  }\left(  w\right)  }\right)  ^{2}\right\}
dH_{n}\left(  w\right)  ,
\]%
\[
\mathbf{R}_{n2}\left(  x\right)  :=O_{\mathbb{P}}\left(  k^{-1}\right)
\frac{n}{\sqrt{k}}\int_{xZ_{n-k:n}}^{\infty}\left(  \frac{\overline{H}\left(
w\right)  }{\overline{H}^{\left(  1\right)  }\left(  w\right)  }\right)
^{2}d\left(  H_{n}\left(  w\right)  -H\left(  w\right)  \right)
\]
and%
\[
R_{n3}\left(  x\right)  :=O_{\mathbb{P}}\left(  k^{-1}\right)  \frac{n}%
{\sqrt{k}}\int_{xZ_{n-k:n}}^{\infty}\frac{\overline{H}\left(  w\right)
}{\overline{H}^{\left(  1\right)  }\left(  w\right)  }dH\left(  w\right)  .
\]
Let us now also decompose $T_{n}\left(  x\right)  $ into the sum of%
\[
T_{n1}\left(  x\right)  :=\frac{n}{\sqrt{k}}\int_{xZ_{n-k:n}}^{\infty}%
\frac{\overline{H}_{n}\left(  w\right)  -\overline{H}\left(  w\right)
}{\overline{H}_{n}^{\left(  1\right)  }\left(  w\right)  }dH_{n}\left(
w\right)  ,
\]%
\[
T_{n2}\left(  x\right)  :=\frac{n}{\sqrt{k}}\int_{xZ_{n-k:n}}^{\infty}\left(
\frac{1}{\overline{H}_{n}^{\left(  1\right)  }\left(  w\right)  }-\frac
{1}{\overline{H}^{\left(  1\right)  }\left(  w\right)  }\right)  \overline
{H}\left(  w\right)  dH_{n}\left(  w\right)  ,
\]%
\[
T_{n3}\left(  x\right)  :=\frac{n}{\sqrt{k}}\int_{xZ_{n-k:n}}^{\infty}%
\frac{\overline{H}\left(  w\right)  }{\overline{H}^{\left(  1\right)  }\left(
w\right)  }d\left(  H_{n}\left(  w\right)  -H\left(  w\right)  \right)  ,
\]
and%
\[
T_{n4}\left(  x\right)  :=\frac{n}{\sqrt{k}}\int_{xZ_{n-k:n}}^{xh}%
\frac{\overline{H}\left(  w\right)  }{\overline{H}^{\left(  1\right)  }\left(
w\right)  }dH\left(  w\right)  .
\]
For the purpose of establishing Gaussian approximations to $T_{ni}\left(
x\right)  ,$ we introduce the following two crucial tail empirical processes
$\beta_{n}$ and $\beta_{n}^{\ast}$ defined, for $w\geq1,$ by%
\begin{equation}
\beta_{n}\left(  w\right)  :=\dfrac{n}{\sqrt{k}}\left(  \overline{H}%
_{n}^{\left(  1\right)  }\left(  Z_{n-k:n}w\right)  -\overline{H}^{\left(
1\right)  }\left(  Z_{n-k:n}w\right)  \right)  , \label{betan}%
\end{equation}
and%
\begin{equation}
\beta_{n}^{\ast}\left(  w\right)  :=\dfrac{n}{\sqrt{k}}\left(  \overline
{H}_{n}\left(  Z_{n-k:n}w\right)  -\overline{H}\left(  Z_{n-k:n}w\right)
\right)  . \label{beta start}%
\end{equation}

\subsubsection{Asymptotic representations to $T_{ni}\left(  x\right)  $ in
terms of $\beta_{n}$ and $\beta_{n}^{\ast}$}

Let us now decompose $T_{n1}\left(  x\right)  $ into the sum of%
\[
T_{n1}^{\left(  1\right)  }\left(  x\right)  :=\frac{n}{\sqrt{k}}%
\int_{xZ_{n-k:n}}^{\infty}\frac{\overline{H}_{n}\left(  w\right)
-\overline{H}\left(  w\right)  }{\overline{H}^{\left(  1\right)  }\left(
w\right)  }dH_{n}\left(  w\right)  ,
\]
and%
\[
\mathbf{R}_{n4}\left(  x\right)  :=\frac{n}{\sqrt{k}}\int_{xZ_{n-k:n}}%
^{\infty}\left(  \frac{1}{\overline{H}_{n}^{\left(  1\right)  }\left(
w\right)  }-\frac{1}{\overline{H}^{\left(  1\right)  }\left(  w\right)
}\right)  \left(  \overline{H}_{n}\left(  w\right)  -\overline{H}\left(
w\right)  \right)  dH_{n}\left(  w\right)  .
\]
The change of variable $w=Z_{n-k:n}u$ and the definition $\left(
\ref{beta start}\right)  ,$ yield that%
\[
T_{n1}^{\left(  1\right)  }\left(  x\right)  =\int_{x}^{\infty}\frac{\beta
_{n}^{\ast}\left(  u\right)  }{\overline{H}^{\left(  1\right)  }\left(
Z_{n-k:n}u\right)  }dH_{n}\left(  Z_{n-k:n}u\right)  ,
\]
and%
\[
\mathbf{R}_{n4}\left(  x\right)  =\int_{x}^{\infty}\left(  \frac{1}%
{\overline{H}_{n}^{\left(  1\right)  }\left(  Z_{n-k:n}u\right)  }-\frac
{1}{\overline{H}^{\left(  1\right)  }\left(  Z_{n-k:n}u\right)  }\right)
\beta_{n}^{\ast}\left(  u\right)  dH_{n}\left(  Z_{n-k:n}u\right)  .
\]
The term $T_{n1}^{\left(  1\right)  }\left(  x\right)  $ may in turn be
decomposed into%
\[
T_{n1}^{\left(  2\right)  }\left(  x\right)  :=\frac{1}{p}\frac{n}{k}\int
_{x}^{\infty}u^{1/\gamma}\beta_{n}^{\ast}\left(  u\right)  dH_{n}\left(
Z_{n-k:n}u\right)  ,
\]
and%
\[
\mathbf{R}_{n5}\left(  x\right)  :=\int_{x}^{\infty}\left\{  \frac
{1}{\overline{H}^{\left(  1\right)  }\left(  Z_{n-k:n}u\right)  }-\frac{n}%
{k}\frac{1}{p}u^{1/\gamma}\right\}  \beta_{n}^{\ast}\left(  u\right)
dH_{n}\left(  Z_{n-k:n}u\right)  .
\]
Thus, we end up with
\[
T_{n1}\left(  x\right)  =\frac{1}{p}\frac{n}{k}\int_{x}^{\infty}u^{1/\gamma
}\beta_{n}^{\ast}\left(  u\right)  dH_{n}\left(  Z_{n-k:n}u\right)
+o_{\mathbb{P}}\left(  x^{-\tau/\gamma}\right)  .
\]
Now, we use similar decompositions to the second term $T_{n2}\left(  x\right)
.$ Observe that this latter equals%
\[
T_{n2}^{\left(  1\right)  }\left(  x\right)  :=-\int_{x}^{\infty}%
\frac{\overline{H}\left(  Z_{n-k:n}u\right)  }{\left[  \overline{H}^{\left(
1\right)  }\left(  Z_{n-k:n}u\right)  \right]  ^{2}}\beta_{n}\left(  u\right)
dH_{n}\left(  Z_{n-k:n}u\right)  ,
\]
plus a remainder term%
\[
\mathbf{R}_{n6}\left(  x\right)  :=\int_{x}^{\infty}\left\{  \frac
{1}{\overline{H}_{n}^{\left(  1\right)  }\left(  Z_{n-k:n}u\right)  }-\frac
{1}{\overline{H}^{\left(  1\right)  }\left(  Z_{n-k:n}u\right)  }\right\}
\frac{\overline{H}\left(  Z_{n-k:n}u\right)  }{\left[  \overline{H}^{\left(
1\right)  }\left(  Z_{n-k:n}u\right)  \right]  ^{2}}\beta_{n}\left(  u\right)
dH_{n}\left(  Z_{n-k:n}u\right)  .
\]
Likewise, $T_{n2}^{\left(  1\right)  }\left(  x\right)  $ may be rewritten as
the sum of%
\[
T_{n2}^{\left(  2\right)  }\left(  x\right)  :=-\frac{1}{p^{2}}\frac{n}{k}%
\int_{x}^{\infty}u^{1/\gamma}\beta_{n}\left(  u\right)  dH_{n}\left(
Z_{n-k:n}u\right)  ,
\]
and%
\[
\mathbf{R}_{n7}\left(  x\right)  :=-\int_{x}^{\infty}\left\{  \frac
{\overline{H}\left(  Z_{n-k:n}u\right)  }{\left(  \overline{H}^{\left(
1\right)  }\left(  Z_{n-k:n}u\right)  \right)  ^{2}}-\frac{1}{p^{2}}\frac
{n}{k}u^{1/\gamma}\right\}  \beta_{n}\left(  u\right)  dH_{n}\left(
Z_{n-k:n}u\right)  .
\]
Thereby%
\[
T_{n2}\left(  x\right)  =-\frac{1}{p^{2}}\frac{n}{k}\int_{x}^{\infty
}u^{1/\gamma}\beta_{n}\left(  u\right)  dH_{n}\left(  Z_{n-k:n}u\right)
+o_{\mathbb{P}}\left(  x^{-\tau/\gamma}\right)  .
\]
For the third term $T_{n3}\left(  x\right)  ,$ we make a change of variables
and an integration by parts to get%
\[
T_{n3}\left(  x\right)  =\frac{\overline{H}\left(  xZ_{n-k:n}\right)
}{\overline{H}^{\left(  1\right)  }\left(  xZ_{n-k:n}\right)  }\beta_{n}%
^{\ast}\left(  x\right)  +\int_{x}^{\infty}\beta_{n}^{\ast}\left(  u\right)
d\left\{  \frac{\overline{H}\left(  Z_{n-k:n}u\right)  }{\overline{H}^{\left(
1\right)  }\left(  Z_{n-k:n}u\right)  }\right\}  .
\]
Making use of Lemma \ref{Lemma-3}, we get $T_{n3}\left(  x\right)  =\dfrac
{1}{p}\beta_{n}^{\ast}\left(  x\right)  +o_{\mathbb{P}}\left(  x^{-\tau
/\gamma}\right)  .$ The fourth term $T_{n4}\left(  x\right)  $ may in turn be
rewritten as the sum of%
\[
T_{n4}^{\left(  1\right)  }\left(  x\right)  :=p^{-1}\frac{n}{\sqrt{k}%
}\overline{H}\left(  xh\right)  \left\{  \left(  \frac{Z_{n-k:n}}{h}\right)
^{-1/\gamma}-1\right\}  ,
\]
plus two remainder terms%
\[
\mathbf{R}_{n8}\left(  x\right)  :=\frac{n}{\sqrt{k}}\int_{xZ_{n-k:n}}%
^{xh}\left\{  \frac{\overline{H}\left(  w\right)  }{\overline{H}^{\left(
1\right)  }\left(  w\right)  }-p^{-1}\right\}  dH\left(  w\right)  ,
\]
and%
\[
\mathbf{R}_{n9}\left(  x\right)  :=p^{-1}\frac{n}{\sqrt{k}}\overline{H}\left(
xh\right)  \left\{  \frac{\overline{H}\left(  xZ_{n-k:n}\right)  }%
{\overline{H}\left(  xh\right)  }-\left(  \frac{Z_{n-k:n}}{h}\right)
^{-1/\gamma}\right\}  .
\]
Note that $\sqrt{k}\left(  Z_{n-k:n}/h-1\right)  $ is asymptotically Gaussian,
it follows that $Z_{n-k:n}/h\overset{\mathbb{P}}{\rightarrow}1$ and $\sqrt
{k}\left(  Z_{n-k:n}/h-1\right)  =O_{\mathbb{P}}\left(  1\right)  $
\citep[see, for instance, Theorem 2.1 (assertion 2.7)  in ][]{BMN-2015}$.$ On
the other hand, from proposition \ref{Potter} (see the Appendix), we infer
that $\dfrac{n}{k}\overline{H}\left(  xh\right)  =\overline{H}\left(
xh\right)  /\overline{H}\left(  h\right)  \leq\left(  1+\epsilon\right)
x^{-1/\gamma+\epsilon}.$ Then, by using the mean value theorem, in
$T_{n4}^{\left(  1\right)  }\left(  x\right)  ,$ yields%
\[
T_{n4}^{\left(  1\right)  }\left(  x\right)  =-\frac{1}{p\gamma}x^{-1/\gamma
}\sqrt{k}\left(  \frac{Z_{n-k:n}}{h}-1\right)  +o_{\mathbb{P}}\left(
x^{-\tau/\gamma}\right)  .
\]
From assertion $\left(  3.22\right)  $ in \cite{BMN-2015}, we have $\sqrt
{k}\left(  Z_{n-k:n}/h-1\right)  =\gamma\beta_{n}^{\ast}\left(  1\right)
+o_{\mathbb{P}}\left(  1\right)  ,$ it follows that
\[
T_{n4}^{\left(  1\right)  }\left(  x\right)  =-\dfrac{1}{p}x^{-1/\gamma}%
\beta_{n}^{\ast}\left(  1\right)  +o_{\mathbb{P}}\left(  x^{-\tau/\gamma
}\right)  =T_{n4}\left(  x\right)  .
\]
To summarize, we showed that
\[
\sqrt{k}\left(  \Delta_{n}\left(  x\right)  -r_{n}\left(  x\right)  \right)
=\sum_{i=1}^{3}\mathbb{T}_{ni}\left(  x\right)  +o_{\mathbb{P}}\left(
x^{-\tau/\gamma}\right)  ,
\]
where%
\begin{equation}
\mathbb{T}_{n1}\left(  x\right)  :=-\frac{1}{p}\int_{x}^{\infty}u^{1/\gamma
}\beta_{n}^{\ast}\left(  u\right)  d\left(  \frac{n}{k}\overline{H}_{n}\left(
Z_{n-k:n}u\right)  \right)  , \label{tn1}%
\end{equation}%
\[
\mathbb{T}_{n2}\left(  x\right)  :=\frac{1}{p^{2}}\int_{x}^{\infty}%
u^{1/\gamma}\beta_{n}\left(  u\right)  d\left(  \frac{n}{k}\overline{H}%
_{n}\left(  Z_{n-k:n}u\right)  \right)  ,
\]
and
\[
\mathbb{T}_{n3}\left(  x\right)  :=\dfrac{1}{p}\left(  \beta_{n}^{\ast}\left(
x\right)  -x^{-1/\gamma}\beta_{n}^{\ast}\left(  1\right)  \right)  .
\]
In the following Section we provide Gaussian approximations to $\mathbb{T}%
_{ni}\left(  x\right)  ,$ $i=1,2,3.$

\subsection{Gaussian approximation to $\mathbb{T}_{ni}\left(  x\right)  $}

We show that%
\begin{equation}
\mathbb{T}_{n1}\left(  x\right)  =-\frac{1}{p}\int_{0}^{x^{-1/\gamma}}%
s^{-1}\sqrt{\frac{n}{k}}B_{n}^{\ast}\left(  \frac{k}{n}s\right)
ds+o_{\mathbb{P}}\left(  x^{-\tau/\gamma}\right)  , \label{approxTn1}%
\end{equation}%
\begin{equation}
\mathbb{T}_{n2}\left(  x\right)  =\frac{1}{p^{2}}\int_{0}^{x^{-1/\gamma}%
}s^{-1}\sqrt{\frac{n}{k}}B_{n}\left(  p\frac{k}{n}s\right)  ds+o_{\mathbb{P}%
}\left(  x^{-\tau/\gamma}\right)  , \label{approxTn2}%
\end{equation}
and%
\begin{equation}
\mathbb{T}_{n3}\left(  x\right)  =\frac{1}{p}\sqrt{\frac{n}{k}}\left\{
B_{n}^{\ast}\left(  \frac{k}{n}x^{-1/\gamma}\right)  -\dfrac{1}{p}%
x^{-1/\gamma}B_{n}^{\ast}\left(  \frac{k}{n}\right)  \right\}  +o_{\mathbb{P}%
}\left(  x^{-\tau/\gamma}\right)  , \label{approxTn3}%
\end{equation}
where $B_{n}^{\ast}$ is the centred Gaussian process given in Theorem
\ref{Theorem1}. We will only give details for $\left(  \ref{approxTn1}\right)
$ since the proofs of $\left(  \ref{approxTn2}\right)  $ and $\left(
\ref{approxTn3}\right)  $ follow by using similar arguments. Let us also
introduce the following Gaussian processes that we define, for $0<\overline
{H}^{\left(  1\right)  }\left(  v\right)  <\theta,$ by%
\begin{equation}
\mathbf{B}_{n}\left(  v\right)  :=B_{n}\left(  \overline{H}^{\left(  1\right)
}\left(  v\right)  \right)  \text{ and }\mathbf{B}_{n}^{\ast}\left(  v\right)
:=B_{n}\left(  \overline{H}^{\left(  1\right)  }\left(  v\right)  \right)
-B_{n}\left(  1-\overline{H}^{\left(  0\right)  }\left(  v\right)  \right)  .
\label{Bn}%
\end{equation}
It is clear that $\mathbb{T}_{n1}\left(  x\right)  ,$ given in $\left(
\ref{tn1}\right)  ,$ may be rewritten as the sum of%
\[
\mathbb{T}_{n1}^{\left(  1\right)  }\left(  x\right)  :=-\frac{1}{p}\int
_{x}^{\infty}u^{1/\gamma}\sqrt{\frac{n}{k}}B_{n}^{\ast}\left(  \frac{k}%
{n}u^{-1/\gamma}\right)  d\left\{  \frac{n}{k}\overline{H}_{n}\left(
Z_{n-k:n}u\right)  \right\}  ,\medskip
\]
and two remainder terms%
\[
\mathbf{R}_{n10}\left(  x\right)  :=-\frac{1}{p}\int_{x}^{\infty}u^{1/\gamma
}\left\{  \beta_{n}^{\ast}\left(  u\right)  -\sqrt{\frac{n}{k}}\mathbf{B}%
_{n}^{\ast}\left(  Z_{n-k:n}u\right)  \right\}  d\left\{  \frac{n}{k}%
\overline{H}_{n}\left(  Z_{n-k:n}u\right)  \right\}  ,\medskip
\]
and%
\[
\mathbf{R}_{n11}\left(  x\right)  :=-\frac{1}{p}\int_{x}^{\infty}u^{1/\gamma
}\sqrt{\frac{n}{k}}\left\{  \mathbf{B}_{n}^{\ast}\left(  Z_{n-k:n}u\right)
-B_{n}^{\ast}\left(  \frac{k}{n}u^{-1/\gamma}\right)  \right\}  d\left\{
\frac{n}{k}\overline{H}_{n}\left(  Z_{n-k:n}u\right)  \right\}  .\medskip
\]
Let us now focus on the term $\mathbb{T}_{n1}^{\left(  1\right)  }\left(
x\right)  .$ Since $x\geq1$ then $H_{n}\left(  w\right)  =1,$ for $w\geq
xZ_{n:n},$ it follows that $H_{n}\left(  Z_{n-k:n}u\right)  =1,$ for $u\geq
xZ_{n:n}/Z_{n-k:n},$ therefore%
\[
\mathbb{T}_{n1}^{\left(  1\right)  }\left(  x\right)  =-\frac{1}{p}\int
_{x}^{xZ_{n:n}/Z_{n-k:n}}u^{1/\gamma}\sqrt{\frac{n}{k}}B_{n}^{\ast}\left(
\frac{k}{n}u^{-1/\gamma}\right)  d\left\{  \frac{n}{k}\overline{H}_{n}\left(
Z_{n-k:n}u\right)  \right\}  .\medskip
\]
Let $Q_{n}\left(  t\right)  :=\inf\left\{  w:H_{n}\left(  w\right)  \geq
t\right\}  ,$ $0<t<1,$ be the empirical quantile function pertaining to cdf
$H_{n}.$ For convenience, we set%
\begin{equation}
\ell_{n}\left(  s\right)  :=Q_{n}\left(  1-ks/n\right)  /Z_{n-k:n},
\label{ln(s)}%
\end{equation}
and make the change of variable $u=\ell_{n}\left(  s\right)  $ to get%
\[
\mathbb{T}_{n1}^{\left(  1\right)  }\left(  x\right)  =\frac{1}{p}\int
_{\frac{n}{k}\overline{H}_{n}\left(  Z_{n:n}x\right)  }^{\frac{n}{k}%
\overline{H}_{n}\left(  Z_{n-k:n}x\right)  }\left(  \ell_{n}\left(  s\right)
\right)  ^{1/\gamma}\sqrt{\frac{n}{k}}B_{n}^{\ast}\left(  \frac{k}{n}\left(
\ell_{n}\left(  s\right)  \right)  ^{-1/\gamma}\right)  ds.\medskip
\]
In view of the algebraic equation $a_{n}b_{n}=\left(  a_{n}-a\right)  \left(
b_{n}-b\right)  +\left(  a_{n}-a\right)  b+a\left(  b_{n}-b\right)  +ab,$ we
decompose $\mathbb{T}_{n1}^{\left(  1\right)  }\left(  x\right)  $ into the
sum of%
\[
\mathbb{T}_{n1}^{\left(  2\right)  }\left(  x\right)  :=\frac{1}{p}\int
_{\frac{n}{k}\overline{H}_{n}\left(  Z_{n:n}x\right)  }^{\frac{n}{k}%
\overline{H}_{n}\left(  Z_{n-k:n}x\right)  }s^{-1}\sqrt{\frac{n}{k}}%
B_{n}^{\ast}\left(  \frac{k}{n}s\right)  ds,\medskip
\]
and three remainder terms%
\[
\mathbf{R}_{n12}\left(  x\right)  :=\frac{1}{p}\int_{\frac{n}{k}\overline
{H}_{n}\left(  Z_{n-k:n}x\right)  }^{\frac{n}{k}\overline{H}_{n}\left(
Z_{n:n}x\right)  }\left(  \left(  \ell_{n}\left(  s\right)  \right)
^{1/\gamma}-s^{-1}\right)  \sqrt{\frac{n}{k}}\left\{  B_{n}^{\ast}\left(
\frac{k}{n}\left(  \ell_{n}\left(  s\right)  \right)  ^{-1/\gamma}\right)
-B_{n}^{\ast}\left(  \frac{k}{n}s\right)  \right\}  ds,\medskip
\]%
\[
\mathbf{R}_{n13}\left(  x\right)  :=\frac{1}{p}\int_{\frac{n}{k}\overline
{H}_{n}\left(  Z_{n-k:n}x\right)  }^{\frac{n}{k}\overline{H}_{n}\left(
Z_{n:n}x\right)  }s^{-1}\sqrt{\frac{n}{k}}\left\{  B_{n}^{\ast}\left(
\frac{k}{n}\left(  \ell_{n}\left(  s\right)  \right)  ^{-1/\gamma}\right)
-B_{n}^{\ast}\left(  \frac{k}{n}s\right)  \right\}  ds,\medskip
\]
and%
\[
\mathbf{R}_{n14}\left(  x\right)  :=\frac{1}{p}\int_{\frac{n}{k}\overline
{H}_{n}\left(  Z_{n-k:n}x\right)  }^{\frac{n}{k}\overline{H}_{n}\left(
Z_{n:n}x\right)  }\left(  \left(  \ell_{n}\left(  s\right)  \right)
^{1/\gamma}-s^{-1}\right)  \sqrt{\frac{n}{k}}B_{n}^{\ast}\left(  \frac{k}%
{n}s\right)  ds.\medskip
\]
Let us also decompose $\mathbb{T}_{n1}^{\left(  2\right)  }\left(  x\right)  $
into the sum of
\[
\mathbb{T}_{n1}^{\left(  3\right)  }\left(  x\right)  :=\frac{1}{p}\int
_{0}^{x^{-1/\gamma}}s^{-1}\sqrt{\frac{n}{k}}B_{n}^{\ast}\left(  \frac{k}%
{n}s\right)  ds,\medskip
\]%
\[
\mathbf{R}_{n15}\left(  x\right)  :=-\frac{1}{p}\int_{x^{-1/\gamma}}^{\frac
{n}{k}\overline{H}_{n}\left(  Z_{n-k:n}x\right)  }s^{-1}\sqrt{\frac{n}{k}%
}B_{n}^{\ast}\left(  \frac{k}{n}s\right)  ds,\medskip
\]
and%
\[
\mathbf{R}_{n16}\left(  x\right)  :=-\frac{1}{p}\int_{0}^{\frac{n}{k}%
\overline{H}_{n}\left(  Z_{n:n}x\right)  }s^{-1}\sqrt{\frac{n}{k}}B_{n}^{\ast
}\left(  \frac{k}{n}s\right)  ds.\medskip
\]
Finally, we end up
\[
\mathbb{T}_{n1}\left(  x\right)  =-\dfrac{1}{p}\int_{0}^{x^{-1/\gamma}}%
s^{-1}\sqrt{\dfrac{n}{k}}B_{n}^{\ast}\left(  \dfrac{k}{n}s\right)
ds+o_{\mathbb{P}}\left(  x^{-\tau/\gamma}\right)  ,
\]
which meets with approximation $\left(  \ref{approxTn1}\right)  .$\ Recall
that in Theorem \ref{Theorem1}$,$ we set%
\begin{align*}
p^{-1}\sqrt{\dfrac{n}{k}}\mathcal{L}_{n}\left(  x^{-1/\gamma}\right)   &
=p^{-1}\sqrt{\frac{n}{k}}\left\{  B_{n}^{\ast}\left(  \frac{k}{n}x^{-1/\gamma
}\right)  -\dfrac{1}{p}x^{-1/\gamma}B_{n}^{\ast}\left(  \frac{k}{n}\right)
\right\} \\
&  -p^{-1}\int_{0}^{x^{-1/\gamma}}s^{-1}\sqrt{\frac{n}{k}}B_{n}^{\ast}\left(
\frac{k}{n}s\right)  ds+p^{-2}\int_{0}^{x^{-1/\gamma}}s^{-1}\sqrt{\frac{n}{k}%
}B_{n}\left(  \frac{k}{n}s\right)  ds,
\end{align*}
this means that $p^{-1}\sqrt{\dfrac{n}{k}}\mathcal{L}_{n}\left(  x^{-1/\gamma
}\right)  -\sum_{i=1}^{3}\mathbb{T}_{ni}\left(  x\right)  =o_{\mathbb{P}%
}\left(  x^{-\tau/\gamma}\right)  ,$ therefore
\[
\sqrt{k}\left(  \Delta_{n}\left(  x\right)  -r_{n}\left(  x\right)  \right)
-p^{-1}\sqrt{\dfrac{n}{k}}\mathcal{L}_{n}\left(  x^{-1/\gamma}\right)
=o_{\mathbb{P}}\left(  x^{\tau/\gamma}\right)  ,
\]
uniformly over $x\geq1,$ this completes the proof of weak approximation
$\left(  \ref{TP-C1}\right)  .$

\subsection{Gaussian approximation to $D_{n}\left(  x\right)  $}

Recall that $D_{n}\left(  x\right)  =\sqrt{k}\left(  \Delta_{n}\left(
x\right)  -px^{-1/\gamma}\right)  ,$ $x\geq1,$ which may be rewritten into the
sum of
\[
I_{n}\left(  x\right)  :=\sqrt{k}\left(  \Delta_{n}\left(  x\right)
-r_{n}\left(  x\right)  \right)  \text{ and }J_{n}\left(  x\right)  :=\sqrt
{k}\left(  r_{n}\left(  x\right)  -px^{-1/\gamma}\right)  .
\]
In view of weak approximation $\left(  \ref{TP-C1}\right)  ,$ it suffices to
show that $J_{n}\left(  x\right)  $ meets the remainder term $\mathcal{R}%
_{n}\left(  x\right)  .$ Indeed, by letting $\psi\left(  z\right)
:=\overline{H}\left(  z\right)  /\overline{H}^{\left(  1\right)  }\left(
z\right)  ,$ we write
\[
k^{-1/2}J_{n}\left(  x\right)  =-\psi\left(  h\right)  \int_{x}^{\infty}%
\frac{\psi\left(  hz\right)  }{\psi\left(  h\right)  }\frac{d\overline
{H}\left(  hz\right)  }{\overline{H}\left(  h\right)  }-px^{-1/\gamma},
\]
which in turn may be decomposed into the sum of%
\[
k^{-1/2}J_{1n}\left(  x\right)  :=p^{-1}\left(  \frac{\overline{H}\left(
hx\right)  }{\overline{H}\left(  h\right)  }-x^{-1/\gamma}\right)  ,
\]%
\[
k^{-1/2}J_{2n}\left(  x\right)  :=-\psi\left(  z\right)  \int_{x}^{\infty
}\left(  \frac{\psi\left(  hz\right)  }{\psi\left(  h\right)  }-1\right)
d\frac{\overline{H}\left(  hz\right)  }{\overline{H}\left(  h\right)  }%
\]
and
\[
k^{-1/2}J_{3n}\left(  x\right)  :=-\psi\left(  z\right)  A_{2}\left(
h\right)  \dfrac{\overline{H}\left(  hx\right)  }{\overline{H}\left(
h\right)  },
\]
where $A_{2}:=\overline{H}^{\left(  1\right)  }/\overline{H}-p.$ By using the
uniform inequalities (for the second-order regularly varying functions) to
$\overline{H}$ \citep[see, e.g., ][bottom of page 161]{deHF06} we write: for
any $0<\epsilon<1,$ there exists $n_{0}=n_{0}\left(  \epsilon\right)  ,$ such
that for all $n>n_{0}$ and $x\geq1$%
\[
\left\vert \frac{\overline{H}\left(  hx\right)  /\overline{H}\left(  h\right)
-x^{-1/\gamma}}{A\left(  h\right)  }-x^{-1/\gamma}\frac{x^{\nu/\gamma}-1}%
{\nu\gamma}\right\vert <\epsilon x^{-1/\gamma+\nu/\gamma+\epsilon}.
\]
Recall that, by assumption, we have $\sqrt{k}A\left(  h\right)  =O\left(
1\right)  ,$ it follows that
\[
J_{1n}\left(  x\right)  =p^{-1}x^{-1/\gamma}\dfrac{x^{\nu/\gamma}-1}{\nu
\gamma}\sqrt{k}A\left(  h\right)  +o\left(  x^{-\tau/\gamma}\right)  .
\]
For the second and third terms, let us write
\[
J_{2n}\left(  x\right)  =\sqrt{k}A_{\ast}\left(  h\right)  \psi\left(
h\right)  \int_{x}^{\infty}\dfrac{1-\psi\left(  hz\right)  /\psi\left(
h\right)  }{A_{\ast}\left(  h\right)  }d\dfrac{\overline{H}\left(  hz\right)
}{\overline{H}\left(  h\right)  }.
\]
Making use of Lemma \ref{Lemma-2} and Proposition \ref{Potter} (applied to
$\overline{H})$ with the fact that $\psi\left(  h\right)  \rightarrow p^{-1}$
and $\sqrt{k}A_{\ast}\left(  h\right)  =O\left(  1\right)  ,$ we end up, after
integration, with%
\[
J_{2n}\left(  x\right)  =\frac{\nu+x^{\nu_{\ast}/\gamma}-1}{p\gamma
x^{1/\gamma}\nu_{\ast}\left(  1-\nu_{\ast}\right)  }\sqrt{k}A_{\ast}\left(
h\right)  +o\left(  x^{-\tau/\gamma}\right)  .
\]
For the third term, we also assumed that$\sqrt{k}A_{2}\left(  h\right)
=O\left(  1\right)  ,$ then it suffices to use on again Proposition
\ref{Potter} to readily get $J_{3n}\left(  x\right)  :=x^{1/\gamma-1}\sqrt
{k}A_{2}\left(  h\right)  +o\left(  x^{-\tau/\gamma}\right)  ,$ as sought.

\subsection{Proof of Theorem \ref{Theorem2}}

Let us begin by the consistency of $\widehat{\gamma}_{1}$ which may be
rewritten into $\int_{1}^{\infty}x^{-1}\Delta_{n}\left(  x\right)  dx$\ and
write
\[
\widehat{\gamma}_{1}=\int_{1}^{\infty}x^{-1}r_{n}\left(  x\right)
dx+k^{-1/2}p^{-1}\sqrt{\dfrac{n}{k}}\int_{1}^{\infty}x^{-1}\mathcal{L}%
_{n}\left(  x^{-1/\gamma}\right)  +\epsilon_{n},
\]
where
\[
\epsilon_{n}:=\int_{1}^{\infty}x^{-1}\left\{  \Delta_{n}\left(  x\right)
-r_{n}\left(  x\right)  -k^{-1/2}p^{-1}\sqrt{\dfrac{n}{k}}\mathcal{L}%
_{n}\left(  x^{-1/\gamma}\right)  \right\}  dx.
\]
It is clear that for a fixed $0\leq\tau<1/8,$ we have
\[
\left\vert \epsilon_{n}\right\vert \leq\sup_{x\geq1}x^{\tau/\gamma}\left\vert
\Delta_{n}\left(  x\right)  -r_{n}\left(  x\right)  -k^{-1/2}p^{-1}%
\sqrt{\dfrac{n}{k}}\mathcal{L}_{n}\left(  x^{-1/\gamma}\right)  \right\vert
\int_{1}^{\infty}x^{-1-\tau/\gamma}dx.
\]
Since $\int_{1}^{\infty}x^{-1-\tau/\gamma}dx=\gamma/\tau$ (finite), then by
using weak approximation $\left(  \ref{TP-C1}\right)  $ given in Theorem
\ref{Theorem1}, we end up with $\epsilon_{n}\overset{\mathbb{P}}{\rightarrow
}0.$ On the other hand $p^{-1}\sqrt{\dfrac{n}{k}}\int_{1}^{\infty}%
x^{-1}\mathcal{L}_{n}\left(  x^{-1/\gamma}\right)  $ is a centred Gaussian
random with variance
\[
\sigma_{n}^{2}:=\mathbf{E}\left[  p^{-1}\sqrt{\dfrac{n}{k}}\int_{1}^{\infty
}x^{-1}\mathcal{L}_{n}\left(  x^{-1/\gamma}\right)  \right]  ^{2}.
\]
By a tedious (but elementary) computation we show that $\sigma_{n}%
^{2}\rightarrow\left(  9p^{-1}-8\right)  \gamma_{1}^{2}.$ A similar
calculation may be found in the proof of Corollary 2.1 of \cite{BMN-2015}. It
remains to show that $\int_{1}^{\infty}x^{-1}r_{n}\left(  x\right)
dx\rightarrow\gamma_{1}$ as $n\rightarrow\infty.$ Indeed, we have
\[
\int_{1}^{\infty}x^{-1}r_{n}\left(  x\right)  dx=\frac{1}{\overline{H}\left(
h\right)  }\int_{1}^{\infty}\left\{  \int_{x}^{\infty}\frac{\overline
{H}\left(  hw\right)  dH\left(  hw\right)  }{\overline{H}^{\left(  1\right)
}\left(  hw\right)  +k^{-1}\overline{H}\left(  hw\right)  }\right\}  d\log x.
\]
By an integration by parts, this latter becomes%
\[
\frac{1}{\overline{H}\left(  h\right)  }\int_{1}^{\infty}\frac{\overline
{H}\left(  hx\right)  dH\left(  hx\right)  }{\overline{H}^{\left(  1\right)
}\left(  hx\right)  +k^{-1}\overline{H}\left(  hx\right)  }\left(  \log
x\right)  dx,
\]
which, from Lemma \ref{Lemma-1}, tends to $\gamma_{1}$ as $n\rightarrow
\infty,$ as sought. Let us now consider the asymptotic normality. Recall that
$\sqrt{k}\left(  \widehat{\gamma}_{1}-\gamma_{1}\right)  =\int_{1}^{\infty
}x^{-1}D_{n}\left(  x\right)  dx$ and use the weak approximation $\left(
\ref{TP-C2}\right)  $ to get
\[
\sqrt{k}\left(  \widehat{\gamma}_{1}-\gamma_{1}\right)  =p^{-1}\sqrt{\dfrac
{n}{k}}\int_{1}^{\infty}x^{-1}\mathcal{L}_{n}\left(  x^{-1/\gamma}\right)
+p^{-1}\mu_{n}+o_{\mathbb{P}}\left(  1\right)  ,
\]
where $\mu_{n}:=\int_{1}^{\infty}x^{-1}\mathcal{R}_{n}\left(  x\right)  dx.$
By using integrations by parts with elementary calculations, we show that
$\mu_{n}$ meets formula $\left(  \ref{mu-n}\right)  $ and%
\begin{align*}
\sqrt{\dfrac{n}{k}}\int_{1}^{\infty}x^{-1}\mathcal{L}_{n}\left(  x^{-1/\gamma
}\right)   &  =\sqrt{\dfrac{n}{k}}\gamma_{1}%
%TCIMACRO{\dint _{0}^{1}}%
%BeginExpansion
{\displaystyle\int_{0}^{1}}
%EndExpansion
s^{-1}B_{n}^{\ast}\left(  \dfrac{k}{n}s\right)  ds-p^{-1}\sqrt{\dfrac{n}{k}%
}B_{n}^{\ast}\left(  \dfrac{k}{n}\right) \\
&  +\sqrt{\dfrac{n}{k}}\gamma_{1}%
%TCIMACRO{\dint _{0}^{1}}%
%BeginExpansion
{\displaystyle\int_{0}^{1}}
%EndExpansion
s^{-1}\left\{  B_{n}^{\ast}\left(  \dfrac{k}{n}s\right)  -p^{-1}B_{n}\left(
p\dfrac{k}{n}s\right)  \right\}  \log\left(  s\right)  ds,
\end{align*}
with variance tending to $\left(  9p^{-1}-8\right)  \gamma_{1}^{2},$ which
leads to the Gaussian approximation and therefore the asymptotic normality of
$\sqrt{k}\left(  \widehat{\gamma}_{1}-\gamma_{1}\right)  .$

\section{\textbf{Concluding notes}}

\noindent In this work, we first defined a tail empirical process for
Pareto-like distributions which are randomly right-censored and established
its Gaussian approximation. The latter will play a central role, in the
context of right censorship, in determining the asymptotic distributions of
statistics that are functionals of the tail index estimator such as the
estimators of large quantiles, risk measures, second-order parameters of
regular variation... Then, we introduced a new Hill-type estimator for
positive EVI of right-censored heavy-tailed data whose asymptotic behavior
(consistency and asymptotic normality) is assessed by using the
above-mentioned tail process. Compared to other existing estimators, the new
tail index estimator performs better, as far as bias and mean squared error
are concerned, at least from a simulation viewpoint. As a case study, we
provided an estimator to the survival time of Australian male Aids patients.
It is also worth mentioning, that assumption $\overline{H}\in\mathcal{RV}%
_{\left(  -1/\gamma\right)  }$ also implies that
\[
\Gamma_{t}\left(  g,\alpha\right)  :=\frac{\int_{t}^{\infty}g\left(
\dfrac{\overline{H}\left(  x\right)  }{\overline{H}\left(  t-\right)
}\right)  \left(  \log\left(  x/t\right)  \right)  ^{\alpha}\dfrac{dH\left(
x\right)  }{\overline{H}\left(  t\right)  }}{\int_{0}^{1}g\left(  x\right)
\left(  -\log x\right)  ^{\alpha}dx}\rightarrow\gamma^{\alpha},\text{ as
}t\rightarrow\infty,
\]
where $g$ is a suitable weight function and $\alpha$ is some positive real
number. Thus, assertion $\left(  \ref{hill-theo}\right)  $ becomes a special
case of the limit above. Thereby, the functional $\Gamma_{t}\left(
g,\alpha\right)  $ can be considered as a basic tool to constructing a whole
class of estimators for distribution tail parameters for complete data, see
for instance \cite{Mercadier-10}. Recently, \cite{BchMN-16a} and \cite{HNB-17}
benefited from this result to derive, respectively, a kernel estimator of tail
index and a second-order parameter estimator for random right-truncation data.
As we deal with the tail modeling of underlying cdf $F,$ we have to work with
the limit
\[
\Gamma_{t_{n}}^{\left(  c\right)  }\left(  g,\alpha\right)  :=\frac{\int
_{t}^{\infty}g\left(  \dfrac{\overline{H}\left(  x\right)  }{\overline
{H}\left(  t_{n}-\right)  }\right)  \left(  \dfrac{\log\left(  x/t_{n}\right)
}{\overline{H}^{\left(  1\right)  }\left(  z\right)  /\overline{H}\left(
z\right)  +k^{-1}}\right)  ^{\alpha}\dfrac{dH\left(  x\right)  }{\overline
{H}\left(  t_{n}\right)  }}{\int_{0}^{1}g\left(  x\right)  \left(  -\log
x\right)  ^{\alpha}dx}\rightarrow\gamma_{1}^{\alpha},\text{ as }%
n\rightarrow\infty,
\]
which is generalization of the result $\left(  \ref{lim-princ}\right)  .$ This
latter may be readily shown by using similar arguments as those used in Lemma
\ref{Lemma-2}. Thereby, by letting $t=Z_{n-k:n},$ then by replacing
$\overline{H}^{\left(  1\right)  }$ and $\overline{H}$ by their respective
empirical counterparts $\overline{H}_{n}^{\left(  1\right)  }$ and
$\overline{H}_{n},$ we end up with an estimator of $\Gamma_{t}^{\left(
c\right)  }\left(  g,\alpha\right)  $ given by
\[
\Gamma_{n,k}^{\left(  c\right)  }\left(  g,\alpha\right)  :=\frac
{\sum\limits_{i=1}^{k-1}a_{n,k}^{\left(  i\right)  }\left(  \log\left(
Z_{n-i:n}/Z_{n-k:n}\right)  \right)  ^{\alpha}}{\int_{0}^{1}g\left(  x\right)
\left(  -\log x\right)  ^{\alpha}dx},
\]
where $a_{n,k}^{\left(  i\right)  }:=\dfrac{i}{k}g\left(  \dfrac{i}%
{k+1}\right)  /\left(  \sum_{j=1}^{i}\delta_{\left[  n-j+1:n\right]
}+i/k\right)  .$ This would have fruitful consequences on the statistical
analysis of extremes under random censoring. To finish, notice that the
Gaussian approximations corresponding to $\widehat{p},$ $\widehat{\gamma}_{1}$
and $D_{n}\left(  x\right)  $ are jointly established in terms of the same
sequence of Brownian bridges $B_{n}.$ This allows to establish the limit
distributions to the statistics of Kolmogorov-Smirnov and Cramer-von Mises
type respectively defined by%
\[
KS_{n,k}:=\sqrt{k}\sup_{x\geq1}\left\vert \Delta_{n}\left(  x\right)
-\widehat{p}^{-1}x^{-1/\widehat{\gamma}}\right\vert ,
\]
and%
\[
CV_{n,k}:=\left(  \widehat{p}\widehat{\gamma}\right)  ^{-1}k\int_{1}^{\infty
}x^{-1/\widehat{\gamma}-1}\left(  \Delta_{n}\left(  x\right)  -\widehat
{p}^{-1}x^{-1/\widehat{\gamma}}\right)  ^{2}dx.
\]
These statistics provide goodness-of-fit tests for Pareto-like distributions
under right-censorship. This matter will be addressed in our future work. In
the case of complete data $\left(  p=1\right)  ,$ the latter two statistics
respectively become%
\[
\sqrt{k}\sup_{x\geq1}\left\vert \dfrac{n}{k}\overline{F}_{n}\left(
xX_{n-k:n}\right)  -x^{-1/\widehat{\gamma}}\right\vert \text{ and }\left(
\widehat{\gamma}\right)  ^{-1}k\int_{1}^{\infty}x^{-1/\widehat{\gamma}%
-1}\left(  \dfrac{n}{k}\overline{F}_{n}\left(  xX_{n-k:n}\right)
-x^{-1/\widehat{\gamma}}\right)  ^{2}dx.
\]
They are used in \cite{KP-2008}, amongst others, to test the heaviness of cdf's.

\section{\textbf{Appendix}}

\noindent A key result related to the regular variation concept, namely
Potter-type inequalities
\citep[see, e.g., Proposition B.1.10, page 369  in][]{deHF06}, will be applied
quite frequently. For this reason, we need to recall this very useful tool here.

\begin{proposition}
\label{Potter}Let $f$ $\in\mathcal{RV}_{\left(  \alpha\right)  },$ for
$\alpha\in\mathbb{R}.$ Then, for any sufficiently small $\epsilon>0,$ there
exists $t_{0}=t_{0}\left(  \epsilon\right)  >0,$ such that for any $t\geq
t_{0},$ we have $\left\vert f\left(  tx\right)  /f\left(  t\right)
-x^{\alpha}\right\vert \leq\epsilon x^{\alpha+\epsilon},$ uniformly on
$x\geq1.$
\end{proposition}

\begin{lemma}
\label{Lemma-1}Let $\overline{F}\in\mathcal{RV}_{\left(  -1/\gamma_{1}\right)
}$ and $\overline{G}\in\mathcal{RV}_{\left(  -1/\gamma_{2}\right)  }.$ Then,
for every real $r\geq0,$ we have%
\begin{equation}
\lim_{t\rightarrow\infty}\frac{1}{\overline{H}\left(  t\right)  }\int
_{t}^{\infty}\left(  \frac{\overline{H}\left(  z\right)  }{\overline
{H}^{\left(  1\right)  }\left(  z\right)  }\right)  ^{r}\log\left(
z/t\right)  dH\left(  z\right)  =\frac{\gamma}{p^{r}}. \label{as-lemma1}%
\end{equation}
Moreover, for a given sequence $t=t_{n}\rightarrow\infty,$ as $n\rightarrow
\infty,$ we have
\begin{equation}
\lim_{n\rightarrow\infty}\frac{1}{\overline{H}\left(  t\right)  }\int
_{t}^{\infty}\frac{\log\left(  z/t\right)  dH\left(  z\right)  }{\overline
{H}^{\left(  1\right)  }\left(  z\right)  /\overline{H}\left(  z\right)
+k^{-1}}=\gamma_{1}. \label{as-lemma2}%
\end{equation}

\end{lemma}

\begin{proof}
For $r\geq0,$ let us set%
\[
I_{r}\left(  t\right)  :=\frac{1}{\overline{H}\left(  t\right)  }\int
_{t}^{\infty}\left(  \frac{\overline{H}\left(  z\right)  }{\overline
{H}^{\left(  1\right)  }\left(  z\right)  }\right)  ^{r}\log\left(
z/t\right)  dH\left(  z\right)  ,
\]
which, by a change of variables, becomes
\[
\frac{1}{\overline{H}\left(  t\right)  }\int_{1}^{\infty}\left(
\frac{\overline{H}\left(  tz\right)  }{\overline{H}^{\left(  1\right)
}\left(  tz\right)  }\right)  ^{r}\log\left(  z\right)  dH\left(  tz\right)
.
\]
This may be decomposed into the sum of
\[
I_{1}\left(  t\right)  :=\frac{1}{\overline{H}\left(  t\right)  }\int
_{1}^{\infty}\left\{  \left(  \frac{\overline{H}\left(  tz\right)  }%
{\overline{H}^{\left(  1\right)  }\left(  tz\right)  }\right)  ^{r}%
-p^{-r}\right\}  \log\left(  z\right)  dH\left(  tz\right)  ,
\]
and $\ I_{2}\left(  t\right)  :=p^{-r}\int_{1}^{\infty}\log\left(  z\right)
dH\left(  tz\right)  /\overline{H}\left(  t\right)  .$ We will show that
$I_{1}\left(  t\right)  \rightarrow0$ and $I_{2}\left(  t\right)
\rightarrow\gamma/p^{r},$ as $t\rightarrow\infty.$ Indeed, we have%
\[
\left\vert I_{1}\left(  t\right)  \right\vert \leq\frac{1}{\overline{H}\left(
t\right)  }\int_{1}^{\infty}\left\vert \left(  \frac{\overline{H}\left(
tz\right)  }{\overline{H}^{\left(  1\right)  }\left(  tz\right)  }\right)
^{r}-\frac{1}{p^{r}}\right\vert \log\left(  z\right)  dH\left(  tz\right)  .
\]
Recall that from Lemma 4.1 of \cite{BMN-2015}, under the first-order
conditions $\left(  \ref{rv}\right)  ,$ we have $\overline{H}\left(  z\right)
/\overline{H}^{\left(  1\right)  }\left(  z\right)  \rightarrow p^{-1},$ as
$z\rightarrow\infty,$ it follows that $\overline{H}/\overline{H}^{\left(
1\right)  }\in\mathcal{RV}_{\left(  0\right)  }.$ Thus, by applying
Proposition $\left(  \ref{Potter}\right)  $ to $\overline{H}/\overline
{H}^{\left(  1\right)  },$\ we write : for all large $t$ and any sufficiently
small $\epsilon>0$
\begin{equation}
\text{ }\left\vert \frac{\overline{H}\left(  tz\right)  /\overline{H}^{\left(
1\right)  }\left(  tz\right)  }{\overline{H}\left(  t\right)  /\overline
{H}^{\left(  1\right)  }\left(  t\right)  }-1\right\vert <\epsilon,\text{ for
any }z\geq1, \label{pot}%
\end{equation}
leading to%
\[
\left\vert \frac{\overline{H}\left(  tz\right)  }{\overline{H}^{\left(
1\right)  }\left(  tz\right)  }-\frac{\overline{H}\left(  t\right)
}{\overline{H}^{\left(  1\right)  }\left(  t\right)  }\right\vert
<\epsilon\left(  1+\epsilon\right)  p^{-1}z^{\epsilon}.
\]
Observe now that%
\begin{equation}
\left\vert \frac{\overline{H}\left(  tz\right)  }{\overline{H}^{\left(
1\right)  }\left(  tz\right)  }-p^{-1}\right\vert <\left\vert \frac
{\overline{H}\left(  tz\right)  }{\overline{H}^{\left(  1\right)  }\left(
tz\right)  }-\frac{\overline{H}\left(  t\right)  }{\overline{H}^{\left(
1\right)  }\left(  t\right)  }\right\vert +\left\vert \frac{\overline
{H}\left(  t\right)  }{\overline{H}^{\left(  1\right)  }\left(  t\right)
}-p^{-1}\right\vert , \label{inequa-p}%
\end{equation}
it follows that $\overline{H}\left(  tz\right)  /\overline{H}^{\left(
1\right)  }\left(  tz\right)  -p^{-1}=o\left(  z^{\epsilon}\right)  ,$
uniformly over $z\geq1.$ Using the mean-value theorem yields that $\left(
\overline{H}\left(  tz\right)  /\overline{H}^{\left(  1\right)  }\left(
tz\right)  \right)  ^{r}-p^{-r}=o\left(  z^{\epsilon}\right)  $ as well, which
implies that for all large $t,$ $I_{1}\left(  t\right)  =o\left(  1\right)
\int_{1}^{\infty}z^{\epsilon}\log\left(  z\right)  d\left\{  H\left(
tz\right)  /\overline{H}\left(  t\right)  \right\}  .$ Then, it suffices to
show that the latter integral is finite. Indeed, let $J_{M}:=\int_{1}%
^{M}z^{\epsilon}\log\left(  z\right)  d\left\{  H\left(  tz\right)
/\overline{H}\left(  t\right)  \right\}  ,$ which, by an integration by parts,
equals
\[
-M^{\epsilon}\left\{  \overline{H}\left(  tM\right)  /\overline{H}\left(
t\right)  \right\}  \log M+\int_{1}^{M}\left\{  \overline{H}\left(  tz\right)
/\overline{H}\left(  t\right)  \right\}  d\left(  z^{\epsilon}\log z\right)
.
\]
Once again, we use Proposition $\left(  \ref{Potter}\right)  $ to
$\overline{H},$ to write that for any $z\geq1,$
\begin{equation}
\left\vert \overline{H}\left(  tz\right)  /\overline{H}\left(  t\right)
-z^{-1/\gamma}\right\vert <\epsilon z^{-1/\gamma+\epsilon\text{ }},
\label{inequa-H}%
\end{equation}
for sufficiently large $t.$ It follows that
\[
M^{\epsilon}\left\{  \overline{H}\left(  tM\right)  /\overline{H}\left(
t\right)  \right\}  \log M=M^{\epsilon}\left(  1+o\left(  M^{\epsilon}\right)
\right)  M^{-1/\gamma}\log M,
\]
which tends to zero as $M\rightarrow\infty.$ By using similar arguments, we
end up with
\[
\int_{1}^{M}\left\{  \overline{H}\left(  tz\right)  /\overline{H}\left(
t\right)  \right\}  d\left(  z^{\epsilon}\log z\right)  =\int_{1}^{M}\left(
1+o\left(  z^{\epsilon}\right)  \right)  z^{-1/\gamma}d\left(  z^{\epsilon
}\log z\right)  ,
\]
which tends to $\gamma\left(  \gamma\epsilon-1\right)  /\left(  2\gamma
\epsilon-1\right)  ^{2}\left(  1+o\left(  1\right)  \right)  $ as
$M\rightarrow\infty,$ hence $I_{1}\left(  t\right)  =o\left(  1\right)  .$ For
the second term $I_{2}\left(  t\right)  ,$ it suffices to use $\left(
\ref{hill-theo}\right)  $ to have $p^{r}I_{2}\left(  t\right)  =\int
_{1}^{\infty}\log\left(  z\right)  dH\left(  tz\right)  /\overline{H}\left(
t\right)  \rightarrow\gamma$ as $t\rightarrow\infty,$ which meets $\left(
\ref{as-lemma1}\right)  .$ To prove $\left(  \ref{as-lemma2}\right)  ,$ we
first apply Taylor's expansion. Indeed, since $1/k\rightarrow0$ as
$n\rightarrow\infty,$ then%
\[
\frac{1}{\overline{H}^{\left(  1\right)  }\left(  z\right)  /\overline
{H}\left(  z\right)  +1/k}=\frac{\overline{H}\left(  z\right)  }{\overline
{H}^{\left(  1\right)  }\left(  z\right)  }+o\left(  \left(  \frac
{\overline{H}\left(  z\right)  }{\overline{H}^{\left(  1\right)  }\left(
z\right)  }\right)  ^{2}\right)  ,\text{ as }n\rightarrow\infty,
\]
thereby, we use result $\left(  \ref{as-lemma1}\right)  $ twice, for $r=1$ and
$r=2,$ which completes the proof of the lemma.
\end{proof}

\begin{lemma}
\label{Lemma-2}Let $\psi:=\overline{H}/\overline{H}^{\left(  1\right)  }.$
Then, for any small $\epsilon>0$ and $x\geq1,$ we have%
\begin{equation}
\left\vert \frac{\psi\left(  tx\right)  /\psi\left(  t\right)  -1}{A_{\ast
}\left(  t\right)  }-\dfrac{x^{\nu_{\ast}/\gamma}-1}{\nu_{\ast}\gamma
}\right\vert <\epsilon x^{\nu_{\ast}/\gamma+\epsilon},\text{ as }%
t\rightarrow\infty, \label{pshi}%
\end{equation}
where $\nu_{\ast}:=\max\left(  \nu,\nu_{1}\right)  $ and $A_{\ast}\left(
t\right)  :=\mathbf{1}\left\{  \nu>\nu_{1}\right\}  A\left(  t\right)
-\mathbf{1}\left\{  \nu\leq\nu_{1}\right\}  A_{1}\left(  t\right)  .$
\end{lemma}

\begin{proof}
Let us write%
\[
\frac{\psi\left(  tx\right)  }{\psi\left(  t\right)  }=\left(  \frac
{\overline{H}\left(  tx\right)  }{\overline{H}\left(  t\right)  }%
-x^{-1/\gamma}\right)  \frac{\overline{H}^{\left(  1\right)  }\left(
t\right)  }{\overline{H}^{\left(  1\right)  }\left(  tx\right)  }%
+x^{-1/\gamma}\left(  \frac{\overline{H}^{\left(  1\right)  }\left(  t\right)
}{\overline{H}^{\left(  1\right)  }\left(  tx\right)  }-x^{1/\gamma}\right)
+1.
\]
It is easy to check that%
\begin{align*}
\frac{\psi\left(  tx\right)  /\psi\left(  t\right)  -1}{A_{\ast}\left(
t\right)  }  &  =\frac{A\left(  t\right)  }{A_{\ast}\left(  t\right)
}\left\{  \vartheta_{t}\left(  x\right)  +x^{-1/\gamma}\dfrac{x^{\nu/\gamma
}-1}{\nu\gamma}\right\}  \frac{\overline{H}^{\left(  1\right)  }\left(
t\right)  }{\overline{H}^{\left(  1\right)  }\left(  tx\right)  }\\
&  \ \ \ \ \ \ \ \ \ \ \ +\frac{A_{1}\left(  t\right)  }{A_{\ast}\left(
t\right)  }\left\{  \vartheta_{t}^{\left(  1\right)  }\left(  x\right)
-x^{1/\gamma}\dfrac{x^{\nu_{1}/\gamma}-1}{\nu_{1}\gamma}\right\}
x^{-1/\gamma},
\end{align*}
where%
\[
\vartheta_{t}\left(  x\right)  :=\frac{\overline{H}\left(  tx\right)
/\overline{H}\left(  t\right)  -x^{-1/\gamma}}{A\left(  t\right)
}-x^{-1/\gamma}\dfrac{x^{\nu/\gamma}-1}{\nu\gamma}%
\]
and%
\[
\vartheta_{t}^{\left(  1\right)  }\left(  x\right)  :=\frac{\overline
{H}^{\left(  1\right)  }\left(  t\right)  /\overline{H}^{\left(  1\right)
}\left(  tx\right)  -x^{1/\gamma}}{A_{1}\left(  t\right)  }+x^{1/\gamma}%
\dfrac{x^{\nu_{1}/\gamma}-1}{\nu_{1}\gamma}.
\]
Thus, the uniform inequalities (for second-order regularly varying functions)
to both tails $\overline{H}$ and $\overline{H}^{\left(  1\right)  }$ (see,
e.g., Proposition 4 together with Remark 1 in \cite{HJ11}) yield that : for
any $0<\epsilon<1$ and $x\geq1,$
\[
\left\vert \vartheta_{t}\left(  x\right)  \right\vert <\epsilon x^{-1/\gamma
+\nu/\gamma+\epsilon}\text{ and }\left\vert \vartheta_{t}^{\left(  1\right)
}\left(  x\right)  \right\vert <\epsilon x^{1/\gamma+\nu_{1}/\gamma+\epsilon
},\text{ as }t\rightarrow\infty.
\]
On the other hand, we have $1/\overline{H}^{\left(  1\right)  }\in
\mathcal{RV}_{\left(  1/\gamma\right)  },$ then from Proposition $\left(
\ref{Potter}\right)  ,$ we infer that, uniformly over $x\geq1,$ $\overline
{H}^{\left(  1\right)  }\left(  t\right)  /\overline{H}^{\left(  1\right)
}\left(  tx\right)  =\left(  1+o\left(  x^{\epsilon}\right)  \right)
x^{1/\gamma},$ as $t\rightarrow\infty.$ Therefore%
\begin{align*}
\frac{\psi\left(  tx\right)  /\psi\left(  t\right)  -1}{A_{\ast}\left(
t\right)  }  &  =\frac{A\left(  t\right)  }{A_{\ast}\left(  t\right)
}\left\{  o\left(  x^{-1/\gamma+\nu/\gamma+\epsilon}\right)  +x^{-1/\gamma
}\dfrac{x^{\nu/\gamma}-1}{\nu\gamma}\right\}  \left(  1+o\left(  x^{\epsilon
}\right)  \right)  x^{1/\gamma}\\
&  \ \ \ \ \ \ \ \ \ \ \ \ \ \ \ \ \ \ \ \ +\frac{A_{1}\left(  t\right)
}{A_{\ast}\left(  t\right)  }\left\{  o\left(  x^{1/\gamma+\nu_{1}%
/\gamma+\epsilon}\right)  -x^{1/\gamma}\dfrac{x^{\nu_{1}/\gamma}-1}{\nu
_{1}\gamma}\right\}  x^{-1/\gamma}.
\end{align*}
If $\nu>\nu_{1},$ then $A_{\ast}\left(  t\right)  =A\left(  t\right)  $ and
$A_{1}\left(  t\right)  /A_{\ast}\left(  t\right)  =A_{1}\left(  t\right)
/A\left(  t\right)  =O\left(  t^{\left(  \nu_{1}-\nu\right)  /\gamma}\right)
$ which tends to zero as $t\rightarrow\infty,$ it follows that
\[
\frac{\psi\left(  tx\right)  /\psi\left(  t\right)  -1}{A_{\ast}\left(
t\right)  }-\dfrac{x^{\nu/\gamma}-1}{\nu\gamma}=o\left(  x^{\nu/\gamma
+\epsilon}\right)  .
\]
If $\nu\leq\nu_{1},$ then $A_{\ast}\left(  t\right)  =-A_{1}\left(  t\right)
$ and $A\left(  t\right)  /A_{\ast}\left(  t\right)  =A\left(  t\right)
/A_{1}\left(  t\right)  =O\left(  t^{\left(  \nu-\nu_{1}\right)  /\gamma
}\right)  =O\left(  1\right)  ,$ which goes to zero as $t\rightarrow\infty,$
then%
\[
\frac{\psi\left(  tx\right)  /\psi\left(  t\right)  -1}{A_{\ast}\left(
t\right)  }-\dfrac{x^{\nu_{1}/\gamma}-1}{\nu_{1}\gamma}=o\left(  x^{\nu
_{1}/\gamma+\epsilon}\right)  ,
\]
which completes the proof of $\left(  \ref{pshi}\right)  .$
\end{proof}

\begin{lemma}
\label{Lemma-3}For every $0<\eta<1/2,$ we have%
\[
\sup_{u\geq1}u^{\left(  1/2-\eta\right)  /\gamma}\left\vert \beta_{n}^{\ast
}\left(  u\right)  \right\vert =O_{\mathbb{P}}\left(  1\right)  =\sup_{u\geq
1}u^{\left(  1/2-\eta\right)  /\gamma}\left\vert \beta_{n}\left(  u\right)
\right\vert .
\]

\end{lemma}

\begin{proof}
Let $0<\eta<1/2$ and recall, from $\left(  \ref{beta start}\right)  ,$ that
\[
\beta_{n}^{\ast}\left(  u\right)  =\dfrac{n}{\sqrt{k}}\left(  \overline{H}%
_{n}\left(  uZ_{n-k:n}\right)  -\overline{H}\left(  uZ_{n-k:n}\right)
\right)  .
\]
By using the change of variables $t=uZ_{n-k:n}/h,$ we write%
\[
\sup_{u\geq1}u^{\left(  1/2-\eta\right)  /\gamma}\left\vert \beta_{n}^{\ast
}\left(  u\right)  \right\vert =\dfrac{n}{\sqrt{k}}\left(  \frac{h}{Z_{n-k:n}%
}\right)  ^{\left(  1/2-\eta\right)  /\gamma}\sup_{ht\geq Z_{n-k:n}}t^{\left(
1/2-\eta\right)  /\gamma}\dfrac{n}{\sqrt{k}}\left\vert \overline{H}_{n}\left(
th\right)  -\overline{H}\left(  th\right)  \right\vert ..
\]
From Theorem 2 in \cite{SW86} (page 4), we may write
\begin{equation}
\left\{  \overline{H}_{n}\left(  z\right)  ,\text{ }z\geq0\right\}  _{n\geq
1}\overset{\mathcal{D}}{=}\left\{  \mathbb{U}_{n}\left(  \overline{H}\left(
z\right)  \right)  ,\text{ }z\geq0\right\}  _{n\geq1}, \label{rep-H}%
\end{equation}
where $\mathbb{U}_{n}$ being the uniform empirical cdf given in $\left(
\ref{e}\right)  .$ Let us now introduce the uniform tail empirical process
$\vartheta_{n}\left(  s\right)  :=\sqrt{k}\left(  \dfrac{n}{k}\mathbb{U}%
_{n}\left(  \dfrac{k}{n}s\right)  -s\right)  ,$ $s>0.$ It is easy to verify,
in view of $\left(  \ref{rep-H}\right)  ,$ that
\[
\left\{  \dfrac{n}{\sqrt{k}}\left(  \overline{H}_{n}\left(  th\right)
-\overline{H}\left(  th\right)  \right)  ,\text{ }t>0\right\}  _{n\geq
1}\overset{\mathcal{D}}{=}\left\{  \vartheta_{n}\left(  \frac{n}{k}%
\overline{H}\left(  th\right)  \right)  ,\text{ }t>0\right\}  _{n\geq1},
\]
it follows that%
\[
\sup_{u\geq1}u^{\left(  1/2-\eta\right)  /\gamma}\left\vert \beta_{n}^{\ast
}\left(  u\right)  \right\vert \overset{\mathcal{D}}{=}\left(  \frac
{h}{Z_{n-k:n}}\right)  ^{\left(  1/2-\eta\right)  /\gamma}\sup_{ht\geq
Z_{n-k:n}}t^{\left(  1/2-\eta\right)  /\gamma}\left\vert \vartheta_{n}\left(
\frac{n}{k}\overline{H}\left(  th\right)  \right)  \right\vert .
\]
Since $\overline{H}$ is decreasing and continuous, then the previous supremum
is equivalent to that of over $0<\frac{n}{k}\overline{H}\left(  th\right)
\leq\frac{n}{k}\overline{H}\left(  Z_{n-k:n}\right)  ,$ thereby the latter
expression becomes%
\[
\left(  \frac{h}{Z_{n-k:n}}\right)  ^{\left(  1/2-\eta\right)  /\gamma}%
\sup_{0<\frac{n}{k}\overline{H}\left(  th\right)  \leq\frac{n}{k}\overline
{H}\left(  Z_{n-k:n}\right)  }t^{\left(  1/2-\eta\right)  /\gamma}\left(
\frac{n}{k}\overline{H}\left(  th\right)  \right)  ^{\left(  1/2-\eta\right)
}\sup_{0<s\leq\frac{n}{k}\overline{H}\left(  Z_{n-k:n}\right)  }%
\frac{\left\vert \vartheta_{n}\left(  s\right)  \right\vert }{s^{\left(
1/2-\eta\right)  }}.
\]
Let $c>0$ be fixed and $\epsilon>0$ be sufficiently small. It is clear that%
\begin{align*}
&  \mathbb{P}\left\{  \sup_{0<s\leq\frac{n}{k}\overline{H}\left(
Z_{n-k:n}\right)  }\frac{\left\vert \vartheta_{n}\left(  s\right)  \right\vert
}{s^{\left(  1/2-\eta\right)  }}>c\right\} \\
&  =\mathbb{P}\left\{  \left\vert \dfrac{n}{k}\overline{H}\left(
Z_{n-k:n}\right)  -1\right\vert >\epsilon\right\}  +\mathbb{P}\left\{
\sup_{0<s\leq1+\epsilon}\frac{\left\vert \vartheta_{n}\left(  s\right)
\right\vert }{s^{\left(  1/2-\eta\right)  }}>c\right\}  =:P_{n1}+P_{n2}.
\end{align*}
Since $\dfrac{n}{k}\overline{H}\left(  Z_{n-k:n}\right)  \overset{\mathbb{P}%
}{\rightarrow}1,$ then $\mathcal{A}_{n}:=\left\{  \left\vert \dfrac{n}%
{k}\overline{H}\left(  Z_{n-k:n}\right)  -1\right\vert <\epsilon\right\}
\mathbb{\ }$occurs with a large probability, that is $P_{n1}\downarrow0,$ as
$n\rightarrow\infty.$\ Next we show that $P_{n2}$ is asymptotically bounded.
Indeed, in set $\mathcal{A}_{n},$ we have
\[
\sup_{0<s\leq\frac{n}{k}\overline{H}\left(  Z_{n-k:n}\right)  }s^{-\left(
1/2-\eta\right)  }\left\vert \vartheta_{n}\left(  s\right)  \right\vert
\leq\sup_{0<s\leq1+\epsilon}s^{-\left(  1/2-\eta\right)  }\left\vert
\vartheta_{n}\left(  s\right)  \right\vert ,
\]
which, from Proposition 3.1 in \cite{EHL-2006}, is asymptotically bounded. On
the other hand, by applying Proposition \ref{Potter} (to $\overline{H}),$ we
write:
\begin{equation}
\left(  1-\epsilon\right)  t^{-1/\gamma+\epsilon}\leq\dfrac{n}{k}\overline
{H}\left(  th\right)  \leq\left(  1+\epsilon\right)  t^{-1/\gamma+\epsilon},
\label{H-inequa}%
\end{equation}
uniformly over $v\geq1.$ Then is easy to verify that, in set $\mathcal{A}%
_{n},$ we have
\[
\sup_{0<\frac{n}{k}\overline{H}\left(  th\right)  \leq\frac{n}{k}\overline
{H}\left(  Z_{n-k:n}\right)  }t^{\left(  1/2-\eta\right)  /\gamma}\left(
\frac{n}{k}\overline{H}\left(  th\right)  \right)  ^{\left(  1/2-\eta\right)
}=O_{\mathbb{P}}\left(  1\right)  ,
\]
this means that $P_{n2}=O\left(  1\right)  ,$ which meets the first assertion.
To show the second one, let us set
\begin{equation}
\widetilde{\beta}_{n}\left(  u\right)  :=\dfrac{n}{\sqrt{k}}\left(
\overline{H}_{n}^{\left(  0\right)  }\left(  Z_{n-k:n}u\right)  -\overline
{H}^{\left(  0\right)  }\left(  Z_{n-k:n}u\right)  \right)  .
\label{beta-tild}%
\end{equation}
From $\left(  \ref{beta start}\right)  ,$ we have $\widetilde{\beta}%
_{n}\left(  u\right)  =-\sqrt{\dfrac{n}{k}}\alpha_{n}\left(  1-\overline
{H}^{\left(  0\right)  }\left(  Z_{n-k:n}u\right)  \right)  ,$ (almost
surely). Since
\[
\left\{  \alpha_{n}\left(  s\right)  ,\text{ }0\leq s\leq1\right\}  _{n\geq
1}\overset{\mathcal{D}}{=}-\left\{  \alpha_{n}\left(  1-s\right)  ,\text{
}0\leq s\leq1\right\}  _{n\geq1},
\]
then
\[
\left\{  \widetilde{\beta}_{n}\left(  u\right)  ,\text{ }u\geq1\right\}
_{n\geq1}\overset{\mathcal{D}}{=}\left\{  \sqrt{\dfrac{n}{k}}\alpha_{n}\left(
\overline{H}^{\left(  0\right)  }\left(  Z_{n-k:n}u\right)  \right)  ,\text{
}u\geq1\right\}  _{n\geq1}.
\]
As we already did for $\beta_{n},$ we write%
\[
\sup_{u\geq1}u^{\left(  1/2-\eta\right)  /\gamma}\left\vert \widetilde{\beta
}_{n}\left(  u\right)  \right\vert \overset{\mathcal{D}}{=}\left(  \frac
{h}{Z_{n-k:n}}\right)  ^{\left(  1/2-\eta\right)  /\gamma}\sup_{ht\geq
Z_{n-k:n}}t^{\left(  1/2-\eta\right)  /\gamma}\left\vert \vartheta_{n}\left(
\frac{n}{k}\overline{H}^{\left(  0\right)  }\left(  th\right)  \right)
\right\vert .
\]
Recall that $\overline{H}\left(  t\right)  /\overline{H}^{\left(  1\right)
}\left(  t\right)  \rightarrow p^{-1}$ and\ $\overline{H}^{\left(  1\right)
}\left(  t\right)  =$ $\overline{H}\left(  t\right)  -\overline{H}^{\left(
0\right)  }\left(  t\right)  ,$ then is easy to verify that $\overline
{H}\left(  t\right)  /\overline{H}^{\left(  0\right)  }\left(  t\right)
\rightarrow q^{-1},$ as $t\rightarrow\infty,$ where $q=1-p,$ this implies that
$\overline{H}^{\left(  0\right)  }$ is regularly varying with index $\left(
-1/\gamma\right)  $ too. By using similar arguments as used for $\beta_{n},$
we also show $\sup_{u\geq1}u^{\left(  1/2-\eta\right)  /\gamma}\left\vert
\widetilde{\beta}_{n}\left(  u\right)  \right\vert =O_{\mathbb{P}}\left(
1\right)  .$ Note that $\beta_{n}\left(  u\right)  =\beta_{n}^{\ast}\left(
u\right)  -\widetilde{\beta}_{n}\left(  u\right)  ,$ it follows that
\[
\sup_{u\geq1}u^{\left(  1/2-\eta\right)  /\gamma}\left\vert \beta_{n}\left(
u\right)  \right\vert \leq\sup_{u\geq1}u^{\left(  1/2-\eta\right)  /\gamma
}\left\vert \beta_{n}^{\ast}\left(  u\right)  \right\vert +\sup_{u\geq
1}u^{\left(  1/2-\eta\right)  /\gamma}\left\vert \widetilde{\beta}_{n}\left(
u\right)  \right\vert .
\]
which by the first and the second results is $O_{\mathbb{P}}\left(  1\right)
.$ This completes the proof of the Lemma.
\end{proof}

\begin{lemma}
\label{Lemma-4}For every $0\leq\tau<1/8,$ we have%
\[
\frac{\overline{H}\left(  xZ_{n-k:n}\right)  }{\overline{H}^{\left(  1\right)
}\left(  xZ_{n-k:n}\right)  }\beta_{n}^{\ast}\left(  x\right)  -\frac{1}%
{p}\beta_{n}^{\ast}\left(  x\right)  =o_{\mathbb{P}}\left(  x^{-\tau/\gamma
}\right)  =\int_{x}^{\infty}\beta_{n}^{\ast}\left(  u\right)  d\left\{
\dfrac{\overline{H}\left(  Z_{n-k:n}u\right)  }{\overline{H}^{\left(
1\right)  }\left(  Z_{n-k:n}u\right)  }\right\}  .
\]

\end{lemma}

\begin{proof}
Let $0\leq\tau<1/8$ and write%
\[
\frac{\overline{H}\left(  xZ_{n-k:n}\right)  }{\overline{H}^{\left(  1\right)
}\left(  xZ_{n-k:n}\right)  }\beta_{n}^{\ast}\left(  x\right)  =\frac{1}%
{p}\beta_{n}^{\ast}\left(  x\right)  +\left\{  \frac{\overline{H}\left(
xZ_{n-k:n}\right)  }{\overline{H}^{\left(  1\right)  }\left(  xZ_{n-k:n}%
\right)  }-\frac{1}{p}\right\}  \beta_{n}^{\ast}\left(  x\right)  .
\]
Let $\epsilon>0$ be sufficiently small and $0<\eta<1/2$ so that $\eta
-1/2+\tau<0.$ By using inequality $\left(  \ref{inequa-p}\right)  $ and Lemma
\ref{Lemma-3} together, we get%
\[
\left\{  \frac{\overline{H}\left(  xZ_{n-k:n}\right)  }{\overline{H}^{\left(
1\right)  }\left(  xZ_{n-k:n}\right)  }-\frac{1}{p}\right\}  \beta_{n}^{\ast
}\left(  x\right)  =o_{\mathbb{P}}\left(  x^{\left(  \eta-1/2\right)
/\gamma+\epsilon}\right)  =o_{\mathbb{P}}\left(  x^{-\tau/\gamma}\right)  ,
\]
uniformly over $x\geq1,$ which meets the first assertion of this lemma. For
the second assertion, we once again make use of Lemma \ref{Lemma-3}, to write
\[
\int_{x}^{\infty}\beta_{n}^{\ast}\left(  u\right)  d\left\{  \frac
{\overline{H}\left(  Z_{n-k:n}u\right)  }{\overline{H}^{\left(  1\right)
}\left(  Z_{n-k:n}u\right)  }\right\}  =O_{\mathbb{P}}\left(  1\right)
\int_{x}^{\infty}u^{-\left(  1/2-\eta\right)  /\gamma}d\left\{  \frac{H\left(
Z_{n-k:n}u\right)  }{\overline{H}^{\left(  1\right)  }\left(  Z_{n-k:n}%
u\right)  }\right\}  .
\]
By making an integration by parts, then by using inequality $\left(
\ref{inequa-p}\right)  ,$ we readily show that the latter integral equals
$o_{\mathbb{P}}\left(  x^{-\tau/\gamma}\right)  ,$ as well. This completes the
proof of the lemma.
\end{proof}

\begin{lemma}
\label{lemma-5}For $n\geq1,$ we have%
\[%
\begin{tabular}
[c]{l}%
$\left(  i\right)  \text{ }\sup\limits_{U_{1:n}\leq t\leq1}\dfrac
{t}{\mathbb{U}_{n}\left(  t\right)  }=O_{\mathbb{P}}\left(  1\right)
=\sup\limits_{U_{1:n}\leq t\leq1}\dfrac{\mathbb{U}_{n}\left(  t\right)  }%
{t}.\medskip$\\
$\left(  ii\right)  \text{ }\sup\limits_{1/n\leq t\leq1}\dfrac{\mathbb{V}%
_{n}\left(  t\right)  }{t}=O_{\mathbb{P}}\left(  1\right)  =\sup
\limits_{1/n\leq t\leq1}\dfrac{t}{\mathbb{V}_{n}\left(  t\right)  }.\medskip$%
\end{tabular}
\ \ \ \
\]
For every $0<\eta<1/2,$ we have%
\[
\left(  iii\right)  \text{ }\sup\limits_{0\leq t\leq1}\dfrac{n^{\eta
}\left\vert \mathbb{U}_{n}\left(  t\right)  -t\right\vert }{\left(  t\left(
1-t\right)  \right)  ^{1-\eta}}=O_{\mathbb{P}}\left(  1\right)  =\sup
\limits_{1/n\leq t\leq1-1/n}\dfrac{n^{\eta}\left\vert t-\mathbb{V}_{n}\left(
t\right)  \right\vert }{\left(  t\left(  1-t\right)  \right)  ^{1-\eta}}.
\]

\end{lemma}

\begin{proof}
The proofs of the first two assertions may be found in \cite{SW86} (pages 415
and 416, inequality 2). The left result of assertion $\left(  iii\right)  $ is
also\ given in \cite{SW86} (page 425, assertion 15). For the right one, we
have%
\begin{align*}
&  \sup\limits_{1/n\leq t\leq1-1/n}\dfrac{n^{\eta}\left\vert t-\mathbb{V}%
_{n}\left(  t\right)  \right\vert }{\left(  t\left(  1-t\right)  \right)
^{1-\eta}}\\
&  =\sup\limits_{1/n\leq t\leq1-1/n}\left\{  \frac{1-\mathbb{V}_{n}\left(
t\right)  }{1-t}\right\}  ^{1-\eta}\left\{  \frac{\mathbb{V}_{n}\left(
t\right)  }{t}\right\}  ^{1-\eta}\frac{n^{\eta}\left\vert \mathbb{V}%
_{n}\left(  t\right)  -\mathbb{U}_{n}\left(  \mathbb{V}_{n}\left(  t\right)
\right)  \right\vert }{\left(  \mathbb{V}_{n}\left(  t\right)  \left(
1-\mathbb{V}_{n}\left(  t\right)  \right)  \right)  ^{1-\eta}},
\end{align*}
which is less than or equal to%
\begin{align*}
&  \sup\limits_{1/n\leq t\leq1-1/n}\left\{  \frac{1-\mathbb{V}_{n}\left(
t\right)  }{1-t}\right\}  ^{1-\eta}\sup\limits_{1/n\leq t\leq1-1/n}\left\{
\frac{\left(  \mathbb{V}_{n}\left(  t\right)  \right)  }{t}\right\}  ^{1-\eta
}\\
&
\ \ \ \ \ \ \ \ \ \ \ \ \ \ \ \ \ \ \ \ \ \ \ \ \ \ \ \ \ \ \ \ \ \ \ \times
\sup\limits_{1/n\leq t\leq1-1/n}\frac{n^{\eta}\left\vert \mathbb{V}_{n}\left(
t\right)  -\mathbb{U}_{n}\left(  \mathbb{V}_{n}\left(  t\right)  \right)
\right\vert }{\left(  \mathbb{V}_{n}\left(  t\right)  \left(  1-\mathbb{V}%
_{n}\left(  t\right)  \right)  \right)  ^{1-\eta}}.
\end{align*}
Note that $\left\{  1-\mathbb{V}_{n}\left(  t\right)  ;\text{ }0\leq
t\leq1\right\}  \overset{\mathcal{D}}{=}\left\{  \mathbb{V}_{n}\left(
1-t\right)  ;\text{ }0\leq t\leq1\right\}  ,$ which may be rewritten into
$\left\{  \mathbb{V}_{n}\left(  s\right)  ;\text{ }0\leq s\leq1\right\}  .$
Hence, without loss of generality, we may write%
\[
\sup\limits_{1/n\leq t\leq1-1/n}\left\{  \frac{1-\mathbb{V}_{n}\left(
t\right)  }{1-t}\right\}  ^{1-\eta}=\sup\limits_{1/n\leq s\leq1-1/n}\left\{
\frac{\mathbb{V}_{n}\left(  s\right)  }{s}\right\}  ^{1-\eta},
\]
which, from $\left(  ii\right)  ,$ equals $O_{\mathbb{P}}\left(  1\right)  .$
Observe now that $1/n\leq t\leq1-1/n$ is equivalent to $U_{1:n}\leq
\mathbb{V}_{n}\left(  t\right)  \leq U_{n:n},$ it follows that%
\[
\sup\limits_{1/n\leq t\leq1-1/n}\frac{n^{\eta}\left\vert \mathbb{V}_{n}\left(
t\right)  -\mathbb{U}_{n}\left(  \mathbb{V}_{n}\left(  t\right)  \right)
\right\vert }{\left(  \mathbb{V}_{n}\left(  t\right)  \left(  1-\mathbb{V}%
_{n}\left(  t\right)  \right)  \right)  ^{1-\eta}}\leq\sup\limits_{0\leq
s\leq1}\dfrac{n^{\eta}\left\vert \mathbb{U}_{n}\left(  s\right)  -s\right\vert
}{\left(  s\left(  1-s\right)  \right)  ^{1-\eta}}.
\]
Finally, using the left result of $\left(  iii\right)  $ completes the proof.
\end{proof}

\begin{lemma}
\label{Lemma-6}For every $0\leq\tau<1/8,$ we have $R_{ni}\left(  x\right)
=o_{\mathbb{P}}\left(  x^{-\tau/\gamma}\right)  ,$ for $i=1,2,3.$
\end{lemma}

\begin{proof}
By using representation $\left(  \ref{rep-H}\right)  $ and $\left(
\ref{HN1}\right)  $ with assertion $\left(  i\right)  $ of Lemma
\ref{lemma-5}, we get $\overline{H}_{n}\left(  w\right)  /\overline{H}\left(
w\right)  =O_{\mathbb{P}}\left(  1\right)  =\overline{H}_{n}^{\left(
1\right)  }\left(  w\right)  /\overline{H}^{\left(  1\right)  }\left(
w\right)  ,$ uniformly over $Z_{1:n}\leq w\leq Z_{n:n}.$ Since $H_{n}\left(
w\right)  =1,$ for $w\geq Z_{n:n},$ it follows that for any $x\geq1,$%
\[
R_{n1}\left(  x\right)  =O_{\mathbb{P}}\left(  k^{-1}\right)  \frac{n}%
{\sqrt{k}}\int_{xZ_{n-k:n}}^{\infty}\left(  \frac{\overline{H}\left(
w\right)  }{\overline{H}^{\left(  1\right)  }\left(  w\right)  }\right)
^{2}dH\left(  w\right)  .
\]
As we did in the proof of Lemma \ref{lemma-5}, let us consider the set
$\mathcal{A}_{n}:=\left\{  \left\vert Z_{n-k:n}/h-1\right\vert <\epsilon
\right\}  ,$ in which the previous equation becomes
\[
R_{n1}\left(  x\right)  =O_{\mathbb{P}}\left(  k^{-1}\right)  \frac{n}%
{\sqrt{k}}\int_{\left(  1-\epsilon\right)  xh}^{\infty}\left(  \frac
{\overline{H}\left(  w\right)  }{\overline{H}^{\left(  1\right)  }\left(
w\right)  }\right)  ^{2}dH\left(  w\right)  .
\]
Let us rewrite this latter into%
\[
R_{n1}\left(  x\right)  =O_{\mathbb{P}}\left(  k^{-1/2}\right)  \int_{\left(
1-\epsilon\right)  x}^{\infty}\left(  \frac{\overline{H}\left(  wh\right)
}{\overline{H}^{\left(  1\right)  }\left(  wh\right)  }\right)  ^{2}%
d\frac{\overline{H}\left(  wh\right)  }{\overline{H}\left(  h\right)  }.
\]
From inequality $\left(  \ref{inequa-p}\right)  ,$ we have $\overline
{H}\left(  wh\right)  /\overline{H}^{\left(  1\right)  }\left(  wh\right)
=\left(  1+o\left(  1\right)  \right)  p^{-1},$ as $n\rightarrow\infty,$
uniformly on $w>x.$ Then by using Proposition \ref{Potter} (for $\overline
{H}),$ we obtain $R_{n1}\left(  x\right)  =O_{\mathbb{P}}\left(
k^{-1/2}\right)  x^{-1/\gamma},$ which equals $o_{\mathbb{P}}\left(
x^{-\tau/\gamma}\right)  $ as $n\rightarrow\infty,$ uniformly on $x\geq1.$ For
the second term, in set $\mathcal{A}_{n},$ we write%
\[
R_{n2}\left(  x\right)  =O_{\mathbb{P}}\left(  k^{-1}\right)  \frac{n}%
{\sqrt{k}}\int_{\left(  1-\epsilon\right)  xh}^{\infty}\left(  \frac
{\overline{H}\left(  w\right)  }{\overline{H}^{\left(  1\right)  }\left(
w\right)  }\right)  ^{2}d\left(  \overline{H}_{n}\left(  w\right)
-\overline{H}\left(  w\right)  \right)  .
\]
For convenience let us set $h_{\epsilon}:=\left(  1-\epsilon\right)  h.$ By
using a change of variables and integration by parts, the latter integral
becomes%
\[
-\left(  \frac{\overline{H}\left(  h_{\epsilon}\right)  }{\overline
{H}^{\left(  1\right)  }\left(  h_{\epsilon}\right)  }\right)  ^{2}\left(
\overline{H}_{n}\left(  h_{\epsilon}\right)  -\overline{H}\left(  h_{\epsilon
}\right)  \right)  -\int_{x}^{\infty}\left(  \overline{H}_{n}\left(
h_{\epsilon}w\right)  -\overline{H}\left(  h_{\epsilon}w\right)  \right)
d\left(  \frac{\overline{H}\left(  h_{\epsilon}w\right)  }{\overline
{H}^{\left(  1\right)  }\left(  h_{\epsilon}w\right)  }\right)  ^{2}.
\]
Let us fix $0<\eta<1/2.$ The representation $\left(  \ref{rep-H}\right)  $ and
assertion $\left(  iii\right)  $ of of Lemma \ref{lemma-5} together imply
that, uniformly over $w\geq x,$ we have $\overline{H}_{n}\left(  h_{\epsilon
}w\right)  -\overline{H}\left(  h_{\epsilon}w\right)  =O_{\mathbb{P}}\left(
n^{-\eta}\right)  \left(  \overline{H}\left(  h_{\epsilon}w\right)  \right)
^{1-\eta}.$ By using this latter with inequality $\left(  \ref{inequa-p}%
\right)  $ and Proposition \ref{Potter} (to together $\overline{H}),$ we end
up with $R_{n2}\left(  x\right)  =O_{\mathbb{P}}\left(  k^{-\eta}\right)
x^{\left(  \eta-1\right)  /\gamma}=o_{\mathbb{P}}\left(  1\right)
x^{-\tau/\gamma}.$ For the third term
\[
R_{n3}\left(  x\right)  =O_{\mathbb{P}}\left(  k^{-1}\right)  \frac{n}%
{\sqrt{k}}\int_{xZ_{n-k:n}}^{\infty}\frac{\overline{H}\left(  w\right)
}{\overline{H}^{\left(  1\right)  }\left(  w\right)  }dH\left(  w\right)  ,
\]
we use similar arguments to show that equals $O_{\mathbb{P}}\left(
k^{-1/2}\right)  x^{-1/\gamma}=o_{\mathbb{P}}\left(  1\right)  x^{-\tau
/\gamma},$ which finishes the proof of the lemma.
\end{proof}

\begin{lemma}
\label{Lemma-7}For every $0<\eta<1/2,$ we have $\sup_{k^{-1}\leq s\leq
1}s^{\eta+\gamma}\left\vert \ell_{n}\left(  s\right)  -s^{-\gamma}\right\vert
=o_{\mathbb{P}}\left(  1\right)  .$
\end{lemma}

\begin{proof}
Recall that in $\left(  \ref{ln(s)}\right)  ,$ we set $\ell_{n}\left(
s\right)  =Q_{n}\left(  1-ks/n\right)  /Z_{n-k:n},$ $0<s\leq1.$ Observe that
$\ell_{n}\left(  s\right)  -s^{-\gamma}$ may be rewritten into the sum of%
\[
L_{1,n}\left(  s\right)  :=\ell_{n}\left(  s\right)  -\left(  \frac
{\mathbb{V}_{n}\left(  ks/n\right)  }{U_{k+1:n}}\right)  ^{-\gamma},\text{
}L_{2,n}\left(  s\right)  :=\left(  \frac{\mathbb{V}_{n}\left(  ks/n\right)
}{U_{k+1:n}}\right)  ^{-\gamma}-\left(  \frac{\mathbb{V}_{n}\left(
ks/n\right)  }{k/n}\right)  ^{-\gamma},
\]
and%
\[
L_{3,n}\left(  s\right)  :=\left(  \frac{\mathbb{V}_{n}\left(  ks/n\right)
}{k/n}\right)  ^{-\gamma}-s^{-\gamma}.
\]
First note that, for each $n\geq1,$ we have $Z_{n-k:n}\overset{\mathcal{D}}%
{=}Q\left(  1-U_{k+1:n}\right)  $ and
\[
\left\{  Q_{n}\left(  1-ks/n\right)  ,\text{ }0<s<1\right\}  \overset
{\mathcal{D}}{=}\left\{  Q\left(  1-\mathbb{V}_{n}\left(  ks/n\right)
\right)  ,\text{ }0<s<1\right\}  ,
\]
then, without loss of generality, we may assume that
\[
\ell_{n}\left(  s\right)  =\frac{Q\left(  1-\mathbb{V}_{n}\left(  ks/n\right)
\right)  }{Q\left(  1-U_{k+1:n}\right)  },\text{ }0<s<1.
\]
On the other hand, we have $Q\left(  1-\cdot\right)  \in\mathcal{RV}_{\left(
-\gamma\right)  }$ and $\left(  k/n\right)  /U_{k+1:n}\overset{\mathbb{P}%
}{\rightarrow}1.$ By making use of Proposition $\ref{Potter}$ with $f\left(
s\right)  =Q\left(  1-s\right)  $ and $t=\left(  k/n\right)  /U_{k+1:n},$ we
infer that, for any small $\epsilon>0,$ we have $L_{1,n}\left(  s\right)
=o_{\mathbb{P}}\left(  1\right)  \left(  \mathbb{V}_{n}\left(  ks/n\right)
/U_{k+1:n}\right)  ^{-\gamma+\epsilon},$ uniformly over $k^{-1}\leq s<1,$ in
other term $L_{1,n}\left(  s\right)  =o_{\mathbb{P}}\left(  1\right)  \left(
\mathbb{V}_{n}\left(  ks/n\right)  /\left(  k/n\right)  \right)
^{-\gamma+\epsilon}.$ By using assertion $\left(  ii\right)  $ of Lemma
\ref{lemma-5}, we end up with
\begin{equation}
L_{1,n}\left(  s\right)  =o_{\mathbb{P}}\left(  s^{-\gamma+\epsilon}\right)
,\text{ as }n\rightarrow\infty.\label{L1}%
\end{equation}
The second term may be rewritten into%
\[
L_{2,n}\left(  s\right)  =\left(  \frac{\mathbb{V}_{n}\left(  ks/n\right)
}{k/n}\right)  ^{-\gamma}\left\{  \left(  \frac{k/n}{U_{k+1:n}}\right)
^{-\gamma}-1\right\}  .
\]
Since $nU_{k+1:n}/k\overset{\mathbb{P}}{\rightarrow}1,$ then uniformly over
$k^{-1}\leq s<1,$ we have%
\begin{equation}
L_{2,n}\left(  s\right)  =o_{\mathbb{P}}\left(  s^{-\gamma}\right)  ,\text{ as
}n\rightarrow\infty.\label{L2}%
\end{equation}
Observe now that $L_{3,n}\left(  s\right)  =\left(  k/n\right)  ^{\gamma
}\left(  \mathbb{V}_{n}\left(  ks/n\right)  ^{-\gamma}-\left(  sk/n\right)
^{-\gamma}\right)  ,$ which by using the mean value theorem equals
$-\gamma\left(  k/n\right)  ^{\gamma}\left\{  \mathbb{V}_{n}^{\ast}\left(
ks/n\right)  \right\}  ^{-\gamma-1}\left(  \mathbb{V}_{n}\left(  ks/n\right)
-sk/n\right)  ,$ where $\mathbb{V}_{n}^{\ast}\left(  ks/n\right)  $ is between
$\mathbb{V}_{n}\left(  ks/n\right)  $ and $sk/n.$ Once again by using
assertion $\left(  ii\right)  $ of Lemma \ref{lemma-5}, we show that
$\mathbb{V}_{n}^{\ast}\left(  ks/n\right)  /\left(  ks/n\right)
=O_{\mathbb{P}}\left(  1\right)  $ uniformly over $k^{-1}\leq s<1,$ therefore
\[
L_{3,n}\left(  s\right)  =O_{\mathbb{P}}\left(  1\right)  \left(  n/k\right)
s^{-\gamma-1}\left\vert \mathbb{V}_{n}\left(  ks/n\right)  -sk/n\right\vert .
\]
It is clear that $L_{3,n}\left(  s\right)  =s^{-\gamma-\eta}O_{\mathbb{P}%
}\left(  k^{-\eta}\right)  n^{\eta}\left\vert \mathbb{V}_{n}\left(
ks/n\right)  -sk/n\right\vert /\left(  sk/n\right)  ^{1-\eta}.$ Making use of
assertion $\left(  iii\right)  $ of Lemma \ref{lemma-5}, we get
\begin{equation}
L_{3,n}\left(  s\right)  =o_{\mathbb{P}}\left(  s^{-\gamma-\eta}\right)
,\text{ as }n\rightarrow\infty.\label{L3}%
\end{equation}
By combining $\left(  \ref{L1}\right)  ,$ $\left(  \ref{L2}\right)  $ and
$\left(  \ref{L3}\right)  $ we show that $\sum_{i=1}^{3}L_{i,n}\left(
s\right)  =o_{\mathbb{P}}\left(  s^{-\gamma-\eta}\right)  $ uniformly over
$k^{-1}\leq s<1,$ as sought.
\end{proof}

\begin{lemma}
\label{Lemma-8}For every $0\leq\tau<1/8,$ we have $R_{ni}\left(  x\right)
=o_{\mathbb{P}}\left(  x^{-\tau/\gamma}\right)  ,$ for $i=4,5.$
\end{lemma}

\begin{proof}
Recall that%
\[
R_{n4}\left(  x\right)  =-\int_{x}^{\infty}\frac{\overline{H}_{n}^{\left(
1\right)  }\left(  uZ_{n-k:n}\right)  -\overline{H}^{\left(  1\right)
}\left(  uZ_{n-k:n}\right)  }{\overline{H}_{n}^{\left(  1\right)  }\left(
uZ_{n-k:n}\right)  \overline{H}^{\left(  1\right)  }\left(  uZ_{n-k:n}\right)
}\beta_{n}^{\ast}\left(  u\right)  dH_{n}\left(  uZ_{n-k:n}\right)  .
\]
First note that $\overline{H}^{\left(  1\right)  }>\overline{H}$ and
$\overline{H}_{n}^{\left(  1\right)  }>\overline{H}_{n}.$ On the other hand,
by using $\left(  \ref{rep-H}\right)  $ and assertion $\left(  i\right)  $ of
Lemma \ref{lemma-5}, we infer that $\overline{H}\left(  uZ_{n-k:n}\right)
/\overline{H}_{n}\left(  uZ_{n-k:n}\right)  =O_{\mathbb{P}}\left(  1\right)
,$ uniformly over $u\geq1.$ It follows that%
\[
R_{n4}\left(  x\right)  =O_{\mathbb{P}}\left(  1\right)  \int_{x}^{\infty
}\frac{\left\vert \overline{H}_{n}^{\left(  1\right)  }\left(  uZ_{n-k:n}%
\right)  -\overline{H}^{\left(  1\right)  }\left(  uZ_{n-k:n}\right)
\right\vert }{\left[  \overline{H}\left(  uZ_{n-k:n}\right)  \right]  ^{2}%
}\left\vert \beta_{n}^{\ast}\left(  u\right)  \right\vert dH_{n}\left(
uZ_{n-k:n}\right)  .
\]
Observe that $\overline{H}_{n}^{\left(  1\right)  }\left(  uZ_{n-k:n}\right)
-\overline{H}^{\left(  1\right)  }\left(  uZ_{n-k:n}\right)  $ may be
rewritten into%
\[
\left(  \overline{H}_{n}\left(  uZ_{n-k:n}\right)  -\overline{H}\left(
uZ_{n-k:n}\right)  \right)  -\left(  \overline{H}_{n}^{\left(  0\right)
}\left(  uZ_{n-k:n}\right)  -\overline{H}^{\left(  0\right)  }\left(
uZ_{n-k:n}\right)  \right)  .
\]
It is clear that $R_{n,4}\left(  x\right)  =O_{\mathbb{P}}\left(  1\right)
\left(  C_{n}\left(  x\right)  +C_{n}^{\left(  0\right)  }\left(  x\right)
\right)  ,$ where%
\[
C_{n}\left(  x\right)  :=\int_{x}^{\infty}\frac{\left\vert \overline{H}%
_{n}\left(  uZ_{n-k:n}\right)  -\overline{H}\left(  uZ_{n-k:n}\right)
\right\vert }{\left[  \overline{H}\left(  uZ_{n-k:n}\right)  \right]  ^{2}%
}\left\vert \beta_{n}^{\ast}\left(  u\right)  \right\vert dH_{n}\left(
uZ_{n-k:n}\right)  ,
\]
and%
\[
C_{n}^{\left(  0\right)  }\left(  x\right)  :=\int_{x}^{\infty}\frac
{\left\vert \overline{H}_{n}^{\left(  0\right)  }\left(  uZ_{n-k:n}\right)
-\overline{H}^{\left(  0\right)  }\left(  uZ_{n-k:n}\right)  \right\vert
}{\left[  \overline{H}\left(  uZ_{n-k:n}\right)  \right]  ^{2}}\left\vert
\beta_{n}^{\ast}\left(  u\right)  \right\vert dH_{n}\left(  uZ_{n-k:n}\right)
.
\]
Next, we show that both $C_{n}\left(  x\right)  $ and $C_{n}^{\left(
0\right)  }\left(  x\right)  $ are equal to $o_{\mathbb{P}}\left(
x^{-\tau/\gamma}\right)  .$ Indeed, by making use of assertion $\left(
ii\right)  $ in Lemma \ref{lemma-5}, we write
\[
C_{n}\left(  x\right)  =O_{\mathbb{P}}\left(  n^{-\eta}\right)  \int
_{x}^{\infty}\frac{u^{-\left(  1/2-\eta\right)  /\gamma+\epsilon}}{\left[
\overline{H}\left(  uZ_{n-k:n}\right)  \right]  ^{1+\eta}}dH_{n}\left(
uZ_{n-k:n}\right)  .
\]
In view of inequality $\left(  \ref{inequa-H}\right)  ,$ we get
\[
C_{n}\left(  x\right)  =O_{\mathbb{P}}\left(  n^{-\eta}\right)  \left(
\frac{n}{k}\right)  ^{1+\eta}\int_{x}^{\infty}u^{\left(  1+\eta\right)
/\gamma-\left(  1/2-\eta\right)  /\gamma+\epsilon}dH_{n}\left(  uZ_{n-k:n}%
\right)  .
\]
By an integration by parts and once again by making use of assertion $\left(
i\right)  $ of Lemma \ref{lemma-5}, we end up with $C_{n}\left(  x\right)
=O_{\mathbb{P}}\left(  k^{-\eta}\right)  x^{\left(  \eta-1/2\right)  /\gamma
}.$ Since $k\rightarrow\infty$ and $x^{\left(  \eta-1/2\right)  /\gamma
}=O\left(  x^{-\tau/\gamma}\right)  ,$ uniformly over $x\geq1,$ it follows
that $C_{n}\left(  x\right)  =o_{\mathbb{P}}\left(  x^{-\tau/\gamma}\right)
.$ Let us now consider $C_{n}^{\left(  0\right)  }\left(  x\right)  .$ Note
that we also have $\overline{H}^{\left(  0\right)  }>\overline{H}$ and
$\overline{H}_{n}^{\left(  0\right)  }>\overline{H}_{n}$ and recall that from
$\left(  \ref{HN0}\right)  ,$ we have almost surely%
\[
\overline{H}_{n}^{\left(  0\right)  }\left(  v\right)  -\overline{H}^{\left(
0\right)  }\left(  v\right)  =\mathbb{U}_{n}\left(  1-\overline{H}^{\left(
0\right)  }\left(  v\right)  \right)  -\left(  1-\overline{H}^{\left(
0\right)  }\left(  v\right)  \right)  .
\]
Then, by using similar arguments as for $C_{n}\left(  x\right)  $ we also
readily show that $C_{n}^{\left(  0\right)  }\left(  x\right)  =o_{\mathbb{P}%
}\left(  x^{-\tau/\gamma}\right)  ,$ therefore we omit further details. Let us
now take care of the second term%
\[
R_{n5}\left(  x\right)  =\int_{x}^{\infty}\left\{  \frac{1}{\overline
{H}^{\left(  1\right)  }\left(  Z_{n-k:n}u\right)  }-\frac{n}{k}\frac{1}%
{p}u^{1/\gamma}\right\}  \beta_{n}^{\ast}\left(  u\right)  dH_{n}\left(
Z_{n-k:n}u\right)  .
\]
Recall that $1/\overline{H}^{\left(  1\right)  }\in\mathcal{RV}_{\left(
1/\gamma\right)  },$ then by using Lemma \ref{lemma-5}, we infer that with a
probability tending to $1,$ large $n$ and for every small $\epsilon>0,$ we
have%
\[
\left\vert \frac{\overline{H}^{\left(  1\right)  }\left(  Z_{n-k:n}\right)
}{\overline{H}^{\left(  1\right)  }\left(  Z_{n-k:n}u\right)  }-u^{1/\gamma
}\right\vert <\epsilon u^{\epsilon},
\]
uniformly over $u\geq1.\ $However, by making use of inequality $\left(
\ref{inequa-p}\right)  ,$ we get
\[
\left\vert \frac{\overline{H}\left(  Z_{n-k:n}\right)  }{\overline{H}^{\left(
1\right)  }\left(  Z_{n-k:n}\right)  }-\frac{1}{p}\right\vert <\epsilon.
\]
On the other hand, since $\dfrac{n}{k}\overline{H}\left(  Z_{n-k:n}\right)
\overset{\mathbb{P}}{\rightarrow}1$ then with a probability tending to $1$ and
large $n,$ we have $\left\vert \dfrac{n}{k}\overline{H}\left(  Z_{n-k:n}%
\right)  -1\right\vert <\epsilon.$ By combining the last three inequalities,
we end up with%
\begin{equation}
\left\vert \frac{1}{\overline{H}^{\left(  1\right)  }\left(  Z_{n-k:n}%
u\right)  }-\frac{n}{k}\frac{1}{p}u^{1/\gamma}\right\vert <\epsilon
u^{\epsilon+1/\gamma}, \label{three}%
\end{equation}
uniformly over $u\geq1,\ $this implies that
\[
R_{n5}\left(  x\right)  =o_{\mathbb{P}}\left(  1\right)  \int_{x}^{\infty
}u^{-\left(  1/2-\eta\right)  /\gamma-\epsilon}z^{\epsilon+1/\gamma}%
dH_{n}\left(  Z_{n-k:n}u\right)  .
\]
Once again, by using assertion $\left(  i\right)  $ of Lemma \ref{lemma-5},
yields
\[
R_{n5}\left(  x\right)  =o_{\mathbb{P}}\left(  1\right)  \int_{x}^{\infty
}u^{-\left(  1/2-\eta\right)  /\gamma-\epsilon}u^{-1/\gamma-1+1/\gamma
+\epsilon+\epsilon}du,
\]
which equals $o_{\mathbb{P}}\left(  1\right)  x^{\left(  \eta-1/2\right)
/\gamma+\epsilon}=o_{\mathbb{P}}\left(  x^{-\tau/\gamma}\right)  .$ This
achieves the proof of the lemma.
\end{proof}

\begin{lemma}
\label{Lemma-9}For every $0\leq\tau<1/8,$ we have $R_{ni}\left(  x\right)
=o_{\mathbb{P}}\left(  x^{-\tau/\gamma}\right)  ,$ $i=6,7.$
\end{lemma}

\begin{proof}
Recall that%
\[
R_{n6}\left(  x\right)  =\int_{x}^{\infty}\left\{  \frac{\overline{H}\left(
Z_{n-k:n}\right)  }{\overline{H}^{\left(  1\right)  }\left(  uZ_{n-k:n}%
\right)  }-\frac{1}{p}u^{1/\gamma}\right\}  \beta_{n}^{\ast}\left(  u\right)
d\left\{  \frac{H_{n}\left(  Z_{n-k:n}u\right)  }{\overline{H}\left(
Z_{n-k:n}\right)  }\right\}  .
\]
It it clear that%
\[
\left\vert R_{n6}\left(  x\right)  \right\vert \leq\int_{x}^{\infty}\left\vert
\frac{\overline{H}\left(  Z_{n-k:n}\right)  }{\overline{H}^{\left(  1\right)
}\left(  uZ_{n-k:n}\right)  }-\frac{1}{p}u^{1/\gamma}\right\vert \left\vert
\beta_{n}^{\ast}\left(  u\right)  \right\vert d\left\{  \frac{H_{n}\left(
Z_{n-k:n}u\right)  }{\overline{H}\left(  Z_{n-k:n}\right)  }\right\}  .
\]
Let us write%
\begin{align*}
&  \frac{\overline{H}\left(  Z_{n-k:n}\right)  }{\overline{H}^{\left(
1\right)  }\left(  uZ_{n-k:n}\right)  }-\frac{1}{p}u^{1/\gamma}\\
&  =\left(  \frac{\overline{H}\left(  Z_{n-k:n}\right)  }{\overline
{H}^{\left(  1\right)  }\left(  Z_{n-k:n}\right)  }-\frac{1}{p}\right)
\frac{\overline{H}^{\left(  1\right)  }\left(  Z_{n-k:n}\right)  }%
{\overline{H}^{\left(  1\right)  }\left(  uZ_{n-k:n}\right)  }+\frac{1}%
{p}\left(  \frac{\overline{H}^{\left(  1\right)  }\left(  Z_{n-k:n}\right)
}{\overline{H}^{\left(  1\right)  }\left(  uZ_{n-k:n}\right)  }-u^{1/\gamma
}\right)  .
\end{align*}
By using inequality $\left(  \ref{inequa-p}\right)  $ and Proposition
\ref{Potter} to $\overline{H}^{\left(  1\right)  },$ we infer that%
\[
\frac{\overline{H}\left(  Z_{n-k:n}\right)  }{\overline{H}^{\left(  1\right)
}\left(  uZ_{n-k:n}\right)  }-\frac{1}{p}u^{1/\gamma}=o_{\mathbb{P}}\left(
u^{1/\gamma+\epsilon}\right)  ,\text{ uniformly over }u\geq1.
\]
It follows that $R_{n6}\left(  x\right)  =o_{\mathbb{P}}\left(  1\right)
\int_{x}^{\infty}u^{\left(  1/2+\eta\right)  /\gamma+\epsilon}d\left\{
H_{n}\left(  Z_{n-k:n}u\right)  /\overline{H}\left(  Z_{n-k:n}\right)
\right\}  .$ By using an integration by parts then once again by making use of
assertion $\left(  i\right)  $ of Lemma \ref{lemma-5}, we show easily that
$R_{n6}\left(  x\right)  =o_{\mathbb{P}}\left(  x^{-\tau/\gamma}\right)  .$ By
similar routine arguments, we also show $R_{n7}\left(  x\right)
=o_{\mathbb{P}}\left(  x^{-\tau/\gamma}\right)  $ that we omit further details.
\end{proof}

\begin{lemma}
\label{Lemma-10}For every $0\leq\tau<1/8,$ we have $R_{ni}\left(  x\right)
=o_{\mathbb{P}}\left(  x^{-\tau/\gamma}\right)  ,$ $i=8,9.$
\end{lemma}

\begin{proof}
Recall that
\[
R_{n8}\left(  x\right)  =-\frac{n}{\sqrt{k}}\int_{xZ_{n-k:n}}^{xh}\left\{
\frac{\overline{H}\left(  w\right)  }{\overline{H}^{\left(  1\right)  }\left(
w\right)  }-p^{-1}\right\}  d\overline{H}\left(  w\right)  .
\]
By making a change of variables, we get%
\[
R_{n8}\left(  x\right)  =-\sqrt{k}\int_{xZ_{n-k:n}/h}^{x}\left\{
\frac{\overline{H}\left(  hw\right)  }{\overline{H}^{\left(  1\right)
}\left(  hw\right)  }-p^{-1}\right\}  d\frac{\overline{H}\left(  hw\right)
}{\overline{H}\left(  h\right)  }.
\]
In view of inequality $\left(  \ref{inequa-p}\right)  $ and Proposition
\ref{Potter} to $\overline{H}$ we get
\[
R_{n8}\left(  x\right)  =o_{\mathbb{P}}\left(  1\right)  \sqrt{k}%
\int_{xZ_{n-k:n}/h}^{x}w^{-1/\gamma-1+\epsilon}dw.
\]
After integration, we obtain $R_{n8}\left(  x\right)  =o_{\mathbb{P}}\left(
x^{-1/\gamma+\epsilon}\right)  \sqrt{k}\left(  1-\left(  Z_{n-k:n}/h\right)
^{-1/\gamma+\epsilon}\right)  .$ Recall that $\sqrt{k}\left(  1-Z_{n-k:n}%
/h\right)  =O_{\mathbb{P}}\left(  1\right)  ,$ then by using the mean theorem,
we infer that
\[
\sqrt{k}\left(  1-\left(  Z_{n-k:n}/h\right)  ^{-1/\gamma+\epsilon}\right)
=O_{\mathbb{P}}\left(  1\right)  ,
\]
as well. It follows that $R_{n8}\left(  x\right)  =o_{\mathbb{P}}\left(
x^{-1/\gamma+\epsilon}\right)  $ which also equals $o_{\mathbb{P}}\left(
x^{-\tau/\gamma}\right)  .$ Observe now that%
\[
R_{n9}\left(  x\right)  =p^{-1}\frac{\overline{H}\left(  xh\right)
}{\overline{H}\left(  h\right)  }\sqrt{k}\left(  \frac{\overline{H}\left(
xZ_{n-k:n}\right)  }{\overline{H}\left(  xh\right)  }-\left(  \frac{Z_{n-k:n}%
}{h}\right)  ^{-1/\gamma}\right)  .
\]
By applying Proposition \ref{Potter} implies that, for any $0<\epsilon<1,$
there exists $n_{0}=n_{0}\left(  \epsilon\right)  ,$ such that for all
$n>n_{0}$ and $y\geq1$%
\[
\left\vert \frac{\overline{H}\left(  xZ_{n-k:n}\right)  }{\overline{H}\left(
xh\right)  }-\left(  \frac{Z_{n-k:n}}{h}\right)  ^{-1/\gamma}\right\vert
\leq\epsilon\left(  \frac{Z_{n-k:n}}{h}\right)  ^{-1/\gamma+},
\]
it follows from the previous inequality, that
\[
\frac{\overline{H}\left(  xZ_{n-k:n}\right)  }{\overline{H}\left(  xh\right)
}-\left(  \frac{Z_{n-k:n}}{h}\right)  ^{-1/\gamma}=o_{\mathbb{P}}\left(
1\right)  ,.
\]
uniformly over $x\geq1.$ Since $\overline{H}\left(  xh\right)  /\overline
{H}\left(  h\right)  \rightarrow x^{-1/\gamma}$ thus $R_{n9}\left(  x\right)
=o_{\mathbb{P}}\left(  x^{-1/\gamma}\right)  =o_{\mathbb{P}}\left(
x^{-\tau/\gamma}\right)  ,$ as sought.
\end{proof}

\begin{lemma}
\label{Lemma-11}For every $0<\eta<1/2,$ we have%
\[
\sqrt{\frac{n}{k}}\sup_{k^{-1}\leq s\leq1}s^{\left(  \eta-1\right)
/2}\left\vert B_{n}^{\ast}\left(  \frac{k}{n}\left(  \ell_{n}\left(  s\right)
\right)  ^{-1/\gamma}\right)  -B_{n}^{\ast}\left(  \frac{k}{n}s\right)
\right\vert =o_{\mathbb{P}}\left(  1\right)  .
\]

\end{lemma}

\begin{proof}
Let $0<\eta<1/2$ be fixed, then from Lemma \ref{Lemma-7} and by using the mean
value theorem, we may write
\begin{equation}
\sup_{k^{-1}\leq s\leq1}s^{1-\eta}\left\vert \left(  \ell_{n}\left(  s\right)
\right)  ^{-1/\gamma}-s\right\vert =o_{\mathbb{P}}\left(  1\right)  .
\end{equation}
Recall now that $B_{n}^{\ast}\left(  s\right)  =B_{n}\left(  ps\right)
-B_{n}\left(  1-qs\right)  $ with $q=1-p,$ and write%
\begin{align*}
&  \left\vert B_{n}^{\ast}\left(  \dfrac{k}{n}\left(  \ell_{n}\left(
s\right)  \right)  ^{-1/\gamma}\right)  -B_{n}^{\ast}\left(  \dfrac{k}%
{n}s\right)  \right\vert \\
&  \leq\left\vert B_{n}\left(  p\frac{k}{n}\left(  \ell_{n}\left(  s\right)
\right)  ^{-1/\gamma}\right)  -B_{n}\left(  p\frac{k}{n}s\right)  \right\vert
+\left\vert B_{n}\left(  1-q\frac{k}{n}\left(  \ell_{n}\left(  s\right)
\right)  ^{-1/\gamma}\right)  -B_{n}\left(  1-q\frac{k}{n}s\right)
\right\vert .
\end{align*}
From inequality (1.11) in \cite{CsCsHM86}, we infer that
\begin{equation}
\mathbb{P}\left\{  \sup_{\left\vert d\right\vert <1,\text{ }0\leq x+d\leq
1}d^{1/2}\left\vert B_{n}\left(  x+d\right)  -B_{n}\left(  x\right)
\right\vert \geq\lambda\right\}  \leq c\lambda^{-1}\exp\left(  -\lambda
^{2}/8\right)  , \label{inequa}%
\end{equation}
for every $0<\lambda<\infty,$ where $c$ is a suitably chosen universal
constant. Then, by using this inequality with $d=p\dfrac{k}{n}\left(  \ell
_{n}\left(  s\right)  \right)  ^{-1/\gamma}-p\dfrac{k}{n}s$ and $x=p\dfrac
{k}{n}s,$ together with $\left(  \ref{ksi}\right)  ,$ we may readily show
that
\[
\sqrt{\frac{n}{k}}\left\vert B_{n}\left(  p\frac{k}{n}\left(  \ell_{n}\left(
s\right)  \right)  ^{-1/\gamma}\right)  -B_{n}\left(  p\frac{k}{n}s\right)
\right\vert =o_{\mathbb{P}}\left(  s^{\left(  \eta-1\right)  /2}\right)  .
\]
Likewise, with $d=q\dfrac{k}{n}s-q\dfrac{k}{n}\left(  \ell_{n}\left(
s\right)  \right)  ^{-1/\gamma}$ and $x=1-q\dfrac{k}{n}s,$ we get
\[
\sqrt{\frac{n}{k}}\left\vert B_{n}\left(  1-q\frac{k}{n}\left(  \ell
_{n}\left(  s\right)  \right)  ^{-1/\gamma}\right)  -B_{n}\left(  1-q\frac
{k}{n}s\right)  \right\vert =o_{\mathbb{P}}\left(  s^{\left(  \eta-1\right)
/2}\right)  ,
\]
uniformly over $k^{-1}\leq s\leq1,$ which finishes the proof of the Lemma.
\end{proof}

\begin{lemma}
\label{Lemma-12}For every $0\leq\tau<1/8,$ we have $R_{ni}\left(  x\right)
=o_{\mathbb{P}}\left(  x^{-\tau/\gamma}\right)  ,$ $i=10,11.$
\end{lemma}

\begin{proof}
Let us first fix $0\leq\tau<1/8$ and recall that
\[
R_{n10}\left(  x\right)  =-p^{-1}\int_{x}^{Z_{n:n}/Z_{n-k:n}}u^{1/\gamma
}\left\{  \beta_{n}^{\ast}\left(  u\right)  -\sqrt{\frac{n}{k}}\mathbf{B}%
_{n}^{\ast}\left(  uZ_{n-k:n}\right)  \right\}  d\left\{  \frac{n}{k}%
\overline{H}_{n}\left(  Z_{n-k:n}u\right)  \right\}  ,
\]
which may be rewritten, by letting $\widetilde{\beta}_{n}:=\beta_{n}^{\ast
}-\beta_{n}$ and $\widetilde{\mathbf{B}}_{n}:=\mathbf{B}_{n}^{\ast}%
-\mathbf{B}_{n},$ as the sum of
\[
K_{1n}\left(  x\right)  :=-p^{-1}\int_{x}^{Z_{n:n}/Z_{n-k:n}}u^{1/\gamma
}\left\{  \beta_{n}\left(  u\right)  -\sqrt{\frac{n}{k}}\mathbf{B}_{n}\left(
uZ_{n-k:n}\right)  \right\}  d\left\{  \frac{n}{k}\overline{H}_{n}\left(
Z_{n-k:n}u\right)  \right\}  ,
\]
and%
\[
K_{2n}\left(  x\right)  :=-p^{-1}\int_{x}^{Z_{n:n}/Z_{n-k:n}}u^{1/\gamma
}\left\{  \widetilde{\beta}_{n}\left(  u\right)  -\sqrt{\frac{n}{k}}%
\widetilde{\mathbf{B}}_{n}\left(  uZ_{n-k:n}\right)  \right\}  d\left\{
\frac{n}{k}\overline{H}_{n}\left(  Z_{n-k:n}u\right)  \right\}  .
\]
Next we show that, $K_{in}\left(  x\right)  =o_{\mathbb{P}}\left(
x^{-\tau/\gamma}\right)  ,$ $i=1,2.$ Indeed, let $0<\eta<1/8$ and
$0<\epsilon<1$ be fixed and write
\begin{align*}
K_{1n}\left(  x\right)   &  =-p^{-1}\int_{x}^{Z_{n:n}/Z_{n-k:n}}u^{1/\gamma
}\frac{\beta_{n}\left(  u\right)  -\sqrt{\dfrac{n}{k}}\mathbf{B}_{n}\left(
uZ_{n-k:n}\right)  }{\left[  \overline{H}^{\left(  1\right)  }\left(
uZ_{n-k:n}\right)  \right]  ^{1/2-\eta}}\\
&  \ \ \ \ \ \ \ \ \ \ \ \ \ \ \ \ \ \ \ \ \ \times\left[  \overline
{H}^{\left(  1\right)  }\left(  uZ_{n-k:n}\right)  \right]  ^{1/2-\eta
}d\left\{  \frac{n}{k}\overline{H}_{n}\left(  Z_{n-k:n}u\right)  \right\}  .
\end{align*}
Since is $\overline{H}^{\left(  1\right)  }$ is a decreasing function, then
over the interval $x<u<Z_{n:n}/Z_{n-k:n},$ we have%
\[
\overline{H}^{\left(  1\right)  }\left(  Z_{n:n}\right)  \leq\overline
{H}^{\left(  1\right)  }\left(  uZ_{n-k:n}\right)  \leq\overline{H}^{\left(
1\right)  }\left(  xZ_{n-k:n}\right)  \leq\overline{H}^{\left(  1\right)
}\left(  Z_{n-k:n}\right)  ,
\]
for any $x\geq1.$ Then with a large probability $\left(  1-\epsilon\right)
p\frac{1}{n}\leq\overline{H}^{\left(  1\right)  }\left(  uZ_{n-k:n}\right)
\leq\left(  1+\epsilon\right)  p\frac{k}{n},$ hence, we may apply Gaussian
approximations $\left(  \ref{approx2}\right)  $ to get%
\[
K_{1n}\left(  x\right)  =\sqrt{\frac{n}{k}}O_{\mathbb{P}}\left(  n^{-\eta
}\right)  \int_{x}^{\infty}u^{1/\gamma}\left[  \overline{H}^{\left(  1\right)
}\left(  uZ_{n-k:n}\right)  \right]  ^{1/2-\eta}d\left\{  \frac{n}{k}%
\overline{H}_{n}\left(  Z_{n-k:n}u\right)  \right\}  .
\]
From their definitions above, $\widetilde{\beta}_{n}\left(  u\right)
=-\alpha_{n}\left(  1-\overline{H}^{\left(  0\right)  }\left(  v\right)
\right)  $ and $\widetilde{\mathbf{B}}_{n}\left(  u\right)  =-B_{n}\left(
1-\overline{H}^{\left(  0\right)  }\left(  v\right)  \right)  .$ Then by
applying Gaussian approximations $\left(  \ref{approx}\right)  $ we have%
\[
K_{2n}\left(  x\right)  =\sqrt{\frac{n}{k}}O_{\mathbb{P}}\left(  n^{-\eta
}\right)  \int_{x}^{\infty}u^{1/\gamma}\left[  \overline{H}^{\left(  0\right)
}\left(  uZ_{n-k:n}\right)  \right]  ^{1/2-\eta}d\left\{  \frac{n}{k}%
\overline{H}_{n}\left(  Z_{n-k:n}u\right)  \right\}  .
\]
By using the routine manipulations as used in Lemma \ref{Lemma-9}, we end up
with
\[
K_{in}\left(  x\right)  =O_{\mathbb{P}}\left(  k^{-\eta}\right)  x^{\left(
2\eta-1/2\right)  /\gamma+\epsilon},\text{ }i=1,2.
\]
Since $k\rightarrow\infty$ and $2\eta-1/2+\tau<0,$ then $K_{in}\left(
x\right)  =o_{\mathbb{P}}\left(  x^{-\tau/\gamma}\right)  ,$ $i=1,2.$ Let us
now consider $R_{n11}\left(  x\right)  $ which may be rewritten into the sum
of
\[
\mathcal{A}_{1n}\left(  x\right)  :=-\frac{1}{p}\int_{x}^{\infty}u^{1/\gamma
}\sqrt{\frac{n}{k}}\left\{  \mathbf{B}_{n}\left(  uZ_{n-k:n}\right)
-B_{n}\left(  p\frac{k}{n}u^{-1/\gamma}\right)  \right\}  d\left\{  \frac
{n}{k}\overline{H}_{n}\left(  Z_{n-k:n}u\right)  \right\}  ,
\]
and%
\[
\mathcal{A}_{2n}\left(  x\right)  :=\frac{1}{p}\int_{x}^{\infty}u^{1/\gamma
}\sqrt{\frac{n}{k}}\left\{  \widetilde{\mathbf{B}}_{n}\left(  uZ_{n-k:n}%
\right)  -\widetilde{B}_{n}\left(  \frac{k}{n}u^{-1/\gamma}\right)  \right\}
d\left\{  \frac{n}{k}\overline{H}_{n}\left(  Z_{n-k:n}u\right)  \right\}  .
\]
Recall that from $\left(  \ref{Bn}\right)  $ and write%
\[
B_{n}\left(  p\frac{k}{n}u^{-1/\gamma}\right)  -\mathbf{B}_{n}\left(
uZ_{n-k:n}\right)  =B_{n}\left(  p\frac{k}{n}u^{-1/\gamma}\right)
-B_{n}\left(  \overline{H}^{\left(  1\right)  }\left(  uZ_{n-k:n}\right)
\right)  .
\]
Once again, by applying this inequality $\left(  \ref{inequa}\right)  $ with
$x=p\dfrac{k}{n}u^{-1/\gamma},$ $y=\overline{H}^{\left(  1\right)  }\left(
uZ_{n-k:n}\right)  $ and $\rho=\left\vert \overline{H}^{\left(  1\right)
}\left(  uZ_{n-k:n}\right)  -p\frac{k}{n}u^{-1/\gamma}\right\vert ,$ we may
write
\[
B_{n}\left(  p\frac{k}{n}u^{-1/\gamma}\right)  -B_{n}\left(  \overline
{H}^{\left(  1\right)  }\left(  uZ_{n-k:n}\right)  \right)  =O_{\mathbb{P}%
}\left(  \sqrt{\left\vert \overline{H}^{\left(  1\right)  }\left(
uZ_{n-k:n}\right)  -p\frac{k}{n}u^{-1/\gamma}\right\vert }\right)  ,
\]
uniformly over $u\geq1.$ On the other hand, from inequality $\left(
\ref{three}\right)  ,$ uniformly on $u\geq1,$ we write $\left\vert
\overline{H}^{\left(  1\right)  }\left(  uZ_{n-k:n}\right)  -p\dfrac{k}%
{n}u^{-1/\gamma}\right\vert \leq\epsilon p\dfrac{k}{n}u^{-1/\gamma+\epsilon},$
therefore%
\[
\sqrt{\frac{n}{k}}\left(  B_{n}\left(  p\frac{k}{n}u^{-1/\gamma}\right)
-B_{n}\left(  \overline{H}^{\left(  1\right)  }\left(  uZ_{n-k:n}\right)
\right)  \right)  =o_{\mathbb{P}}\left(  u^{\left(  \epsilon-1/\gamma\right)
/2}\right)  .
\]
which implies that $\mathcal{A}_{in}\left(  x\right)  =o_{\mathbb{P}}\left(
1\right)  \int_{x}^{\infty}u^{\left(  \epsilon-1/\gamma\right)  /2}d\left\{
\frac{n}{k}\overline{H}_{n}\left(  Z_{n-k:n}u\right)  \right\}  ,$ $i=1,2.$ By
using routine manipulations of assertion $\left(  i\right)  $ of Lemma
\ref{lemma-5}, with the fact $\left(  1-\epsilon\gamma\right)  /2>\tau,$ as
$\epsilon\downarrow0,$ we get $\mathcal{A}_{in}\left(  x\right)
=o_{\mathbb{P}}\left(  x^{-\tau/\gamma}\right)  ,$ uniformly over $x\geq1,$ as sought.
\end{proof}

\begin{lemma}
\label{Lemma-13}For every $0\leq\tau<1/8,$ we have $R_{ni}\left(  x\right)
=o_{\mathbb{P}}\left(  x^{-\tau/\gamma}\right)  ,$ $i=12,13,14.$
\end{lemma}

\begin{proof}
Let fix $0\leq\tau<1/8$ and recall that%
\[
R_{n12}\left(  x\right)  =p^{-1}\int_{\frac{n}{k}\overline{H}_{n}\left(
Z_{n-k:n}x\right)  }^{\frac{n}{k}\overline{H}_{n}\left(  Z_{n:n}x\right)
}\left(  \left(  \ell_{n}\left(  s\right)  \right)  ^{1/\gamma}-s^{-1}\right)
\sqrt{\frac{n}{k}}\left\{  B_{n}^{\ast}\left(  \frac{k}{n}\left(  \ell
_{n}\left(  s\right)  \right)  ^{-1/\gamma}\right)  -B_{n}^{\ast}\left(
\frac{k}{n}s\right)  \right\}  ds,
\]
where $\ell_{n}\left(  s\right)  $ is as in $\left(  \ref{ln(s)}\right)  .$
Note that $\overline{H}_{n}\left(  Z_{n:n}x\right)  =n^{-1}$ for $x\geq1,$
therefore
\[
\left\vert R_{n12}\left(  x\right)  \right\vert \leq p^{-1}\int_{k^{-1}%
}^{\frac{n}{k}\overline{H}_{n}\left(  Z_{n-k:n}x\right)  }\left\vert \left(
\ell_{n}\left(  s\right)  \right)  ^{1/\gamma}-s^{-1}\right\vert \sqrt
{\frac{n}{k}}\left\vert B_{n}^{\ast}\left(  \frac{k}{n}\left(  \ell_{n}\left(
s\right)  \right)  ^{-1/\gamma}\right)  -B_{n}^{\ast}\left(  \frac{k}%
{n}s\right)  \right\vert ds.
\]
Let $1/8<\eta<1/4$ and small $\epsilon>0,$ then from Lemma \ref{Lemma-7} and
by using the mean value theorem, we deduce that $\left(  \ell_{n}\left(
s\right)  \right)  ^{1/\gamma}-s^{-1}=o_{\mathbb{P}}\left(  s^{-1-\eta
}\right)  ,$ uniformly over $k^{-1}\leq s<1.$ On the other hand, by using
Lemma \ref{Lemma-12},%
\[
R_{n12}\left(  x\right)  =O_{\mathbb{P}}\left(  1\right)  \int_{0}^{\frac
{n}{k}\overline{H}_{n}\left(  Z_{n-k:n}x\right)  }s^{-1-\eta}s^{-\left(
\eta-1\right)  /2}ds=O_{\mathbb{P}}\left(  1\right)  \left(  \frac{n}%
{k}\overline{H}_{n}\left(  Z_{n-k:n}x\right)  \right)  ^{-\frac{3}{2}%
\eta+\frac{1}{2}}.
\]
Note that $\overline{H}_{n}\left(  Z_{n-k:n}x\right)  =o_{\mathbb{P}}\left(
x^{-1/\gamma+\epsilon}\right)  ,$ uniformly over $x\geq1,$ it follows that
$R_{n12}\left(  x\right)  =o_{\mathbb{P}}\left(  1\right)  x^{\left(  \frac
{3}{2}\eta-\frac{1}{2}\right)  /\gamma+\epsilon^{\ast}},$ with $\epsilon
^{\ast}\downarrow0.$ We verified that for $0\leq\tau<1/8$ and $1/8<\eta<1/4,$
we have $\left(  3\eta-1\right)  /2+\tau<0,$ this means that $R_{n12}\left(
x\right)  =o_{\mathbb{P}}\left(  x^{-\tau/\gamma}\right)  .$ By using similar
arguments, we also end up with $R_{n13}\left(  x\right)  =o_{\mathbb{P}%
}\left(  x^{-\tau/\gamma}\right)  =$ $R_{n14}\left(  x\right)  $ as sought.
\end{proof}

\begin{lemma}
\label{Lemma-14}For every $0\leq\tau<1/8,$ we have $R_{ni}\left(  x\right)
=o_{\mathbb{P}}\left(  x^{-\tau/\gamma}\right)  ,$ $i=15,16.$
\end{lemma}

\begin{proof}
Let us fix $0\leq\tau<1/8$ and recall that%
\[
R_{n15}\left(  x\right)  =-p^{-1}\int_{x^{-1/\gamma}}^{\frac{n}{k}\overline
{H}_{n}\left(  Z_{n-k:n}x\right)  }s^{-1}\sqrt{\frac{n}{k}}B_{n}^{\ast}\left(
\frac{k}{n}s\right)  ds.
\]
It is clear that
\[
\left\vert R_{n15}\left(  x\right)  \right\vert \leq p^{-1}\int_{\min\left(
\frac{n}{k}\overline{H}_{n}\left(  Z_{n-k:n}x\right)  ,x^{-1/\gamma}\right)
}^{\max(\frac{n}{k}\overline{H}_{n}\left(  Z_{n-k:n}x\right)  ,x^{-1/\gamma}%
)}s^{-1}\sqrt{\frac{n}{k}}\left\vert B_{n}^{\ast}\left(  \frac{k}{n}s\right)
\right\vert ds.
\]
From \cite{H-R-98} (Propositions 2.2, 2.3, 3.1), see also \cite{H-R-93}
(Proposition 4.1), there exists a Brownian bridge $\left\{  \mathcal{B}%
\mathbf{(}t\mathbf{);}\text{ }t>0\right\}  $ such that
\[
\sup_{x\geq1}x^{\rho/\gamma}\left\vert \sqrt{k}\left\{  \dfrac{n}{k}%
\overline{H}_{n}\left(  Z_{n-k:n}x\right)  -x^{-1/\gamma}\right\}
-\mathcal{B}\left(  x^{-1/\gamma}\right)  -\sqrt{k}A\left(  h\right)
x^{-1/\gamma}\frac{x^{\nu/\gamma}-1}{\nu\gamma}\right\vert =o_{\mathbb{P}%
}\left(  1\right)  ,
\]
for every $4\tau<\rho<1/2.$\textbf{ }Note that $\sup_{x\geq1}x^{\rho/\gamma
}\left\vert \mathcal{B}\left(  x^{-1/\gamma}\right)  \right\vert =\sup_{0\leq
s\leq1}s^{-\rho}\left\vert \mathcal{B}\left(  s\right)  \right\vert $ which is
stochastically bounded (see for instance Lemma 3.2 in \cite{EHL-2006}). On the
other hand, we have $\sqrt{k}A\left(  h\right)  =O\left(  1\right)  ,$ it
follows that $\sqrt{k}A\left(  h\right)  x^{\rho/\gamma-1/\gamma}\left(
x^{\nu/\gamma}-1\right)  /\left(  \nu\gamma\right)  =O\left(  1\right)  ,$
uniformly over $x\geq1,$ therefore
\[
\sup\limits_{x\geq1}x^{\rho/\gamma}\left\vert \dfrac{n}{k}\overline{H}%
_{n}\left(  Z_{n-k:n}x\right)  -x^{-1/\gamma}\right\vert =O_{\mathbb{P}%
}\left(  k^{-1/2}\right)  =o_{\mathbb{P}}\left(  1\right)  .
\]
Let $0<\epsilon<1$ be sufficiently small and set
\[
\mathcal{D}_{n}\left(  \epsilon;\rho\right)  :=\left\{  \sup_{x\geq1}%
x^{\rho/\gamma}\left\vert \dfrac{n}{k}\overline{H}_{n}\left(  Z_{n-k:n}%
x\right)  -x^{-1/\gamma}\right\vert >\epsilon\right\}  .
\]
For a fixed $c>0,$ let $c_{x}:=cx^{-\tau/\gamma}$ and write
\[
\mathbb{P}\left(  \left\vert R_{n15}\left(  x\right)  \right\vert
>c_{x}\right)  <\mathbb{P}\left(  \left\vert R_{n15}\left(  x\right)
\right\vert >c_{x},\text{ }\mathcal{D}_{n}\left(  \epsilon;\rho\right)
\right)  +\mathbb{P}\left(  \mathcal{D}_{n}\left(  \epsilon;\rho\right)
\right)  .
\]
It is clear that
\[
\mathbb{P}\left(  \left\vert R_{n15}\left(  x\right)  \right\vert
>c_{x},\text{ }\mathcal{D}_{n}\left(  \epsilon;\rho\right)  \right)
\leq\mathbb{P}\left(  p^{-1}\int_{x^{-1/\gamma}-\epsilon x^{-\rho/\gamma}%
}^{x^{-1/\gamma}+\epsilon x^{-\rho/\gamma}}s^{-1}\sqrt{\frac{n}{k}}\left\vert
B_{n}^{\ast}\left(  \frac{k}{n}s\right)  \right\vert ds>c_{x}\right)  ,
\]
which, by using Markov inequality, is
\[
\leq p^{-1}c_{x}^{-2}\int_{x^{-1/\gamma}-\epsilon x^{-\rho/\gamma}%
}^{x^{-1/\gamma}+\epsilon x^{-\rho/\gamma}}s^{-1}\sqrt{\frac{n}{k}}%
\mathbf{E}\left\vert B_{n}^{\ast}\left(  \frac{k}{n}s\right)  \right\vert ds.
\]
It is easy to verify that $\mathbf{E}\left\vert B_{n}^{\ast}\left(  s\right)
\right\vert \leq s^{1/2},$ it follows that the previous quantity is
\[
\leq p^{-1}c_{x}^{-2}\int_{x^{-1/\gamma}-\epsilon x^{-\rho/\gamma}%
}^{x^{-1/\gamma}+\epsilon x^{-\rho/\gamma}}s^{-1/2}ds=2p^{-1}c^{-2}%
x^{2\tau/\gamma}\left\{  \left(  x^{-1/\gamma}+\epsilon x^{-\rho/\gamma
}\right)  ^{1/2}-\left(  x^{-1/\gamma}-\epsilon x^{-\rho/\gamma}\right)
^{1/2}\right\}  .
\]
It is ready to check, in view for the mean value theorem, that the expression
between two brackets equals $o\left(  x^{-\rho/\left(  2\gamma\right)
}\right)  .$ This implies that
\[
p^{-1}c_{x}^{-2}\int_{x^{-1/\gamma}-\epsilon}^{x^{-1/\gamma}+\epsilon}%
s^{-1/2}ds=o\left(  x^{\left(  2\tau-\rho/2\right)  /\gamma}\right)  =o\left(
1\right)  ,
\]
because $4\tau<\rho<1/2,$ therefore $\mathbb{P}\left(  \left\vert
R_{n14}\left(  x\right)  \right\vert >c_{x},\text{ }\mathcal{D}_{n}\left(
\epsilon;\rho\right)  \right)  =o\left(  1\right)  .$ Since $\mathbb{P}\left(
\mathcal{D}_{n}\left(  \epsilon;\rho\right)  \right)  =o\left(  1\right)  ,$
hence $R_{n15}\left(  x\right)  =o_{\mathbb{P}}\left(  x^{-\tau/\gamma
}\right)  ,$ uniformly over $x\geq1.$ By using similar arguments, we also show
that $R_{n16}\left(  x\right)  =o_{\mathbb{P}}\left(  x^{-\tau/\gamma}\right)
,$ that we omit further details.
\end{proof}

\end{document}